\documentclass[a4paper,11pt]{article}
\usepackage[right=3cm,left=3cm,top=3cm,bottom=3cm]{geometry}

\usepackage{tikz}
\usepackage[english,french]{babel}
\usepackage[utf8]{inputenc}
\usepackage[T1]{fontenc}

\usepackage{amssymb,amsmath,amscd,amsfonts,amsthm,latexsym,bbm,mathtools}
\usepackage{xcolor,graphpap,fancybox,framed,enumitem}
 
\usepackage{array}
\newcolumntype{M}[1]{>{\centering\arraybackslash}m{#1}}

\usepackage{float}
\usepackage{enumitem}
\usepackage{ulem}
\usepackage{graphicx} 
\usepackage[colorlinks=true,citecolor=red,linkcolor=blue,pdfpagetransition=Blinds]{hyperref}

\usepackage{appendix}

\newtheorem{theorem}{Theorem}[section]
\newtheorem{prop}[theorem]{Proposition}
\newtheorem{coro}[theorem]{Corollary}
\newtheorem{claim}[theorem]{Conjecture}
\newtheorem{lemma}[theorem]{Lemma}
\newtheorem{defi}[theorem]{Definition}

\newcommand{\RR}{\mathbb{R}}
\newcommand{\NN}{\mathbb{N}}
\newcommand{\CC}{\mathbb{C}}
\newcommand{\QQ}{\mathbb{Q}}

\newcommand{\ZZ}{\mathbb{Z}}

\newcommand{\Cc}{\mathcal{C}}
\newcommand{\Rr}{\Gc}

\newcommand{\Ec}{\mathcal{E}}
\newcommand{\Gc}{\mathcal{G}}

\newcommand{\Nc}{\mathcal{N}}

\newcommand{\Rc}{\Gc}

\newcommand{\Vc}{\mathcal{V}}

\newcommand{\grad}{\nabla}

\newcommand{\dd}{\,{\text{\rm d}}}

\newcommand{\pc}{ \usefont{T1}{cmtl}{m}{n} \selectfont}

\def\beq{\begin{eqnarray}}
\def\eeq{\end{eqnarray}}
\def\beqs{\begin{eqnarray*}}
\def\eeqs{\end{eqnarray*}}

\newdimen\texpscorrection
\newdimen\figcenter
\def\figurewithtex #1 #2 #3 #4 #5\cr{\null
  {\goodbreak\figcenter=\hsize\relax
  \advance\figcenter by -#4truecm
  \divide\figcenter by 2
  \begin{figure}[hbt]
  \vskip #3truecm\noindent\hskip\figcenter
  \includegraphics{#1}{\hskip\texpscorrection\input #2 }
  \vskip 0.8truecm{\baselineskip=0.8\baselineskip
  \noindent \vbox{\noindent {\footnotesize #5}}\par}
  \end{figure}}}
\def\point#1 #2 #3 {\rlap{\kern #1 truecm
\raise #2 truecm \hbox{#3}}}

\newcounter{compte}

\begin{document}
\selectlanguage{english}

\title{The Graph Geometric Control Condition}

\author{Ka\"is \textsc{Ammari}\footnote{LR Analysis and Control of PDEs, LR 22ES03, Department of Mathematics, Faculty of Sciences of Monastir, University of Monastir, Tunisia;  email: {\pc kais.ammari@fsm.rnu.tn}} ,  
Alessandro \textsc{Duca}\footnote{Universit\'e de Lorraine, CNRS, INRIA, IECL, F-54000 Nancy, France; email: {\pc alessandro.duca@inria.fr}} ,
Romain \textsc{Joly}\footnote{Universit\'e Grenoble Alpes, CNRS, 
Institut Fourier, F-38000 Grenoble, France; email: {\pc 
romain.joly@univ-grenoble-alpes.fr}} {~\&~}
Kévin \textsc{Le Balc'h}
\footnote{INRIA,
Sorbonne Universit\'e, F-75005 Paris, France; email: {\pc kevin.le-balc-h@inria.fr}}}
\date{}

\maketitle

\begin{abstract}
In this paper, we introduce a novel concept called the Graph Geometric Control Condition (GGCC).
It turns out to be a simple, geometric rewriting of many of the frameworks in which the controllability of PDEs on graphs has been studied. We prove that (GGCC) is a necessary and sufficient condition for the exact controllability of the wave equation on metric graphs with internal controls and Dirichlet boundary conditions. 
We then investigate the internal exact controllability of the wave equation with mixed boundary conditions and the one of the Schrödinger equation, as well as the internal null-controllability of the heat equation. We show that (GGCC) provides a sufficient condition for the controllability of these equations and we provide explicit examples proving that (GGCC) is not necessary in these cases.
\end{abstract}


\bigskip

\selectlanguage{french}

\begin{abstract}
Dans cet article, nous introduisons un nouveau concept appel\'e la Condition G\'eom\'etrique de Contr\^ole sur les Graphes (GGCC). Il s'av\`ere que cette condition constitue une reformulation g\'eom\'etrique simple de nombreux cadres dans lesquels la contr\^olabilit\'e des \'equations aux d\'eriv\'ees partielles sur les graphes a \'et\'e étudiée. Nous démontrons que la condition (GGCC) est à la fois nécessaire et suffisante pour la contrôlabilité exacte de l'équation des ondes sur les graphes métriques avec des contrôles internes et des conditions aux limites de Dirichlet. Nous examinons ensuite la contrôlabilité exacte interne de l'équation des ondes avec des conditions aux limites mixtes, ainsi que celle de l'équation de Schrödinger, et nous étudions également la contrôlabilité à zéro interne de l'équation de la chaleur. Nous montrons que (GGCC) fournit une condition suffisante pour la contrôlabilité de ces équations et nous présentons des exemples explicites prouvant que (GGCC) n'est pas une condition nécessaire dans ces cas.
\end{abstract}
\selectlanguage{english}

\newpage

\setcounter{tocdepth}{2}
\tableofcontents
\setcounter{tocdepth}{2}

\newpage

\section{Introduction and main results}

Network models have been extensively studied in the literature over the last few decades for modeling phenomena in science, engineering and social sciences. Notable applications include the dynamics of free electrons in organic molecules (see the seminal work \cite{Pau36}), superconductivity in granular and artificial materials (see \cite{Ale83}) and acoustic and electromagnetic wave-guide networks (see \cite{FJK87, ML71}). Control problems on metric graphs, particularly those involving tree-like structures, have been a focal point of extensive research due to their wide-ranging applications in science and engineering. These applications span various fields, including network theory, vibration analysis and structural dynamics (see \cite{AS,AZ21, DZ06, LLS94, Nicaise, Schmidt} and the references therein).

\medskip

Although network-type structures have garnered significant interest in control theory, there is currently no explicit characterization in the literature analogous to the well-known Geometric Control Condition (GCC) of Bardos, Lebeau, Rauch and Taylor 
\cite{BLR_JEDP,BLR,RT} for Euclidean domains. The condition (GCC) provides necessary and sufficient condition for the controllability of the wave equation: it encodes the geometrical properties of the domain $\Omega$, and the controlled subdomain $\omega \subset \Omega$ to ensure controllability. We recall that $\Omega$ and $\omega$ satisfy the Geometric Control Condition (GCC) if there exists $L > 0$ such that any generalized geodesic of $\Omega$ of length greater than $L$ meets the set $\omega$. The Geometric Control Condition for metric graphs has only been hinted at in a few existing works, see for instance \cite{AEL23, AEZ23}, but no complete study has been provided yet. The absence of such a characterization has left a significant gap in the understanding of controllability for wave equations in these non-Euclidean settings.

\medskip

The aim of this work is to address this issue by introducing a novel Graph Geometric Control Condition (GGCC) and comparing it with the hypotheses commonly used in the literature to control PDEs on metric graphs. The condition (GGCC) provides a rigorous and precise framework for characterizing the exact controllability of the wave equation on networks with Dirichlet boundary conditions in the presence of internal controls. Moreover, it serves as a sufficient—though not necessary—condition for establishing the internal exact controllability of the wave equation with mixed boundary conditions, the Schrödinger equation and the internal null-controllability of the heat equation.

\subsection{The Graph Geometric Control Condition}

Throughout the rest of this article, a graph $\Gc$ is defined as a structure composed by some edges (segments) connecting some vertices (points); see Figure \hyperref[path]{1}. 
In this paper, we focus on studying partial differential equations (PDEs) defined on the edges of a graph. In this context, the primary structure of interest lies in the edges, while the vertices play a secondary role, being significant only through the connections they establish rather than as standalone elements.
We refer to vertices that are attached to the structure by only one edge as exterior vertices. The remaining ones are referred to as interior vertices. We consider very general graphs: loops (an edge connecting a vertex to itself) and parallel edges (multiple edges with the same endpoints) are allowed. We impose only the finiteness of the graph: there are a finite number of edges, all having a finite length. 
 Notice that relaxing the finiteness assumption would, of course, alter the analysis of the problem. However, most of the ideas presented in this paper could still be adapted to such cases. 
 
\begin{figure}[ht]
\begin{center}
\vspace{3mm}

\resizebox{11cm}{!}{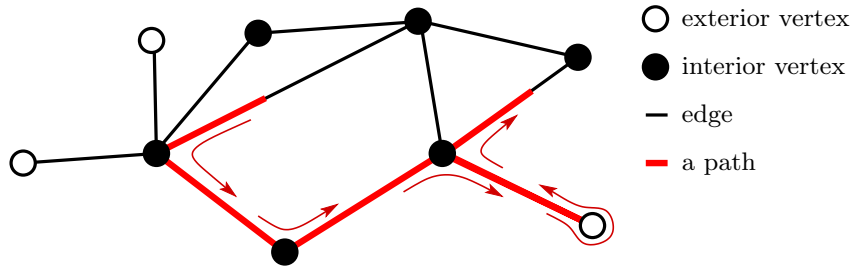}
\end{center}
\caption{\it A graph with $11$ edges, $3$ exterior vertices and $6$ interior vertices. A path is represented on the graph, travelling along the edges at constant speed. Notice the reversal of the path at the exterior vertex, a move that is not allowed at the interior vertices.}\label{path}
\end{figure}

\medskip

We aim to define the Graph Geometric Control Condition adopted in this work and present some equivalent formulations. To this purpose, we need to provide few abstract definitions (their mathematical description is presented in Section \ref{sec_GGCC}). A {\bf path} on the graph $\Gc$ is defined as follows (see Figure \hyperref[path]{1}).   

\begin{itemize}
    \item We start at any point of any edge of the graph and move in one of the two possible directions until we reach a vertex. 
    \item If we reach an exterior vertex, we turn back and continue in the other direction. Otherwise, we move along any other edge starting from the vertex.  
    
    \item We continue the route through the graph by repeating the previous point and stop at some point. 
    
    \end{itemize}
The total {\bf length} of a path is the distance covered in $\Gc$. In all the article, $\omega$ is an open subset of the graph $\Gc$, that is an  open subset of the disjoint union of the edges. A path is said to {\bf meet} $\omega$ if it passes through a point belonging to $\omega$.  

\medskip

We are finally ready to state our Graph Geometric Control Condition.

\begin{defi}[GGCC]\label{defi_GGCC}
Consider a graph $\Gc$ and an open set $\omega\subset\Gc$. We say that $\Gc$ and $\omega$ satisfy the Graph Geometric Control Condition (GGCC) when there exists a length $L>0$ such that any path longer than $L$ meets $\omega$. 
\end{defi}

The setting of this condition is written as an echo of the famous Geometric Control Condition of Bardos, Lebeau, Rauch and Taylor, recalled above.
The main purpose of this article is to show that (GGCC) is, like its inspiration (GCC), an important property for the control and observation of PDE on graphs. 

\medskip

The condition (GGCC) is somehow a natural condition but, to the best of our knowledge, this is the first time that it is introduced in this geometrical way. 
For example, from the perspective of boundary control, it is not difficult to see that (GGCC) means that $\Gc$ is a tree and that the control is efficient at all the exterior vertices except maybe one. This geometric condition appear in many works. For example, Schmidt proves in \cite{Schmidt} that such boundary control is sufficient for the exact controllability of the wave equation on tree graphs, see \cite[Theorem 2.7 in Section 2.4]{DZ06}. The necessity of such a configuration for the exact controllability of the wave equation on tree graphs was latter obtained in \cite[Theorem 2.8 in Section 2.4]{DZ06}, see also \cite[Chapter VII]{AI95}. These results were revisited by Avdonin and Zhao in \cite{AZ22}.  Actually, to study the optimal time of control, these last authors introduce a geometric description that provides a criterion very close, and actually equivalent, to (GGCC). In the present paper, we choose to call it {\bf watershed condition}, see Definition \ref{defi_watershed} in the next section. 

\medskip

In the case of internal control, one has more freedom concerning the geometry. In \cite{AB}, Apraiz and B\'arcena-Petisco prove the internal null-controllability of a parabolic equation on a graph, assuming some technical hypotheses on $\Gc$ and $\omega$ which we denote here {\bf ABP Condition}, see the precise definition in Definition \ref{defi_AB} in Section \ref{sec_GGCC}. As we will show here, the ABP Condition is equivalent to (GGCC), which provides a much more geometric description. 

\medskip

As a first contribution, we show in this article that all the previous conditions introduced to study the controllability of PDE on graphs are actually equivalent to the geometric criterion (GGCC).

\begin{theorem}\label{prop_GGCC}
Let $\Gc$ be a connected graph and let $\omega$ be an open subset of $\Gc$. The following properties are equivalent:
\begin{enumerate}
\item $\Gc$ and $\omega$ satisfy (GGCC).
\item Every cycle of the graph and every path connecting two exterior vertices contain a part of $\omega$.
\item The set of uncontrolled edges is a forest and each one of its trees contains at most one of the exterior vertex of the original graph $\Gc$. 
\item Any periodic path of $\Gc$ meets the set $\omega$.
\item $\Gc$ and $\omega$ satisfy the ABP Condition of Definition \ref{defi_AB}.
\item $\Gc$ and $\omega$ satisfy the watershed Condition of Definition \ref{defi_watershed}.
\end{enumerate}
\end{theorem}

 Theorem \ref{prop_GGCC} is proved in Section \ref{sec_GGCC} and it provides some equivalent definitions of our Graph Geometric Control Condition.

\begin{figure}[p]
\begin{center}
\resizebox{0.9\textwidth}{!}{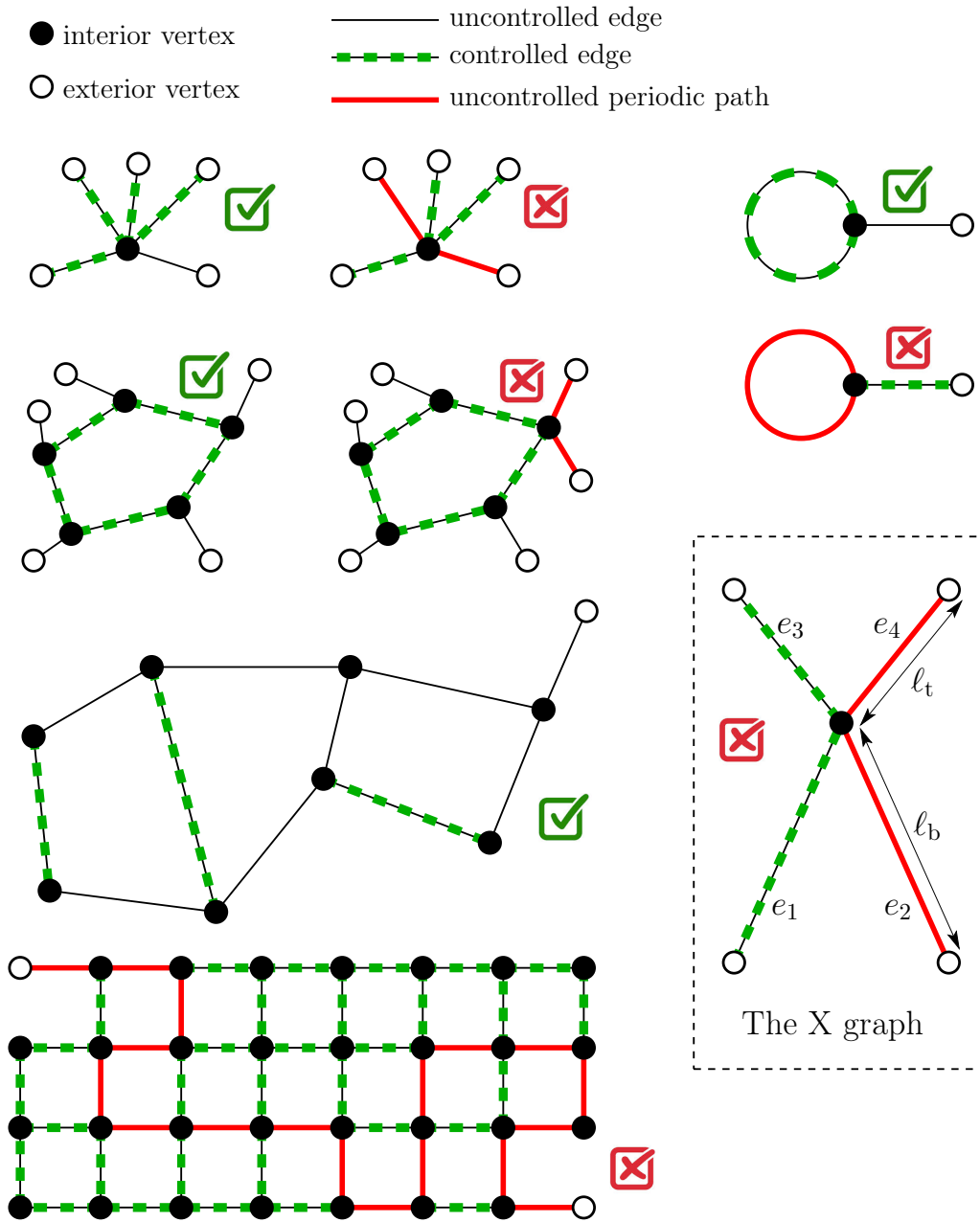}
\end{center}
\caption{\it Some examples of graphs satisfying or not the property (GGCC). In this article, we study in details the case of the Schr\"odinger equation on the X graph, at the bottom right. This graph is symmetric, with top edges of  length $\ell_{\text{t}}$ and bottom edges of length $\ell_{\text{b}}$, but only the left part of it is controlled. In this case, (GGCC) is clearly not satisfied.}\label{fig_ex}
\end{figure}

\subsection{Exact controllability of the wave equation}

We now present one of the main results of this work: we show that (GGCC) precisely encodes the geometrical assumptions required for controllability of the wave equation. In detail, we prove that (GGCC) is both a necessary and sufficient condition for the internal exact controllability of the wave equation with Dirichlet boundary conditions. In what follows, we present our theory but omit to define some standard elements to keep the presentation as simple as possible. We refer to Section \ref{preliminaries} for the formal definitions of all considered mathematical objects.

\medskip

We denote by $\Delta_{\Gc}$ the Laplacian defined on the graph $\Gc$ equipped with either Dirichlet or Neumann boundary conditions on the exterior vertices and with the standard Kirchhoff's transmission conditions on the interior ones. We set $H^1_\Delta(\Gc)$ the set of functions in $H^1(\Gc)$ satisfying the Dirichlet boundary conditions at the suitable exterior vertices. We consider the following controlled wave equation:
\begin{equation}\label{intro_wave}
\left\{
\begin{array}{ll} 
\partial_{t}^2 u = \Delta_\Gc u+ q(x) u + h 1_{\omega} ,\ \ \ \ & t>0,\\
(u,\partial_t u)(t=0)=(u_0,u_1) \in H^1_\Delta(\Gc)\times L^2(\Gc), & 
\end{array}\right. 
\end{equation}
with $q\in L^\infty (\Gc)$, and where, at time $t \geq 0$, $(u, \partial_t u)(t,\cdot) : \Gc \to \mathbb{R}^2$ denotes the state and $h(t,\cdot) : \Gc\to \mathbb{R}$ is the control. 
The definition of the exact controllability of the wave equation \eqref{intro_wave} is as follows.

\begin{defi}
\label{defi_exactwave}
Let $T>0$. The wave equation \eqref{intro_wave} is exactly controllable at time $T>0$ when, for every $(u_0, u_1) \in H^1_\Delta(\Gc)\times L^2(\Gc)$ and for every $(v_0,v_1) \in H^1_\Delta(\Gc)\times L^2(\Gc)$, there exists $h \in L^2((0,T)\times \omega)$ such that the solution $u  \in \Cc^1([0,T],L^2(\Gc))\cap \Cc^0([0,T],H^1_\Delta(\Gc)) $ of \eqref{intro_wave} with initial state $(u,\partial_t u)(t=0)=(u_0,u_1)$ satisfies $$(u,\partial_t u)(t=T)=(v_0,v_1).$$
\end{defi}
Our second main contribution of this article is as follows.

\begin{theorem}[Control of the wave equation with Dirichlet b.c.]
\label{th_exact_wave_intro_dir}
Assume that $\Delta_\Gc$ is the Laplacian operator with Dirichlet boundary conditions at all the exterior vertices and assume that $0$ is not in the spectrum of $\Delta_\Gc +q$. There is a time $T>0$ such that the wave equation \eqref{intro_wave} is exactly controllable at time $T$ if and only if (GGCC) is satisfied. Moreover, when (GGCC) holds, the exact controllability is ensured for any time $T>L$ where $L$ is the length appearing in (GGCC).
\end{theorem}
\begin{theorem}[Control of the wave equation with general b.c.]
\label{th_exact_wave_intro_neu}
Assume that $\Delta_\Gc$ is the Laplacian operator with either Neumann or Dirichlet boundary conditions at the exterior vertices and assume that $0$ is not in the spectrum of $\Delta_\Gc +q$.
If (GGCC) holds for the length $L$, then, for any time $T>L$, the wave equation \eqref{intro_wave} is exactly controllable at time $T$.
\end{theorem}
Theorems \ref{th_exact_wave_intro_dir} and \ref{th_exact_wave_intro_neu} show that (GGCC) characterizes the main geometrical properties required for the exact controllability of \eqref{intro_wave}. 
These results are proved in Section \ref{sec_control_wave} by rewriting the control problem in terms of the observability for the dual system, see Theorem \ref{th:obsexactwave}. If (GGCC) is not satisfied, we can construct quasimodes that are mainly concentrated in the unobserved part of the graph and vanish at the exterior vertices. This strategy is common and very similar proofs can be found in \cite{DZ06} for example. It shows the necessity of (GGCC) when Dirichlet boundary conditions are assumed at the exterior vertices of the uncontrolled edges, whatever are the boundary conditions at the ends of the controlled edges, even Neumann ones. Notice however that (GGCC) may not be necessary when Neumann boundary conditions are imposed at uncontrolled edges, see Theorem  \ref{th_neumann} below.

\medskip

To prove that (GGCC) is sufficient in any cases, it is possible to use explicit computations based on d'Alembert's formula, see \cite{AZ22,DZ06}. However, we choose here another strategy:  
\begin{itemize}
    \item Fix a classical method to extend the observation on intervals. Here, we use the method of multipliers, which is more general than d'Alembert's formula.  
    \item Use the previous tool to propagate the observation on edges up to the interior vertices.  
    \item When the PDE is controlled on all but one edge connected to an interior vertex, use Kirchhoff's transmission condition to propagate the known observations through the vertex to the remaining edge.  
    \item Use Description 3 of Theorem \ref{prop_GGCC} to show that (GGCC) ensures that the above steps are sufficient to extend the observation everywhere.  
\end{itemize}
This global scheme of proof is robust in the sense that it can be applied to other PDEs, see \cite{AB} for a similar strategy applied to the heat equation. 

\medskip 

Let us recall that the observation is also connected to the exponential stability of the damped wave equation
\begin{equation}\label{intro_wave_damp}
\left\{
\begin{array}{ll} 
\partial_{t}^2 u  +a(x) \partial_t u = \Delta_\Gc u+ q(x) u,\ \ \ \ & t>0,\\
(u,\partial_t u)(t=0)=(u_0,u_1) \in H^1_0(\Gc)\times L^2(\Gc), & \\
\end{array}\right. 
\end{equation}  
where $a \geq 0$ belongs to $L^\infty (\Gc)$ and $\omega \subset \mathrm{supp}(a)$. 
By leveraging the characterization of exponential stability in terms of the exact observability of the associated undamped problem (where $a=0$), we gain valuable insights into the dynamics of the system. This result, initially established by Haraux in \cite[Proposition 1]{Haraux} and generalized by \cite{AT}, provides a direct link between the ability to observe the full state of the system over time and the exponential decay of its energy. Thus, the following corollary is a direct consequence of Theorem \ref{th_exact_wave_intro_dir} and represents a graph version of a result established by Lebeau in \cite{Lebeau} where the damped wave equation is addressed within the context of an open subset of $\mathbb{R}^n$.
\begin{coro} \label{stab-exp_intro} 
Assume that $\Delta_\Gc$ is the Laplacian operator with Dirichlet boundary conditions at all the exterior vertices and assume that $\Delta_\Gc +q$ is a negative operator. The exponential stability of \eqref{intro_wave_damp} holds if and only if $\Gc$ and $\omega$ satisfy the (GGCC). 
\end{coro}
The assumption that $\Delta_\Gc +q$ is a negative operator is important to ensure that the energy related to \eqref{intro_wave_damp} is positive. In the same way, we need $\Delta_\Gc +q$ to be invertible to obtain the observation inequality and thus the control. For example, if we set $q=0$ and if there is no exterior vertex, then the constant solutions $u\equiv c\in\RR$ are neither observed nor damped and Corollary \ref{stab-exp_intro} obviously fails. In such cases, we should be able to generalize our results by restricting our problems in a suitable subspace. For instance, if there is no exterior vertex and $q=0$, Corollary \ref{stab-exp_intro} holds in the space orthogonal to the constant functions, that is the space of zero-average functions.

\medskip

We would also like to emphasize that the time stated in Theorems \ref{th_exact_wave_intro_dir} and \ref{th_exact_wave_intro_neu} is not optimal for the exact controllability of \eqref{intro_wave}. Actually, we provide in appendix an example where the optimal time is explicit, see Proposition \ref{prop_optimal}. The interested reader can find in Section \ref{optimal_time} a discussion concerning the optimal time for observation and control of the wave equation.

\medskip

We finally emphasize that the presence of Neumann boundary conditions strongly affects the result, which might be surprising given the classical results on open domains of $\RR^d$. The condition (GGCC) is still sufficient but is not necessary as shown by the following example. 
We consider the X graph of Figure \hyperref[fig_ex]{2}: a star-graph with four edges, here assumed to have the same length $\ell_{\text{b}}=\ell_{\text{t}}=1$, two (say $e_1$ and $e_2$) have external vertices endowed with Neumann boundary condition and two ($e_3$ and $e_4$) have external vertices endowed with Dirichlet boundary condition. The set $\omega$ is the union of the left edges $e_1\cup e_3$. Obviously, this geometry does not satisfy (GGCC), since there exists a periodic path living in $e_2\cup e_4$, see Figure \hyperref[fig_ex]{2}. 
\begin{theorem}\label{th_neumann}
Consider the above geometry for the X graph and assume that $q=0$. For any time $T>4$, the wave equation \eqref{intro_wave} is exactly controllable at time $T$.
\end{theorem}
The proof of this result in mainly based on the result of pointwise observation of the wave equation on an interval with Dirichlet boundary condition on the left side and Neumann condition on the right side from \cite{AHT} and is provided in Section \ref{sec:neumann}.

\subsection{Exact controllability of the Schr\"odinger equation}

Keeping in mind the results obtained for the wave equation, we now discuss the exact controllability issue for the Schr\"odinger equation. As we will see, (GGCC) is a sufficient but not necessary condition for the control of the Schrödinger equation.
More precisely, using the previous notations (see also Section \ref{preliminaries}), we consider the Schr\"odinger equation in a graph $\Gc$
\begin{equation}\label{eq_schro_control_intro}
\left\{
\begin{array}{ll} i \partial_{t} u = - \Delta_\Gc u-q(x) u + h 1_{\omega} ,& \ \ \ \ t>0,\\
u(t=0)=u_0\in L^2(\Gc), & 
\end{array}\right. 
\end{equation}
with $q\in L^\infty (\Gc)$ and $\Delta_{\Gc}$ the Laplacian defined on the graph $\Gc$, with either Dirichlet or Neumann conditions at the exterior vertices. In \eqref{eq_schro_control_intro}, at time $t \geq 0$, $u(t,\cdot) : \Gc \to \mathbb{C}$ denotes the state and $h(t,\cdot) : \Gc\to \mathbb{C}$ is the control.
The definition of the exact controllability of the Schr\"odinger equation \eqref{eq_schro_control_intro} is as follows.

\begin{defi} \label{defcontsch}
Let $T>0$. The Schrödinger equation \eqref{eq_schro_control_intro} is exactly controllable at time $T>0$ if for every $u_0 \in L^2(\Gc;\mathbb C)$ and for every $u_1 \in L^2(\Gc;\mathbb C)$, there exists $h \in L^2((0,T)\times \omega;\mathbb C)$ such that the solution $u \in \Cc([0,T];  L^2(\Gc))$ of \eqref{eq_schro_control_intro} with initial state $u(t=0)=u_0$  satisfies $$u(t=T) = u_1.$$
When the Schrödinger equation \eqref{eq_schro_control_intro} is exactly controllable for every time $T>0$, we say that it is small-time exactly controllable.
\end{defi}

It is well-known that the exact controllability of the wave equation implies the controllability in small-time of the Schr\"odinger equation, see for instance \cite[Theorem 3.1]{Mil05}. Then, from Theorem \ref{th_exact_wave_intro_neu}, we obtain the following exact controllability result. 

\begin{theorem}
\label{prop:exactcontrolschro_intro}
Assume that the graph $\Gc$ and the control set $\omega$ satisfy (GGCC). Then, the Schrödinger equation \eqref{eq_schro_control_intro} is small-time exactly controllable.
\end{theorem}

Theorem \ref{prop:exactcontrolschro_intro} is proved in Section \ref{sec_schr} and yields that (GGCC) is a sufficient condition for the exact controllability of the Schr\"odinger equation.  It is then natural to wonder if (GGCC) is a necessary condition. The graph-type domains being an intermediate geometry between the one-dimensional segments and the higher-dimensional domains, guessing the answer is not immediate. Indeed, in the standard one-dimensional cases of intervals, the controllability of the Schrödinger equation is equivalent to the one of the wave equation, $\omega$ could be any arbitrary open subset, see for instance \cite{Laurent}.
On the other hand, in higher dimensional domains, the Schrödinger equation may be controllable even if the classical geometric control condition of Bardos, Lebeau, Rauch and Taylor is not satisfied, see for instance \cite{AM14,BZ12,Jaffard} for the control of Schrödinger equation in rectangles where $\omega$ could be any arbitrary open subset, and \cite{ALM16} for the control of Schrödinger equation in the unit disk where $\omega$ has to be an open subset touching the boundary of the disk.

\medskip

In this article, we show that the graphs are closer to the higher dimensional domains. Indeed, we exhibit a example of graph where the Schr\"odinger equation is controllable even if (GGCC) is not satisfied.
It is also interesting to note how much our example depends on the diophantine properties of the lengths of its edges. Let us recall the following definition.

\begin{defi}\label{badly}
A real number $\alpha\in\RR$ is said to be {\bf badly approximable} if there exists $C>0$ such that, for all $(p,q)\in\ZZ\times\NN^*$, 
$$\left|\alpha-\frac pq\right|\geq \frac C{q^2}.$$
\end{defi}

Examples of badly approximable numbers are $\sqrt{2}$, $\sqrt{3}$, or, in general, all algebraic irrational numbers of order $2$, see Liouville's Theorem \ref{th_Liouville} recalled in Appendix \ref{diophantine}. Also notice that almost all numbers are not badly approximable and that there are explicit examples of non-badly approximable numbers as $e=\exp(1)$. We refer to Appendix \ref{diophantine} for further discussions on the subject.

\smallskip
 
Now, we consider the X graph of Figure \hyperref[fig_ex]{2}: a star-graph with four edges, two (say $e_1$ and $e_2$) of length $\ell_{\text{b}}>0$ and two ($e_3$ and $e_4$), of length $\ell_{\text{t}}>0$. The set $\omega$ has exactly two non-empty parts, one in $e_1$ and one in $e_3$. Obviously, this geometry does not satisfy (GGCC), since there exists a periodic path living in $e_2\cup e_4$, see Figure \hyperref[fig_ex]{2}. To simplify, we set here $\Delta_\Gc$ to be the operator with Dirichlet boundary conditions and $q$ to be $0$.

\begin{theorem}\label{non_ggcc_sch_intro}
Consider the above geometric setting of the X graph $\Gc$. The Schr\"odinger equation 
\begin{equation*}
\left\{
\begin{array}{ll} i \partial_{t} u = - \Delta_\Gc u + h 1_{\omega} ,& \ \ \ \ t>0,\\
u(t=0)=u_0\in L^2(\Gc), & 
\end{array}\right. 
\end{equation*}
is exactly controllable in $L^2(\Gc)$ for some time $T>0$ if and only if the ratio $\ell_{\text{t}}/\ell_{\text{b}}$ is a badly approximable number.
\end{theorem}
As noticed above, Theorem \ref{non_ggcc_sch_intro} provides both examples and counter-examples: Schr\"o\-dinger equation on the X graph is controllable for the lengths $(\ell_{\text{t}},\ell_{\text{b}})=(\sqrt{2},1)$ but not controllable for the lengths $(\ell_{\text{t}},\ell_{\text{b}})=(e,1)$. In particular, the case $(\ell_{\text{t}},\ell_{\text{b}})=(\sqrt{2},1)$ provides an example where \eqref{eq_schro_control_intro} is controllable despite (GGCC) not being satisfied and Dirichlet boundary conditions being imposed. It is worth mentioning that when $\ell_{\text{t}}/\ell_{\text{b}}$ is a badly approximable number, the Schrödinger equation on the X graph holds is controllable for some time $T>0$ but we actually do not know if the small-time controllability is satisfied or not, see Section \ref{sec:discussionschromin} for a discussion on this topic. Last but not least, recalling that the set of badly approximable numbers if of measure zero, it appears that the counterexample of Theorem \ref{non_ggcc_sch_intro} is somehow exceptional. Actually in Section \ref{sec:discussionGGCCSchroAlmost}, we show that (GGCC) is a necessary and sufficient condition for almost all star-graphs.

\medskip

Theorem \ref{non_ggcc_sch_intro} is proved in Section \ref{sec_schr}. It is based on the well-known equivalence between the exact controllability of the Schrödinger equation and a so-called resolvent estimate, see \cite{BZ,Mil05}. We then strongly exploit the presence of a symmetry in the structure of the graph. Such symmetry allows us to rewrite \eqref{eq_schro_control_intro} as three different internal control problems on three different intervals. This approach based on resolvent estimates only furnishes the controllability for some time $T>0$, the small-time exact controllability is an open problem that we discuss in Section \ref{sec:discussionschromin}.

\subsection{Null-controllability of the heat equation}

Let us now consider the case of the heat equation. We study the relation of the condition (GGCC) and the controllability of the following heat equation in $L^2(\Gc)$,
\begin{equation}\label{eq_free_heat_intro}
\left\{
\begin{array}{ll} \partial_{t} u = \Delta_\Gc u +q(x) u +  h 1_{\omega} ,& \ \ \ \ \ t>0,\\
u(t=0)=u_0\in L^2(\Gc), & 
\end{array}\right. 
\end{equation}
with $q\in L^\infty (\Gc)$ and $\Delta_{\Gc}$ the Laplacian defined on the graph $\Gc$, with either Dirichlet or Neumann conditions at the exterior vertices. In \eqref{eq_free_heat_intro}, at time $t \geq 0$, the function $u(t,\cdot) : \Gc \to \mathbb{R}$ is the state and $h(t,\cdot) : \omega \to \mathbb{R}$ is the control.

\medskip

Due to the well-known regularizing effects of the heat equation, the standard notion of the controllability of \eqref{eq_free_heat_intro} is the so-called null-controllability.
\begin{defi}
Let $T>0$. The heat equation \eqref{eq_free_heat_intro} is null-controllable at time $T$ if for every $u_0 \in L^2(\Gc)$, there exists $h \in L^2((0,T)\times \omega)$ such that the solution $u \in \Cc([0,T],L^2(\Gc))$ of \eqref{eq_free_heat_intro} with initial state $u(t=0)=u_0$ satisfies $$u(t=T) = 0.$$
When the heat equation \eqref{eq_free_heat_intro} is null-controllable for every time $T>0$, we say that it is small-time null-controllable.
\end{defi}

As for the Schr\"odinger equation, it turns out that (GGCC) is a sufficient but not necessary condition for the null-controllability of the equation \eqref{eq_free_heat_intro}.
Actually, the sufficiency can be seen as a direct corollary of the work of Apraiz and B\'arcena-Petisco \cite{AB}. Indeed, from Proposition \ref{prop_GGCC}, we know that (GGCC) is equivalent to Apraiz, Barcena-Petisco's condition recalled in Definition \ref{defi_AB}. Thus, \cite{AB} implies that \eqref{eq_free_heat_intro} is small-time null-controllable whenever as (GGCC) holds.
\begin{theorem}[{\cite[Theorem 1.4]{AB}}]\label{prop:nullcontrolheat_intro}
Let $\Gc$ and $\omega$ satisfy the (GGCC). Then the heat equation \eqref{eq_free_heat_intro} is small-time null-controllable.
\end{theorem}
 For sake of completeness, we provide another proof of Theorem \ref{prop:nullcontrolheat_intro} in Section \ref{sec_parabolic} based on the well-known ingredient that the exact controllability of the wave equation for some time $T>0$ implies the small-time null-controllability of the heat equation, see for instance \cite{Mil06}. The previous result shows that the Graph Geometric Control Condition is a sufficient condition for the small-time null-controllability for the heat equation \eqref{eq_free_heat_intro}. We next show that, as in the Schr\"odinger case, the condition is actually not necessary. Let us state the following definition.
 \begin{defi}\label{defi_sigma_approx_intro}
Let $\sigma>0$. We say that a real number $\alpha\in\RR$ is {\bf at most $\sigma-$appro\-xi\-mable} if there exists $C=C(\sigma)>0$ such that, for all $(p,q)\in\ZZ\times\NN^*$, 
$$\left|\alpha-\frac pq\right|\geq \frac C{q^{\sigma}}.$$
If $\alpha$ is at most $\sigma-$approximable for some $\sigma > 0$, we say that it is {\bf at most polynomially approximable}.
\end{defi}
Examples of at most polynomially approximable numbers are all algebraic numbers of order $d \geq 2$, see Theorem \ref{th_Liouville} recalled in Appendix \ref{diophantine}. Note also that Khinchin proves in 1926 that this set of numbers is of full Lebesgue measure. We refer to Appendix \ref{diophantine} for further discussions on the subject.

\medskip

 We consider a star-graph $\Gc$ composed by $N$ edges $e_1$, \dots, $e_N$, of respective lengths $\ell_1$, \dots, $\ell_N$, that meet at a unique interior vertex $v$. We assume that  
 $\Delta_\Gc$ is the Laplacian operator with Dirichlet boundary conditions at the $N$ exterior vertices and we take $q=0$. Assume that $\omega$ is an open set contained in the first edge $e_1$. 
\begin{theorem}
\label{prop:ggccnotnecessaryheat_intro}
Consider the above setting with a star-graph $\Gc$ controlled in only one edge. Assume that all ratio  $\ell_j/\ell_k$ with $j\neq k \in \{2, \dots, n\}$ are at most polynomially approximable numbers. Then, the heat equation
\begin{equation*}
\left\{
\begin{array}{ll} \partial_{t} u = \Delta_\Gc u +  h 1_{\omega} ,& \ \ \ \ \ t>0,\\
u(t=0)=u_0\in L^2(\Gc), & 
\end{array}\right. 
\end{equation*}
is small-time null-controllable.
\end{theorem}

Theorem \ref{prop:ggccnotnecessaryheat_intro} is proved in Section \ref{sec_parabolic}. It shows that (GGCC) is not a necessary condition for the null-controllability in $L^2(\Gc)$, even with Dirichlet boundary conditions. It presents an example where the choice of the lengths of the edges of the graph guarantees controllability, even if we control only one edge.  
Indeed, the choice of the lengths of the edges allows us to “propagate” the control from the edge $e_1$ to the rest of the structure. For instance, we can consider the three star-graph $\Gc$, composed with edges $\{e_1, e_2, e_3\}$ of respective length $(1,1,\sqrt{2})$ or $(1,1,e)$, all connected at a central vertex $v$. Then, (GGCC) is not satisfied because $\mathcal{E} \setminus \mathcal{E}_{\omega} = \{e_2, e_3\}$ is a tree and contains two exterior vertices of the original graph $\Gc $, see for instance the third point of Theorem \ref{prop_GGCC}. Moreover, $\sqrt{2}$ and $e=\exp(1)$ are both polynomially approximable numbers, see Appendix \ref{diophantine}.

\medskip

Note that there is a strong difference between the exact controllability result for the Schrödinger equation stated in Theorem \ref{non_ggcc_sch_intro} and the null-controllability result for the heat equation stated in Theorem \ref{prop:ggccnotnecessaryheat_intro}. Indeed the first one indicates that the Schrödinger equation is exactly controllable only for an exceptional set of X graphs, while the second one proves that the heat equation set on a star-graph, controlled in only one edge, is small-time null-controllable for almost all lengths of the edges, see Section \ref{sec:discussionGGCCSchroAlmost}.

\medskip

The proof of Theorem \ref{prop:ggccnotnecessaryheat_intro} is strongly based on \cite[Corollary 8.6]{DZ06}, considering the possibility of controlling the heat equation on a star-graph with only one boundary control. We then employ a general strategy to pass from the boundary control result to an interior control result. Such a procedure has been widely used, see for instance \cite[Theorem 2.2]{AK11}.\\

\medskip

To conclude this introduction, we sum up in Table 1 the internal controllability properties of the wave equation, Schrödinger equation and heat equation according whether (GGCC) holds or not.

\begin{table}[ht] \label{tableau}
\begin{center}
  \begin{tabular}{|M{5cm}|M{2.5cm}|M{6.2cm}|}
\hline
{\bf Internal control on graph } &  {\bf  (GGCC)} & {\bf  no (GGCC)}\\
\hline
Wave eq. with Dirichlet b.c. & Controllable:\linebreak Theorem \ref{th_exact_wave_intro_dir}  & Never controllable:\linebreak Theorem \ref{th_exact_wave_intro_dir} \\
\hline
Wave eq. with mixed b.c. & Controllable:\linebreak Theorem \ref{th_exact_wave_intro_neu}  & Sometimes controllable:\linebreak example provided by Theorem \ref{th_neumann}\linebreak{}
({\it Conjecture \ref{conj_neu}:  almost never true})\\
\hline
Schrödinger equation & Controllable:\linebreak Theorem \ref{prop:exactcontrolschro_intro} &  Sometimes controllable:\linebreak{}example provided by Theorem \ref{non_ggcc_sch_intro}\linebreak{}
({\it Conjecture \ref{conj_sch}:  almost never true})\\
\hline
Heat equation & Controllable:\linebreak  Theorem \ref{prop:nullcontrolheat_intro} &  Sometimes controllable:\linebreak{}example provided by Theorem \ref{prop:ggccnotnecessaryheat_intro}\linebreak{}({\it Conjecture \ref{conj_heat}: almost always true})\\
\hline
\end{tabular}  
\end{center}
\caption{\it A summary of the results of the present article. The condition (GGCC) is always sufficient for the internal control of the PDEs considered. Its necessity is more complex and depends on the considered PDE and the geometry of the graph. Notice that, for particular graphs with specific ratio of lengths of edges and not satisfying (GGCC), the existence of exact uncontrolled modes precludes the controllability. So the existence of non-controllable examples in the right column is already well known.}
\end{table}

\subsection*{Acknowledgments} 
The authors thank Tanguy Rivoal for interesting discussions and for his help with the appendix.
The second author acknowledges the support of the Agence nationale de la recherche of the French government through the grant {\it QuBiCCS} (ANR-24-CE40-3008). The last author is partially supported by the Project TRECOS ANR-20-CE40-0009 funded by the ANR (2021-2025).


\section{Preliminaries and main notations}\label{preliminaries}

In this section, we introduced the main notations adopted in the work.

\subsection{Graph notions and notation}

Throughout the rest of the article, $\Gc$ is a graph endowed with the following structure:

\begin{itemize}
\item We denote $\Ec=\{e_j,\ j\in J\}$ the set of edges, where $J$ is a finite set of $\NN$, and $ \Vc=\{v_k,\ k\in K\}$ the set of vertices, where $K$ is a finite set of $\NN$.
\item For all $j\in J$, the edge $e_j$ links two vertices $v_k$ and $v_{k'}$ and we set $\Vc_j:=(k,k')\in K^2$. The edge $e_j$ is associated to a finite length $\ell_j>0$ and is assimilated to the segment $(0,\ell_j)\subset\RR$, where $0$ is the end attached to $v_k$ and $\ell_j$ the one attached to $v_{k'}$. 
\item For all $k\in K$, we define $\Ec_k\subset J$ as the set of edges having $v_k$ as vertex, that is 
\[\Ec_k:=\big\{j\in J,\ \Vc_j\in\{k\}\times\NN \text{ or }\Vc_j\in \NN\times\{k\}\big\}.\] 
\item The set of exterior vertices $\Vc_{\rm ext}$ is the set of vertices $v_k$ with $\# \Ec_k=1$ and $K_{\rm ext}$ are the corresponding indices. We denote $\Vc_{\rm int}$ the set of interior vertices of $\Gc$, for which $\# \Ec_k\geq 2$, and $K_{\rm int}$ the corresponding indices. 
\end{itemize}

We recall that we allow very general graphs, including loops (an edge $e_j$ may link $v_k$ to itself, i.e., $\Vc_j = (k,k)$ is possible) and parallel edges (two edges $e_j$ and $e_{j'}$ may have $j \neq j'$ but $\Vc_j = \Vc_{j'}$). As mentioned in the introduction, we only require finiteness, meaning that $\Gc$ is composed of a finite number of edges of finite length. Finally, note that our setting induces an orientation of the edges, but this is only a matter of notation to assimilate the edge $e_j$ to the segment $(0, \ell_j)$, where $\ell_j$ is the length of the edge. This orientation has no consequence on the analysis.

\medskip

We assimilate the graph $\Gc$ to the disjoint union $\sqcup_{j\in J} (0,\ell_j)$ of the edges, where each segment belongs to a different copy of $\RR$. We recall the notion of {\bf path} on the graph: 
\begin{itemize}
\item we start at some point $x_0$ in one of the edges $e_{j_0}$: we choose a direction and we follow the edge until we reach an end;
\item if this end corresponds to an exterior vertex, we come back in $e_{j_0}$ in the other direction; if this end corresponds to an interior vertex, we choose any other edge $e_{j_1}$ attached to this vertex, with $j_1\neq j_0$, and start following this edge $e_{j_1}$;
\item we continue the route and stop at some point. 
\end{itemize}
The total {\bf length} of a path is the sum of all the covered part of the graph. 
We also recall that $\omega$ is an open subset of the graph $\Gc$, that is an open subset of the disjoint union of the edges. The set 
\[\Ec_\omega:=\{e_j \in \mathcal E,\ e_j\cap\omega \neq \emptyset\}, \]
is the set of {\it controlled} edges. A path is said to {\it meet} $\omega$ if the route passes through a point belonging to $\omega$.

\subsection{Sobolev spaces on graphs}
In this paper, we are considering PDEs defined on the edges of a graph $\Gc$. It is natural to introduce the following space of functions. First, we define
$$L^2 (\Gc) := \prod_{j\in J} L^2(e_j):=L^2\big(\sqcup_{j\in J} e_j\big),$$
to be the Hilbert space endowed with the inner product
$$\Big((u_j),(w_j)\Big)_{L^2(\Gc)} =  \sum_{j\in J} \int_0^{\ell_j} u_j(x) \, \overline{w_j(x)} \dd x.$$
The above sum is a sum of the $L^2$-products on all the edges. In this paper, we will write this kind of terms more shortly as the integral on $\Gc$. In other words, for any integrable function $f$ of $L^1\big(\sqcup_{j\in J} e_j\big)$, we set
$$ \int_{\Gc} f(x) \dd x ~:=~ \sum_{j\in J}~ \int_0^{\ell_j} f(x) \dd x.$$
We also use here and everywhere the abuse of notation of assimilating $e_j$ to $(0,\ell_j)$ and we denote by 
$u_j(v)$ the value $u_j(0^+)$ if $v$ is the incident vertex of $e_j$ assimilated to $0$ (resp. the value $u_j(\ell_j^-)$ if $v$ is the incident vertex of $e_j$ assimilated to $\ell_j$) and we set $\frac{\dd u_j}{\dd n_j}(v)=-\partial_x u_j(0^+)$ (resp. $=\partial_x u_j(\ell_j^-)$) the exterior derivative of $u$ at the edge $v$. 
We introduce the following Sobolev spaces. Let
$$H^1(\Gc):=\Big\{u=(u_j)_{j\in J} \in \prod_{j\in J} H^1(e_j)~:~ u_j(v_k)=u_{j'}(v_k),\  k\in K_{\rm int},\  j,j'\in \Ec_k\Big\},$$
be the Hilbert space endowed with the inner product
$$\Big((u_j),(w_j)\Big)_{H^1(\Gc)} =  \int_\Gc u_j(x) \, \overline{w_j(x)} \dd x + \int_\Gc \partial_{x} u_j(x) \, \overline{\partial_x w_j(x)} \dd x .$$
Notice that the condition in the definition of $H^1(\Gc)$ means that the function $u$ is continuous at the interior vertices. We also define the subspace
\[ H_0^1(\Gc) := \big\{u=(u_j)_{j\in J} \in H^1(\Gc)~:~ u_j(v_k)=0,\  k\in K_{\rm ext},\  j\in \Ec_k\big\}\]
of $H^1-$functions satisfying the Dirichlet boundary conditions at the exterior vertices. 
We finally introduce the Hilbert space 
\begin{align*}
H^2(\Gc):=\Big\{u=(u_j)_{j\in J} \in \prod_{j\in J} H^2(e_j)~:&\ u_j(v_k)=u_{j'}(v_k),\ k\in K_{\rm int},\ j,j'\in \Ec_k;\\
&\ \sum_{j\in\Ec_k} \frac{\dd u_j}{\dd n_j}(v_k)=0,\ k\in K_{\rm int}\Big\},
\end{align*}
endowed with the inner product
$$ \Big((u_j),(w_j)\Big)_{H^2(\Gc)} \\
    =  \int_\Gc u_j(x) \,\overline{w_j(x)} \dd x + \int_\Gc \partial_{x} u_j(x) \, \overline{\partial_x w_j(x)} \dd x+   \int_\Gc\partial_{x}^2 u_j(x) \, \overline{\partial_{x}^2 w_j(x)} \dd x.
$$
In the definition of $H^2(\Gc)$, notice the presence of the classical Kirchhoff boundary condition saying that the flux through any interior vertex is conserved.

\subsection{The Laplace operator and well-posedness of PDEs on graphs}

We now define the Laplacian operator $ \Delta_\Gc$ equipped with Dirichlet and Neumann boundary conditions on the exterior vertices. We denote by $K_{\rm ext}^D\subseteq K_{\rm ext}$ the indices of the exterior vertices with Dirichlet boundary conditions and $K_{\rm ext}^N\subseteq K_{\rm ext}$ the indices of the exterior vertices with Neumann boundary conditions.
We can then define the Laplacian $\Delta_\Gc$ as the operator with domain 
\begin{align*}D( \Delta_\Gc)=\Big\{u=(u_j)_{j\in J} \in H^2(\Gc)~:&\ \partial_x u_j(v_k)=0,\  k\in K_{\rm ext}^N,\  j\in \Ec_k;\\
&\ u_j(v_k)=0,\  k\in K_{\rm ext}^D,\  j\in \Ec_k\Big\}.\end{align*}
and such that, for every $u\in D( \Delta_\Gc)$, its action is $\big( \Delta_\Gc u\big)_j = \partial_{x}^2 u_j$ for every $j\in J$.
The Laplacian operator $ \Delta_\Gc$ is a non-positive self-adjoint operator with a compact resolvent, hence its spectrum $\sigma( \Delta_{\Gc})$ is a discrete set, see \cite[Theorem 3.1.1]{BK13}. We finally notice that
\[ D(|\Delta_{\Gc}|^{1/2}) = \big\{u=(u_j)_{j\in J} \in H^1(\Gc)~:~ u_j(v_k)=0,\  k\in K_{\rm ext}^D,\  j\in \Ec_k\big\},\]
and we set in short
$$H^1_\Delta(\Gc) := D(|\Delta_{\Gc}|^{1/2}).$$
In the specific case of the Dirichlet Laplacian operator, we have of course $H^1_\Delta(\Gc)=H^1_0(\Gc)$. Its dual space, with respect to the pivot space $L^2(\Gc)$, will be denoted by
$$ H^{-1}(\Gc) := H^1_0(\Gc)'.$$

\medskip

By using the properties of the operator \( \Delta_\Gc \), in particular the fact that it is sectorial,  and standard semigroup theory (see for instance \cite[Appendix A and Chapter 2]{Cor07}), we obtain the following classical well-posedness results for PDEs on graphs. Let $T>0$, $q \in L^{\infty}(\Gc)$ and $f \in L^2(0,T;L^2(\Gc))$.
\begin{itemize}
    \item Let $a \in L^{\infty}(\Gc)$, $a \geq 0$. For all $(u_0, u_1) \in H^1_\Delta(\Gc)\times L^2(\Gc)$, the wave equation
\begin{equation}\label{wave_wp}
\left\{
\begin{array}{ll} 
\partial_{t}^2 u+a(x) \partial_t u = \Delta_\Gc u+ q(x) u + f,\ \ \ \ & t\in(0,T),\\
(u,\partial_t u)(t=0)=(u_0,u_1) \in H^1_\Delta(\Gc)\times L^2(\Gc), & 
\end{array}\right. 
\end{equation}  
admits a unique solution $u  \in  \Cc^1\big([0,T],L^2(\Gc)\big)\cap \Cc^0\big([0,T],H^1_\Delta(\Gc)\big)$.
\item For all $u_0 \in L^2(\Gc)$ and for all $f \in L^2(0,T;L^2(\Gc))$, the Schrödinger equation
\begin{equation}\label{schro_wp}
\left\{
\begin{array}{ll} 
i \partial_{t} u = - \Delta_\Gc u - q(x) u + f,\ \ \ \ & t\in(0,T),\\
u(t=0)=u_0 \in L^2(\Gc), & 
\end{array}\right. 
\end{equation}
admits a unique solution $u  \in   \Cc^0\big([0,T],L^2(\Gc)\big)$.
\item For all $u_0 \in L^2(\Gc)$ and for all $f \in L^2(0,T;L^2(\Gc))$, the heat equation
\begin{equation}\label{heat_wp}
\left\{
\begin{array}{ll} 
\partial_{t} u =  \Delta_\Gc u + q(x) u + f,\ \ \ \ & t\in(0,T),\\
u(t=0)=u_0 \in L^2(\Gc), & 
\end{array}\right. 
\end{equation}
admits a unique solution $u  \in   \Cc^0\big([0,T],L^2(\Gc)\big)$.
\end{itemize}

In particular, these well-posedness results ensure that the wave equation \eqref{intro_wave}, the Schrödinger equation \eqref{eq_schro_control_intro} and the heat equation \eqref{eq_free_heat_intro} are well-posed linear controlled systems in the sense of \cite[Section 2.3]{Cor07}. Finally, the Cauchy problem associated to the damped wave equation \eqref{intro_wave_damp} is also well-posed.

\section{Study of the Graph Geometrical Control Condition}\label{sec_GGCC}

The aim of this section is to study the Graph Geometric Control Condition and provide some equivalent formulations. 

\subsection{Definitions of the APB Condition and the watershed condition}
Recall that a \textit{cycle} is a sequence of edges $e_0$, $e_1$, $\dots$ , $e_n$ in $\Ec$ such that $\Ec_k\cap \Ec_{k+1}\neq \emptyset$ for every $k \in \{0, \dots, n-1\}$ (meaning that the edges are adjacent) and $e_0 = e_n$. We are ready to formally introduce the {\it Apraiz-B\'arcena-Petisco Condition} (ABP condition in short) mentioned in the introduction and adopted in the work \cite{AB}.

\begin{defi}\label{defi_AB}
We say that a graph $\Gc$ and an open set $\omega\subset\Gc$ satisfy the {\bf ABP Condition} if:
\begin{enumerate}
\item[(i)] the set $\Ec\setminus\Ec_\omega$ of uncontrolled edges does not contain any cycle of edges,
\item[(ii)] there exists an injective function $u$ from $\Ec\setminus\Ec_\omega$ into the set $\Vc_{\rm int}$ of interior vertices such that, for all $e_j\in \Ec\setminus\Ec_\omega$, $u(e_j)$ is one of the incident vertices $\Vc_j$ of $e_j$.
\end{enumerate}
\end{defi}

We also introduce a geometric description following the work of Avdonin and Zhao. They show in \cite{AZ22} that this description is related to the optimal time of control.
\begin{defi}\label{defi_watershed}
We say that a graph $\Gc$ and an open set $\omega\subset\Gc$ satisfy the {\bf watershed Condition} if there exists a finite number of open paths $(p_i)_{i=1\dots N}$ on the graph such that:
\begin{enumerate}
\item each path $p_i$ has at least one end in $\omega$,
\item the union $\cup_i p_i$ of the paths covers $\Gc\setminus \omega$,
\item all the paths are disjoints, 
\item if a vertex belongs to the closure of several paths, only two situations are allowed: each incoming edge belongs to a different path or there are exactly two edges belonging to the same path and all the others edges belong to different paths. 
\end{enumerate}
\end{defi}
The above condition can be seen as the fact that the uncontrolled part $\Gc\setminus\omega$ of the graph can be covered by watersheds consisting of rivers starting from $\omega$ and flowing down, up to reach an exterior vertex or $\omega$ again. The last condition states that no river can cross another one and if several rivers meet at some vertex, only one can continue after this point. 

\subsection{Equivalent formulations of (GGCC)}

As explained in the introduction, the above criteria are equivalent to the Graph Geometric Control Condition. The purpose of the present section is to provide a proof of this fact, i.e. prove Proposition \ref{prop_GGCC}.

\begin{proof}[Proof of Proposition \ref{prop_GGCC}] 
We prove 1 $\Leftrightarrow$ 2, 2 $\Leftrightarrow$ 3, 1 $\Rightarrow$ 4  $\Rightarrow$ 2, 3 $\Rightarrow$ 5 $\Rightarrow$ 2 and 3 $\Rightarrow$ 6 $\Rightarrow$ 5.

\medskip

\noindent
1 $\Rightarrow$ 2.  If the graph contains a cycle or a path connecting two exterior vertices that does not intersect $\omega$, then this cycle or path induces a periodic path of infinite length that never intersects $\omega$.

\medskip

\noindent
2 $\Rightarrow$ 1. Suppose that every cycle of the graph and every path connecting two exterior vertices contains a part of $\omega$. Let $\ell$ be the sum of the lengths of all edges in the graph. Consider a path of length at least $\ell$; it must necessarily pass through the same point twice. This can occur in only two ways: either the path contains a cycle and thus part of $\omega$, or the path has made a U-turn at an exterior vertex. In the latter case, we wait another time interval of $\ell$.  Once again, the path must pass through the same point twice. At this point, either it contains a cycle (and thus part of $\omega$), or it connects to another exterior vertex. In this case, the path has followed a path connecting two exterior vertices and therefore intersects $\omega$.

\medskip

\noindent
2 $\Leftrightarrow$ 3. This equivalence is immediate: (GGCC) means that each connected uncontrolled subgraph of $\Gc$ cannot contain a cycle or two exterior vertices of $\Gc$, which is exactly saying that each such subgraph is indeed a tree containing at most one exterior vertices of the original graph. 

\medskip

\noindent
1 $\Rightarrow$ 4. Following repeatedly any periodic path leads to a path of infinite length, which must necessarily intersect $\omega$.

\medskip

\noindent
4 $\Rightarrow$ 2. Assume, for the sake of contradiction, that the point 4 is satisfied but the point 2 is not. If $\omega$ does not intersect a cycle or a path connecting two exterior vertices, then it is possible to construct a periodic path that does not intersect $\omega$, which leads to a contradiction. Therefore, Property 2 holds.

\medskip

\noindent
3 $\Rightarrow$ 5. Because we have already proved that 2 $\Leftrightarrow$ 3, the condition on the cycles is automatically satisfied. We construct as follows the injective function described in the point (ii) of ABP Condition (see Definition \ref{defi_AB}). We remove all controlled edges from the graph and we obtain a forest where each tree contains at most one exterior vertex from the original graph. If such a vertex exists, then we designate it as the root of the tree; otherwise, we choose any vertex of the tree as the root. For each of the trees, we place the root at the bottom and construct the injective function by associating each edge with the vertex “above”. Since the only potentially restricted vertex is at the root, this procedure successfully constructs the injective function required by ABP Condition.

\medskip 

\noindent
5 $\Rightarrow$ 2. First, notice that the condition on cycles is the same in both items. 
Assume now, for the sake of contradiction, that there exists a path connecting two exterior vertices without intersecting $\omega$. This path consists of $N$ edges and $N-1$ interior vertices. In such a case, it would be impossible to construct an injective function as described in the point (ii) of  ABP Condition. Therefore, Property 2 holds. 

\medskip 

\noindent
3 $\Rightarrow$ 6. For each uncontrolled tree, we can construct a watershed as follows. If the tree contains an (unique) exterior vertex of $\Gc$, put it as the root, else choose any vertex as root. Start all the paths for the upper leaves of the tree: all these leaves are at the boundary of the controlled part $\omega$ and we can start the “rivers” from there. Then simply extend the “river” paths downwards. When several rivers meet at a vertex, stop all of them except one. Obviously, we can cover all the tree up from its leaves to its root.

\medskip 

\noindent
6 $\Rightarrow$ 5. Assume that a “watershed” covers the uncontrolled part of $\Gc$ that is that we have a finite number of disjoint “river” paths $(p_i)$ covering $\Gc\setminus\omega$. Each $(p_i)$ has at least one end in $\omega$ and we use it to define a sense: the “upstream” of $p_i$ must arrive at such a point, called the “source”, and the “downstream” sense is define the other way. If we start from an uncontrolled edge, it must belong to one of the rivers and we can start to move upstream. If we arrive at a “confluence” vertex, we can choose still follow the same river or choose any tributary river. In any case, we must still travel upstream because this vertex cannot be the source of the tributary rivers. By finiteness of the paths (they are disjoints in a finite graph), we must arrive at a source. This excludes the possibility of having a cycle in  $\Gc\setminus\omega$ and shows (i) of Definition \ref{defi_AB}. To obtain (ii), we simply associate to any uncontrolled edge the first vertex reached when moving upstream. This mapping is injective because it is impossible to have two rivers flowing downstream from the same vertex (one must be stopping at it and, by assumption, it cannot be the source). Also notice that the image vertex cannot be an exterior vertex of $\Gc$ because the followed river cannot start from here and so there must exist another edge to travel up. 
\end{proof}


\section{Control and observation of the wave equation}\label{sec_control_wave}

The aim of this section is to prove Theorems \ref{th_exact_wave_intro_dir} and \ref{th_exact_wave_intro_neu}, i.e. that (GGCC) is equivalent to the exact controllability of the wave equation with Dirichlet boundary conditions and at least a sufficient condition for more general boundary conditions.
We introduce the adjoint problem of \eqref{intro_wave}, which is the free wave equation
\begin{equation}\label{eq_dual}
\left\{\begin{array}{ll} \partial_{t}^2 \varphi = \Delta_\Gc \varphi+q(x) \varphi, & t>0,\\
(\varphi,\partial_t \varphi)(t=0)=(\varphi_0,\varphi_1)\in H^1_\Delta(\Gc)\times L^2(\Gc). 
\end{array}
\right. 
\end{equation}
It is well-known that the exact controllability of the wave equation \eqref{intro_wave} is equivalent to a so-called observability inequality for \eqref{eq_dual}.
\begin{theorem}[{\cite[Theorem 2.42]{Cor07}}]
\label{th:obsexactwave}
Let $T>0$. The wave equation \eqref{intro_wave} is exactly controllable at time $T>0$ in the sense of Definition \ref{defi_exactwave} if and only if there exists $C > 0$ such that for every $(\varphi_0,\varphi_1) \in H^1_\Delta(\Gc)\times L^2(\Gc),$ the solution $\varphi\in\Cc^1([0,T],L^2(\Gc))\cap \Cc^0([0,T],H^1_\Delta(\Gc))$ satisfies
\begin{equation}\label{eq_obs}
\|(\varphi_0,\varphi_1)\|_{H^1_\Delta(\Gc)\times L^2(\Gc)}^2 \leq C \int_0^T \int_{\omega} |\partial_t \varphi(x,t)|^2 \dd x \dd t.
\end{equation}
\end{theorem}

\subsection{Unobserved quasimodes with Dirichlet boundary conditions}\label{sec_quasimode}

In this subsection, we construct quasimodes, whose supports remain disjoint from the observation set $\omega$, of arbitrarily high frequency for the Laplacian operator with Dirichlet boundary conditions, on any metric graph whenever (GGCC) is not satisfied. 

\begin{prop}\label{prop_quasimode}
Let $\Gc$ be a connected graph and $\omega$ be an open set of $\Gc$. Assume that (GGCC) is not satisfied and that $\Delta_\Gc$ is the Laplacian operator with Dirichlet boundary conditions. Then, there exists a real sequence $(\mu_n)$ converging to $+\infty$ and a sequence $(u_n)$ of functions in $D(\Delta_\Gc)$ such that:
\begin{itemize}
\item each function $u_n$ vanishes everywhere in the set $\omega$,
\item $\|u_n\|_{L^2(\Gc)}\rightarrow 0$ when $n\rightarrow +\infty$,
\item $\|\grad u_n\|_{L^2(\Gc)}=1$,
\item and $\|(\Delta_\Gc+\mu_n^2)u_n\|_{L^2(\Gc)} \rightarrow 0$ when $n\rightarrow +\infty$.
\end{itemize}
\end{prop}
The above quasimodes provide approximate solutions of the wave equation, that are not controlled or observed since they vanish in $\omega$. We will use these quasimodes to disprove the observability inequality \eqref{eq_obs}, see Corollary \ref{coro_obs_free_wave} below.
This construction is the main argument to show that (GGCC) is necessary to observe or control the wave equation with Dirichlet boundary conditions. When the lengths of the edges belonging to an uncontrolled path of the graph have rational ratios, it is usual to construct exact modes, see for instance \cite[Remarks 1.1 and 1.2]{AB}. 

\medskip

In the general case, we need to approximate real numbers by rational numbers, i.e. we need the tools of Diophantine approximation, more precisely the simultaneous version of Dirichlet's approximation theorem, recalled in Theorem \ref{th_Dirichlet} in Appendix \ref{diophantine}. This type of approximation has already been used in similar problems in \cite[Section 6.3]{DZ06} and \cite[Chapter VII, Section 1]{AI95}. For our applications, we would like to consider high frequencies and we need to ensure that the rational approximations $p/q$ have denominators $q$ as large as needed. Thus, we will use this obvious consequence of Dirichlet's approximation theorem.  
\begin{prop}\label{coroprop_Dirichlet}
For any real numbers $\alpha_1$,\ldots, $\alpha_d$ and any $N\in\NN^*$, there exist integers $p_j$ and $q$ with $\lfloor \sqrt{N}\rfloor \leq q\leq N^{d}$, such that 
$$\forall j=1,\ldots, d~,~~\left|\alpha_j - \frac {p_j}q\right| \leq \frac 1{q\sqrt{N}}.$$
\end{prop}
In order to enlighten the present part, we postpone the proof of the above result in Appendix \ref{diophantine}.

\begin{proof}[Proof of Proposition \ref{prop_quasimode}]
Assume that (GGCC) fails: there exists a path, linking two exterior vertices or being a cycle, that does not meet $\omega$. To simplify the notations, up to re-indexing the edges, assume that this path consists in $d$ edges $e_1$,\ldots ,$e_d$, whose lengths are $\ell_1$,\ldots ,$\ell_d$. Up to re-indexing the vertices, we also assume that $e_j$ and $e_{j+1}$ are connected at the vertex $v_j$ and, if the path is a cycle, that $e_d$ and $e_1$ are connected at the vertex $v_d$. The other edges and vertices of the graph are attached to this structure, but we do not need any further description of this part of the graph. 

\medskip

Let $L= \sum_{j=1}^d \ell_j$ be the total length of the uncontrolled path. Let $n\in\NN^*$, to be fixed later. We apply Proposition \ref{coroprop_Dirichlet} to $\alpha_j=\frac{\ell_j}L$ and $N=n$. We obtain approximations $\frac{p_{j,n}}{q_n}$ of the ratios $\frac{\ell_j}L$ such that $\lfloor \sqrt{n}\rfloor \leq q_n\leq n^d$ and there exist real numbers $\varepsilon_{j,n}$ such that, for all $j=1\ldots d$, we have 
$$\left|\frac{\ell_j}L - \frac{p_{j,n}}{q_n}\right|\leq \frac 1{q_n\sqrt{n}}~~\text{ i.e. }~~
\frac{p_{j,n}}{\ell_j} = \frac {q_n}L + \frac{\varepsilon_{j,n}}{\ell_j}~~~\text{ with }~~~|\varepsilon_{j,n}|\leq \frac 1 {\sqrt{n}}.$$
Note that $p_{j,n}=q_n \frac{\ell_j}L +\varepsilon_{j,n}$ with $q_n\rightarrow +\infty$. Thus, if $n$ is large enough, we must have $p_{j,n}\geq 1$ for all $j$, which will be assumed from now on. We set 
$$ \mu_n:=\frac {2\pi q_n}L~~~\text{ and }~~~\mu_{j,n}:=\frac {2\pi p_{j,n}}{\ell_j} = \mu_n + 2\pi \frac{\varepsilon_{j,n}}{\ell_j}.$$
We define the function $u_n\in D(\Delta_\Gc)$ as follows:
\begin{itemize}
\item on the edge $e_j$, with $1\leq j\leq d$, we set 
$$\forall x\in(0,\ell_j)~,~~u_{j,n}(x):= \sqrt{\frac 2{L}} \frac{1}{\mu_{j,n}}\sin(\mu_{j,n} x),$$
\item outside the uncontrolled path, on the other edges $e_j$ with $j>d$, we set $u_{j,n}\equiv 0$. 
\end{itemize}
We can check that, on all the edges, $u_{j,n}$ vanishes on both ends so that $u_n$ is continuous at the interior vertices and satisfies the Dirichlet boundary condition at the exterior vertices. On the edges where $u_n$ is not zero, we have $\partial_x u_{j,n}(0)=\partial_x u_{j,n}(\ell_j)=\sqrt{\frac 2{L}}$. At the interior vertices belonging to the uncontrolled path, exactly two such edges are connected, with opposite exterior derivatives, and the Kirchhoff  condition is satisfied. At the other vertices, all the derivatives are zero so the Kirchhoff condition is also satisfied. We have just checked that $u_n$ belongs to $D(\Delta_\Gc)$. 

\medskip

The computation of the $L^2$-norm of $u_n$ is explicit. Since 
$$\frac{\mu_n} {\mu_{j,n}}=1-2\pi \frac{\varepsilon_{j,n}}{\ell_j\mu_{j,n}}=1-\frac{\varepsilon_{j,n}}{p_{j,n}}\longrightarrow 1,$$ 
we obtain
$$\|u_n\|_{L^2(\Gc)}^2=\sum_{j=1}^d \frac{\ell_j}{L} \frac{1}{\mu_{j,n}^2}=\frac 1{\mu_n^2} \sum_{j=1}^d \frac{\ell_j}{L}\left(1-\frac{\varepsilon_{j,n}}{p_{j,n}}\right)^2\longrightarrow 0$$
when $n$ goes to $+\infty$. The computation of the norm of the gradient of $u$ is even easier
$$\|\grad u_n\|_{L^2(\Gc)}^2=\sum_{j=1}^d \frac{\ell_j}{L}=1.$$
The computation of $(\Delta_\Gc+\mu_n^2)u_n$ is also explicit and we obtain 
\begin{align*}
\|(\Delta_\Gc+\mu_n^2)u_n\|_{L^2(\Gc)}^2&=\sum_{j=1}^d \frac{\ell_j}{L}\left(-\mu_{j,n}+\frac{\mu_n^2}{\mu_{j,n}}\right)^2\\
&=\sum_{j=1}^d \frac{\ell_j}{L}\left(-\mu_n-2\pi\frac{\varepsilon_{j,n}}{\ell_j}+\mu_n\big(1-\frac{\varepsilon_{j,n}}{p_{j,n}}\big)\right)^2\\
&=\sum_{j=1}^d \frac{\ell_j}{L} |\varepsilon_{j,n}|^2 \left(\frac{2\pi}{\ell_j}+ \frac{\mu_n}{p_{j,n}} \right)^2.
\end{align*}
Finally, notice that $\frac{\mu_n}{p_{j,n}}=\frac{2\pi}L\cdot \frac{q_n}{p_{j,n}}$ and that $\frac {p_{j,n}}{q_n}$ has be constructed to be an approximation of $\frac{\ell_j}{L}$. Thus, $\frac{\mu_n}{p_{j,n}}$ is equivalent to $\frac{2\pi}{\ell_j}$ and both are bounded when $n$ goes to $+\infty$. So, the above estimate yields 
$$\|(\Delta_\Gc+\mu^2)u_n\|_{L^2(\Gc)}^2 \leq \frac{C}{n}$$ 
where $C$ only depends on the geometry of the graph and of the uncontrolled path. 
\end{proof}

As a consequence of Proposition \ref{prop_quasimode}, we prove in the next result that (GGCC) is necessary for the observation of the wave equation with Dirichlet boundary conditions.
\begin{coro}\label{coro_obs_free_wave}
Let $\Gc$ be a connected graph and $\omega$ be an open set of $\Gc$. Assume that (GGCC) is not satisfied and that $\Delta_\Gc$ is the Laplacian operator with Dirichlet boundary conditions. Then, for any $q\in L^\infty(\Gc)$, $T>0$ and $C>0$, there exists a solution $\varphi\in\Cc^1([0,T],L^2(\Gc))\cap \Cc^0([0,T],H^1_0(\Gc))$ of \eqref{eq_dual}
such that
\begin{equation}\label{eq_prop_obs_wave_bis}
\|(\varphi_0,\varphi_1)\|_{H^1_0(\Gc)\times L^2(\Gc)}^2 \geq C \int_0^T \int_{\omega} |\partial_t \varphi(x,t)|^2 \dd x \dd t.
\end{equation}
\end{coro}
\begin{proof}
We apply Proposition \ref{prop_quasimode}. We obtain a sequence of quasimodes $(u_n,\mu_n)$ as described in the proposition. We set 
$$\varphi_0=u_n~,~~\varphi_1=0~\text{ and }~\psi(x,t):=\cos(\mu_n t)u_n(x).$$
When $n$ goes to $+\infty$, we have 
$$\|(\varphi_0,\varphi_1)\|_{H^1_0(\Gc)\times L^2(\Gc)}^2=\|\grad u_n\|_{L^2(\Gc)}^2 + \|u_n\|_{L^2(\Gc)}^2 \longrightarrow 1.$$
Let $\varphi$ the solution of \eqref{eq_dual} with initial data $(\varphi_0,\varphi_1)$. Then, $w:=\varphi-\psi$ solves
\begin{equation}\label{eq_preuve_quasimode}
\left\{\begin{array}{ll} \partial_{t}^2 w - (\Delta_\Gc + q(x))w= (\Delta_\Gc+q(x) +\mu_n^2)\psi, & t>0,\\
w(t=0)=\partial_t w(t=0)=0. & 
\end{array}\right. 
\end{equation}
For all $t\geq 0$, we set $h(t):=(\Delta_\Gc+q(x) +\mu_n^2)\psi$ and we notice that
$$\|h(t)\|_{L^2(\Gc)}\leq 
\|(\Delta_\Gc+\mu_n^2)u_n \|_{L^2(\Gc)} + \|q\|_{L^\infty(\Gc)}\|u_n \|_{L^2(\Gc)}=:\varepsilon_n \longrightarrow 0.$$ 
Finally, we use a classical energy estimate: set 
$$\Nc(t):=\frac 12  \int_\Gc (|\partial_t w|^2 + |\grad w|^2 + |w|^2)(t),$$
we have
\begin{align*}
\partial_t \Nc(t)&~=~ \int_\Gc \partial_t w(t) (\partial_{t}^2 w - \Delta_\Gc  w + w)(t)~= ~\int_\Gc \partial_t w(t) \big(h(t)+ (q(x)+1) w(t)\big)\\
&~\leq ~ \big(\|q\|_{L^\infty(\Gc)}+2\big) \Nc(t) + \frac 12 \| h(t)\|_{L^2(\Gc)}^2~\leq ~ \big(\|q\|_{L^\infty(\Gc)}+2\big) \Nc(t) + \frac {\varepsilon_n^2}2. 
\end{align*}
Since $\Nc(0)=0$, Gr\"onwall Lemma concludes that, for all $t\geq 0$,
$$\Nc(t)\leq t\frac{\varepsilon_n^2}2  e^{(\|q\|_{L^\infty(\Gc)}+2)t}.$$
In particular, in the observation set $\omega$, $\psi$ vanishes and we have $\partial_t w=\partial_t \varphi$. This yields
$$\forall t\geq 0~,~~\int_{\omega} |\partial_t \varphi(x,t)|^2 \dd x \leq t\frac{\varepsilon_n^2}2  e^{(\|q\|_{L^\infty(\Gc)}+2)t}.$$
Thus, for any $C>0$ and $T>0$, we can choose $n$ large enough to obtain a solution of the free wave equation \eqref{eq_dual} satisfying \eqref{eq_prop_obs_wave_bis}.
\end{proof}

To disprove \eqref{eq_obs} when (GGCC) is not satisfied, one can also rely on the well-known equivalence between the observability of the wave equation and a resolvent estimate—also known as the Hautus test, see for instance \cite[Proposition 4.5]{RTTM05}. In our context, this equivalence can be reformulated as follows: there exists a constant $C>0$ such that for every $\varphi \in D(\Delta_{\Gc})$ and every $\lambda > 0$,
\begin{equation}
\label{eq:resolventwave}
\left\|\lambda \varphi\right\|^2_{L^2(\Gc)} \leq C \left( \|(\Delta_\Gc + q + \lambda^2) \varphi\|^2_{L^2(\Gc)} + \|\lambda 1_\omega \varphi\|^2_{L^2(\Gc)}\right).
\end{equation}

We now proceed to show that \eqref{eq:resolventwave} does not hold. The proof follows by contradiction: suppose that \eqref{eq:resolventwave} is indeed satisfied.
Consider the sequence $(\varphi_n) = (u_n),\ (\lambda_n) = (\mu_n)$ where $(u_n)$ and $(\mu_n)$ are given by Proposition \ref{prop_quasimode}. Evaluating the right-hand side of \eqref{eq:resolventwave}, we obtain
\begin{equation*}
     \|(\Delta_\Gc + q + \mu_n^2) \, u_n\|^2_{L^2(\Gc)} + \|\mu_n 1_\omega u_n\|^2_{L^2(\Gc)} = 
		\|(\Delta_\Gc + q + \mu_n^2) \, u_n\|^2_{L^2(\Gc)} \xrightarrow[n \longrightarrow + \infty]{}0.
\end{equation*}
 This implies that $\left\|\mu_n u_n\right\|_{L^2(\Gc)}$ tends to zero $0$ as $n \longrightarrow + \infty$. Next, using the identity
\begin{equation*}
 - \left\|\nabla u_n\right\|^2_{L^2(\Gc)} + \left\langle q u_n,u_n \right\rangle_{L^2(\Gc)} + \left\|\mu_n u_n\right\|^2_{L^2(\Gc)} = \left\langle (\Delta_\Gc + q + \mu_n^2) \, u_n,u_n\right\rangle_{L^2(\Gc)} \xrightarrow[n \longrightarrow + \infty]{}0,
\end{equation*}
we deduce that $\left\|\nabla u_n\right\|^2_{L^2(\Gc)} \longrightarrow 0$ as $n$ goes to infinity. However, this contradicts the assumption
$\left\|\nabla u_n\right\|_{L^2(\Gc)} = 1$. Hence \eqref{eq:resolventwave} cannot hold.

\subsection{The observability inequality under (GGCC)}\label{suff}

In this part, we prove that (GGCC) is a sufficient condition for the observability inequality \eqref{eq_obs} of the wave equation \eqref{eq_dual}. Note that, in this subsection, unlike the results above, the choice of boundary conditions is free.

\begin{prop}\label{prop_obs_wave}
Let $\Gc$ be a graph and $\omega$ an open set of $\Gc$. We assume that:
\begin{itemize}
\item the Graph Geometric Control Condition holds for $\omega$,
\item the kernel of $(\Delta_\Gc +q)$ is reduced to $\{0\}$, that is that the unique stationary solution of \eqref{eq_free_wave} is the trivial solution $\varphi\equiv 0$.
\end{itemize}
Let $T$ be a time larger than the length $L$ appearing in Definition \ref{defi_GGCC} of (GGCC). Then, there exists a positive constant $C$ such that, for all solution $\varphi$ of class $ \Cc^1([0,T],L^2(\Gc))$\linebreak $\cap \Cc^0([0,T],H^1_\Delta(\Gc))$ of \begin{equation}\label{eq_free_wave}
\left\{\begin{array}{ll} \partial_{t}^2 \varphi = \Delta_\Gc \varphi +q(x)\varphi, & t>0,\\
(\varphi,\partial_t\varphi)(t=0)=(\varphi_0,\varphi_1)\in H^1_\Delta(\Gc)\times L^2(\Gc), & 
\end{array}\right. 
\end{equation}
we have
\begin{equation}\label{eq_prop_obs_wave}
\|(\varphi_0,\varphi_1)\|_{H^1_\Delta(\Gc)\times L^2(\Gc)}^2 \leq C \int_0^T \int_{\omega} |\partial_t \varphi(x,t)|^2 \dd x \dd t.
\end{equation}
\end{prop}

Before proceeding with the proof of Proposition \ref{prop_obs_wave}, we can finally prove Theorems \ref{th_exact_wave_intro_dir} and \ref{th_exact_wave_intro_neu}.
\begin{proof}[Proof of Theorems \ref{th_exact_wave_intro_dir} and \ref{th_exact_wave_intro_neu}]  
Proposition \ref{prop_obs_wave} ensures that (GGCC) is a sufficient condition for the observability of the wave equation \eqref{eq_dual}. We also know that, thanks to Corollary \ref{coro_obs_free_wave}, (GGCC) is a necessary condition for the observation in the case of Dirichlet boundary conditions. As stated in Theorem \ref{th:obsexactwave}, observation and control are equivalent properties: (GGCC) is then a necessary and sufficient condition for the exact controllability of the wave equation with Dirichlet boundary condition, concluding the proof of Theorem \ref{th_exact_wave_intro_dir}, and  a sufficient condition for the exact controllability of the wave equation in general, concluding the proof of Theorem \ref{th_exact_wave_intro_neu}.  
\end{proof}

The remaining part of this section is the proof of Proposition \ref{prop_obs_wave}, which is split in several steps. The main arguments are the following.

\begin{itemize}
\item The use of multiplier techniques developed to control the wave equation on a segment. This kind of methods has been introduced by Morawetz in \cite{Morawetz} and popularized by Lions \cite{Lions}. The multiplier techniques have been widely used, in particular in one-dimensional problems where they provide explicit and optimal results. Notice that, in our setting, the solution defined on the graph does not satisfy an independent boundary condition at the interior vertices. This makes the computations more involved than in a segment where we can use that the solution vanishes at the ends of the segment (or another boundary condition).
We underline that \cite{Schmidt} also uses multipliers to control a network of vibrating strings. These are time-independent multipliers, which are simpler, but it is difficult to use them for estimating the optimal time of control. Also notice that it should be possible to obtain similar results by using d'Alembert's formula as in \cite{AZ21,DZ06}. However, the multiplier techniques provide a different strategy, which somehow more general. 
For example, we should be able to consider a time-dependent potential $q(x,t)$ in \eqref{eq_free_wave} without too much complications. It should also be possible to consider a Laplacian operator with non-constant coefficients, but we will not do this here for sake of simplicity.  
\item A strategy to propagate the observation in the graph to unobserved edges. The idea of the strategy is mostly the same as the one of \cite{AB} and uses the tree structure of the unobserved parts of the graph. The same strategy is also used via d'Alembert's formula in \cite{DZ06}. There, the case of boundary control in star-shaped networks is considered but it is noticed that it can be extended to trees as in the present paper.
\item Common tricks to get rid of compact terms in the observation estimate. In this step, we need to show that the set of ``invisible solutions", the ones being $0$ in $\omega$, is reduced to $\{0\}$. The proofs differ from one article to another, but the compactness arguments behind are always the same and go back to the first papers on the subject, see \cite{BLR_JEDP,BLR} for example. It is worth noting that this part of the proof should be adapted if instead of restricting to spatial-dependent potential $q(x)$, we would consider time-dependent potential  $q(x,t)$.
\end{itemize}

\vspace{3mm}

{\noindent\bf Step 1: simplification of the setting.} Since $T$ is larger than $L$, $\Gc$ is compact, and $\omega$ is open, we may assume the following setting preserving the validity of the proposition. First, we replace $\omega$ by a union of a finite number of intervals. This may slightly increase the length $L$, but we can assume that $T>L$ still holds. Then, we can add an artificial vertex at each end of the intervals composing $\omega$. Replacing an edge $(0,\ell)$ by two edges $(0,a)\cup (a,\ell)$ does not change the geometry of the graph, neither the equation. Indeed, the conditions of continuity of value and flux at the vertex imply that $H^k((0,\ell))$ and $H^k((0,a)\cup (a,\ell))$ are equivalent for $k=0$, $1$ and $2$ and thus that the Laplacian operator is unchanged by this transformation. To summarize, we consider now a graph in which any edge is either an edge that does not contain any part of $\omega$, or an edge entirely covered by $\omega$.

\medskip

To be able to give a strong sense to all the following computations, it is convenient to first consider regular solutions with $(\varphi_0,\varphi_1)$ in $ D(\Delta_\Gc)\times H^1_\Delta(\Gc)$. This will be implicitly assumed below. Once the estimate \eqref{eq_prop_obs_wave} will be proved for such regular solutions, it directly extends to all $(\varphi_0,\varphi_1)\in H^1_\Delta(\Gc)\times L^2(\Gc)$ by density.

\vspace{3mm}

{\noindent\bf Step 2: weak observation on a uncontrolled edge.}
In this step, we focus on the observation of a single edge. To simplify the notation, we call this edge $e$ and assimilate it to $(0,\ell)$. We consider a solution of the free wave equation
\begin{equation}\label{eq_free_wave_2}
\forall t\in\RR~,~~\forall x\in(0,\ell)~,~~\partial_{t}^2 \varphi(x,t)=\partial_{x}^2\varphi(x,t) + q(x) \varphi(x,t).
\end{equation}
Notice that we do not specify the boundary conditions since we cannot precise them without knowing what is happening elsewhere on the graph. The central argument of this part is the use of multipliers. Let $\rho$ be a function which is $\Cc^0$ and piecewise $\Cc^1$, from $[0,\ell]\times [t_1,t_2]$ into $\RR$. Assume to simplify that $\rho(\cdot,t_1)\equiv \rho(\cdot,t_2)\equiv 0$. Multiplying the equation \eqref{eq_free_wave_2} by $\rho\partial_x\varphi$, we obtain
\begin{align}
0&=\int_{t_1}^{t_2}\int_0^\ell (\partial_{t}^2 \varphi - \partial_{x}^2\varphi -q \varphi)\rho \partial_x \varphi \,\dd x\dd t =\int_{t_1}^{t_2}\int_0^\ell \rho\big(\partial_{t}^2 \varphi\partial_x \varphi - \partial_{x}^2\varphi \partial_x \varphi -q \varphi\partial_x \varphi \big) \,\dd x\dd t \nonumber \\
&=  \int_{t_1}^{t_2}\int_0^\ell \rho \big(-\partial_t\varphi \, \partial_{t} \partial_{x} \varphi - \partial_{x}^2\varphi\partial_x \varphi \big) \,\dd x\dd t - \int_{t_1}^{t_2}\int_0^\ell \partial_t \rho \partial_t\varphi \partial_{x}\varphi + \rho q \varphi \partial_x \varphi \,\dd x\dd t\nonumber\\
&=\frac 12 \int_{t_1}^{t_2}\int_0^\ell (\partial_x \rho) \big(|\partial_t \varphi|^2 + |\partial_x \varphi|^2 \big)\,\dd x\dd t - \int_{t_1}^{t_2}\int_0^\ell \partial_t \rho \partial_t\varphi \partial_{x}\varphi + \rho  q \varphi \partial_x \varphi\,\dd x\dd t\nonumber\\
&~~~~    - \frac 12\int_{t_1}^{t_2} \Big[ \rho \big(|\partial_t \varphi|^2 + |\partial_x \varphi|^2 \big) \Big]_0^\ell \,\dd t.\label{eq_multiplier}
\end{align}
Assume that we control the quantity $|\partial_t \varphi|^2 + |\partial_x \varphi|^2$ at one of the ends of the edge, say at $x=0$ to fix the notation. We choose $\rho$ as 
\begin{equation}\label{eq_def_rho}
\rho(x,t):= \min\big(0,h(t)-x\big)
\end{equation}
with $h$ a smooth non-negative function satisfying $h(t_1)=h(t_2)=0$, so that $\rho$ is continuous and piecewise $\Cc^1$ and $\rho(\cdot,t_1)\equiv\rho(\cdot,t_2)\equiv 0$.  Applying \eqref{eq_multiplier} to a such choice, we obtain
\begin{align}
\int_{t_1}^{t_2} \rho(0,t) \big(|\partial_t \varphi|^2 + |\partial_x \varphi|^2\big)(0,t) \,\dd t ~= ~ & \int_{t_1}^{t_2}\int_0^{\min(h(t),\ell)} \big(|\partial_t \varphi|^2 + |\partial_x \varphi|^2\big) \,\dd x\dd t\nonumber \\ 
&+ \int_{t_1}^{t_2} \rho(\ell,t) \big(|\partial_t \varphi|^2 + |\partial_x \varphi|^2\big)(\ell,t) \dd t \label{eq_multiplier_20}\\
& + 2 \int_{t_1}^{t_2}\int_0^{\min(h(t),\ell)} h'(t) \partial_t\varphi \partial_{x}\varphi + \rho q \varphi \partial_x \varphi  \,\dd x\dd t.\nonumber
\end{align}
For any $\nu>0$, to be chosen small enough later, we have the bound
$$ \left|\int_{t_1}^{t_2}\int_0^{\min(h(t),\ell)} \rho q \varphi \partial_x \varphi  \,\dd x\dd t \right| \leq \int_{t_1}^{t_2}\left(\nu\|\partial_x \varphi\|^2_{L^2(0,\ell)} + \frac 1\nu \|q\|_{L^\infty}^2\|\rho\|_{L^\infty}^2 \|\varphi\|^2_{L^2(0,\ell)}\right)\,\dd t.$$
We also have 
$$2\left|\int_{t_1}^{t_2}\int_0^{\min(h(t),\ell)}  h'(t) \partial_t\varphi \partial_{x}\varphi  \,\dd x\dd t \right| \leq {\max |h'(t)|} \int_{t_1}^{t_2}\int_0^{\min(h(t),\ell)} \big(|\partial_t \varphi|^2 + |\partial_x \varphi|^2  \big) \,\dd x\dd t.$$
Thus, assuming $\max |h'(t)|<1$ enables to absorb the previous term in the first term of the right-hand side of \eqref{eq_multiplier_20}. Now, we assume that $t_2-t_1>2\ell$ and choose $\tau>0$ smaller that $\frac 12(t_2-t_1)-\ell$. We can choose $h(t)$ with $|h'(t)|<1$ but close enough to $1$ such that there exists $\varepsilon>0$, depending on $\tau$, such that $h(t)>\ell+\varepsilon$ on $(t_1+\ell+\tau,t_2-\ell-\tau)$. The associated weight $\rho$ is described in Figure \hyperref[fig_rho]{3}. 
\begin{figure}[ht]
\begin{center}
\resizebox{4cm}{!}{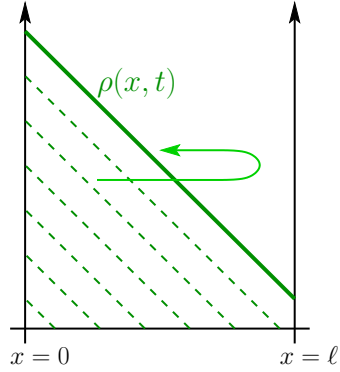}
\end{center}
\caption{\it The weight $\rho$ is a front of slope $1$ and moving at speed slightly less than $1$. It moves back and forth from $0$ to a weight being positive in all the edge.}\label{fig_rho}
\end{figure}
Using this weight, \eqref{eq_multiplier_20} provides a positive constant $C>0$, independent of $\varphi$ and $\nu$, but depending on the geometry of the graph, on $q$ and on the times $t_i$, such that
\begin{align}
 &\int_{t_1+\ell+\tau}^{t_2-\ell-\tau}\int_0^\ell \big(|\partial_t \varphi|^2 + |\partial_x \varphi|^2 \big)(x,t)\,\dd x\dd t + \int_{t_1+\ell+\tau}^{t_2-\ell-\tau} \big(|\partial_t \varphi|^2 + |\partial_x \varphi|^2 \big)(\ell,t)\,\dd t\nonumber \\
& \leq C \int_{t_1}^{t_2}\Big( \big(|\partial_t \varphi|^2 + |\partial_x \varphi|^2 \big)(0,t)  + \frac 1\nu \|\varphi(t)\|^2_{L^2(0,\ell)}  + \nu\|\partial_x \varphi(t)\|_{L^2(0,\ell)}^2\Big) \dd t \label{eq_multiplier_2}.
\end{align}
The last two terms in the upper bound can be seen as parasite terms. This is the reason why we speak for the moment of a ``weak'' observation. We will see how to get rid of them below. To summarize, up to these parasite terms, we can control the quantity $\big(|\partial_t \varphi|^2 + |\partial_x \varphi|^2 \big)$ in the whole edge and at one of its end in some time interval $[t'_1,t'_2]$, here equal to $[t_1+\ell+\tau,t_2-\ell-\tau]$, by knowing this quantity on the other end and during an interval of time $[t_1,t_2]=[t_1'-\ell-\tau, t_2'+\ell+\tau]$ with margins slightly larger than $\ell$ before and after $[t'_1,t'_2]$. 

\vspace{3mm}

{\noindent\bf Step 3: weak observation on a controlled edge.}
In this step, we assume that we are in an ``controlled'' edge, meaning that we aim at controlling the quantities inside the edge by the observation $|\partial_t\varphi|^2$, up to some parasite terms.

\medskip

We consider a weight $\rho=h(t)\big(x-\frac \ell 2\big)$ where $h(t)$ is a smooth non-negative function with $h(t_1)=h(t_2)=0$. Applying \eqref{eq_multiplier} to this choice, we obtain
\begin{align*}
\frac 12 \int_{t_1}^{t_2} \int_0^\ell & h(t) \big(|\partial_t \varphi|^2 + |\partial_x \varphi|^2 \big) \,\dd x\dd t ~=~\int_{t_1}^{t_2} \int_0^\ell h'(t)\Big(x-\frac \ell 2\Big)  \partial_t\varphi\partial_x\varphi \,\dd x\dd t \\ & + \int_{t_1}^{t_2} \int_0^\ell h(t)\Big(x-\frac \ell 2\Big) q(x) \varphi\partial_x\varphi \,\dd x\dd t \\ & + \frac 14 \int_{t_1}^{t_2}  h(t)\ell \Big(\big(|\partial_t \varphi|^2 + |\partial_x \varphi|^2 \big)(0,t)+\big(|\partial_t \varphi|^2 + |\partial_x \varphi|^2 \big)(\ell,t)\Big) \dd t . 
\end{align*}
Thus, there exists a constant $C>0$ such that, for any $\nu>0$ to be fixed later, 
\begin{align}
\int_{t_1}^{t_2}  h(t)\Big(\big(|\partial_t \varphi|^2 +& |\partial_x \varphi|^2 \big)(0,t)+\big(|\partial_t \varphi|^2 + |\partial_x \varphi|^2 \big)(\ell,t)\Big) \dd t  \nonumber \\ 
&\leq C \Big( \int_{t_1}^{t_2} \int_0^\ell  h(t) \big(|\partial_t \varphi|^2 + |\partial_x \varphi|^2) \,\dd x\dd t\nonumber\\ &~~~~~~~ + \frac 1\nu \int_{t_1}^{t_2}\int_0^\ell \big(|\partial_t \varphi|^2 + |\varphi|^2\big)\,\dd x\dd t + \nu  \int_{t_1}^{t_2}\int_0^\ell |\partial_x \varphi|^2 \,\dd x\dd t \Big) \label{eq_multiplier_3}.
\end{align}
Next, we use a new calculus: we multiply the free wave equation \eqref{eq_free_wave_2} by $h(t)\varphi$ and integrate to obtain 
\begin{align*}
0=\int_{t_1}^{t_2}\int_0^\ell (\partial_{t}^2 \varphi - \partial_{x}^2\varphi - q \varphi)h(t)& \varphi \,\dd x\dd t =\int_{t_1}^{t_2}\int_0^\ell h(t)\big( \varphi\partial_{t}^2 \varphi - \varphi \partial_{x}^2\varphi   - q |\varphi|^2 \big) \,\dd x\dd t \\
&=  \int_{t_1}^{t_2}\int_0^\ell h(t) \big(-|\partial_t\varphi|^2 + | \partial_{x}\varphi|^2 -q |\varphi|^2 \big) \,\dd x\dd t \\ & ~~~~~ - \int_{t_1}^{t_2}\int_0^\ell h'(t) \varphi \partial_t\varphi \,\dd x\dd t  - \int_{t_1}^{t_2} \Big[h(t) \varphi \partial_x \varphi \Big]_0^\ell \, \dd t.
\end{align*}
Thus, we have 
\begin{align}
\int_{t_1}^{t_2}\int_0^\ell h(t)| \partial_{x}\varphi|^2 \,\dd x\dd t \leq C \Big( & \int_{t_1}^{t_2}\int_0^\ell |\partial_t\varphi|^2 \,\dd x\dd t + \|\varphi\|_{L^2(0,\ell)}^2  \nonumber \\ & +  \int_{t_1}^{t_2}h(t)\big(|\varphi \partial_x\varphi|(0,t) + |\varphi \partial_x\varphi|^2(\ell,t) \big) \, \dd t\Big).\label{eq_multiplier_encore}
\end{align}
We bound the last terms of the previous estimate as follows. Consider the end $x=0$, the bound at the end $x=\ell$ being similar. From fractional Sobolev embeddings and Gagliardo-Nirenberg interpolation inequalities, we have, for any positive $\nu'$ and $\nu''$,
\begin{align*}
|\varphi \partial_x\varphi|(0,t)&\leq \frac {1}{2\nu'} |\varphi|^2(0,t) + \frac {\nu'}2 |\partial_x\varphi|^2(0,t)\\
&\leq \frac {C}{\nu'} \|\varphi(t)\|^2_{H^{3/4}(0,\ell)} + \frac {\nu'}2 |\partial_x\varphi|^2(0,t)\\
&\leq \frac {C}{{\nu'}^2\nu''} \|\varphi(t)\|^2_{L^2(0,\ell)} + \nu'' \int_0^\ell |\partial_x \varphi(x,t)|^2 \,\dd x + \frac {\nu'}2 |\partial_x\varphi|^2(0,t).
\end{align*}
We can go back to \eqref{eq_multiplier_encore} and choose $\nu''>0$ small enough to absorb the above term $\int_0^\ell |\partial_x \varphi|^2$ in the left-hand side of \eqref{eq_multiplier_encore}. We obtain the estimate
\begin{align}
\int_{t_1}^{t_2}\int_0^\ell h(t)| \partial_{x}\varphi|^2 \,\dd x\dd t \leq C \Big( & \int_{t_1}^{t_2}\int_0^\ell |\partial_t\varphi|^2 \,\dd x\dd t + \left(1+\frac 1{\nu'^2}\right)\|\varphi\|_{L^2(0,\ell)}^2  \nonumber \\ & + \nu' \int_{t_1}^{t_2}h(t)\big(|\partial_x\varphi|^2(0,t) + |\partial_x\varphi|^2(\ell,t) \big) \, \dd t\Big) \label{eq_multiplier_4}.
\end{align}
We use \eqref{eq_multiplier_4} to bound the term $\int \int h(t)| \partial_{x}\varphi|^2$ in the right-hand side of \eqref{eq_multiplier_3} and we choose $\nu'$ small enough to absorb the boundary terms of \eqref{eq_multiplier_4}  in the left side of \eqref{eq_multiplier_3}. We obtain that there exists a constant $C>0$ such that, for all $\nu>0$, we have 
\begin{align*}
\int_{t_1}^{t_2}  h(t)\Big(&\big(|\partial_t \varphi|^2 + |\partial_x \varphi|^2 \big)(0,t)+\big(|\partial_t \varphi|^2 + |\partial_x \varphi|^2 \big)(\ell,t)\Big) \, \dd t \\ 
&\leq C \int_{t_1}^{t_2}\left(\left(1+\frac 1\nu \right) \left(\|\varphi(t)\|^2_{L^2(0,\ell)}  +  \int_0^\ell |\partial_t \varphi (t)|^2 \dd x \right) + \nu  \int_0^\ell |\partial_x \varphi(t)|^2 \dd x \right) \dd t. 
\end{align*}

Using this estimate back to \eqref{eq_multiplier_4} with $\nu'=1$, we also obtain a control of the term $\int \int h(t)| \partial_{x}\varphi|^2$. Thus, for any small $\tau>0$, if we take $h$ uniformly positive in $[t_1+\tau,t_2-\tau]$, we obtain the weak observation 
\begin{align}
\int_{t_1+\tau}^{t_2-\tau}  \Big(&\big(|\partial_t \varphi|^2 + |\partial_x \varphi|^2 \big)(0,t)+\big(|\partial_t \varphi|^2 + |\partial_x \varphi|^2 \big)(\ell,t)\Big)\, \dd t + \int_{t_1+\tau}^{t_2-\tau}\int_0^\ell | \partial_{x}\varphi|^2 \,\dd x\dd t \notag\\  
&\leq C \left(1+\frac 1\nu \right) \int_{t_1}^{t_2} \|\varphi(t)\|^2_{L^2(0,\ell)}\dd t  + C \nu \int_{t_1}^{t_2} \int_0^\ell |\partial_x \varphi(t)|^2 \dd x \dd t \notag\\
&  \qquad + C \left(1+\frac 1\nu \right) \int_{t_1}^{t_2}  \int_0^\ell |\partial_t \varphi (t)|^2 \dd x  \Big) \dd t. \label{eq_multiplier_6}
\end{align}
The first two terms in the right hand side of \eqref{eq_multiplier_6} can be seen as parasite terms. So the last estimate \eqref{eq_multiplier_6} states that, in a controlled edge, one can control $\big(|\partial_t \varphi|^2 + |\partial_x \varphi|^2 \big)$ at both ends and in the whole edge in an interval of time $[t_1',t_2'] = [t_1+\tau, t_2-\tau]$ by knowing the quantity $|\partial_t \varphi|^2$ in the whole edge in a slightly larger interval of time $[t_1,t_2]$.

\medskip

{\noindent\bf Step 4: using (GGCC).}
Here, we intend to use the geometry of the graph and the set $\omega$ to prove the following estimate. 
\begin{lemma}[Weak observation]\label{lemma_obs_wave_1}
Let $\tau_1<\tau_2$ and let $T$ be as in Proposition \ref{prop_obs_wave}. Then, there exists a constant $C>0$ such that, for all $\nu>0$, we have 
\begin{align}
\int_{\tau_1}^{\tau_2}\|(\varphi,\partial_t \varphi)(t)\|_{H^1_\Delta(\Gc)\times L^2 (\Gc)}^2 \,\dd t ~\leq~ & C \int_{\tau_1-T/2}^{\tau_2+T/2}   \left(\left(1+\frac 1\nu\right)\|\varphi(t)\|_{L^2(\Gc)}^2 + \nu\|\grad \varphi(t)\|_{L^2(\Gc)}^2\right.  \nonumber\\ 
&~~+\left. \left(1+\frac 1\nu\right) \int_{\omega} |\partial_t \varphi(x,t)|^2 \dd x \right) \dd t.\label{eq_lemma_obs_wave_1}
\end{align}
\end{lemma}
\begin{proof}
Since the right side of \eqref{eq_lemma_obs_wave_1} contains the $L^2-$norm of $\varphi$, our goal is to control the quantity $|\partial_t\varphi|^2+|\partial_x\varphi|^2$ in all the edges during the time interval $[\tau_1,\tau_2]$. Also, the right side of \eqref{eq_lemma_obs_wave_1} contains the $L^2-$norm of $\partial_t\varphi$ on the controlled edges and it is sufficient to control $|\partial_x\varphi|^2$ there. Due to \eqref{eq_multiplier_6}, this control is possible in the controlled edges composing $\omega$, up to margin of time as small as wanted. 

\medskip

To control the other edges, we use the equivalence between (GGCC) and the fact that the family of uncontrolled edges is a forest, in which each tree contains at most one exterior edge of $\Gc$, see Proposition \ref{prop_GGCC}. To avoid heavy general notations, let us explain the proof in an example. We consider the graph $\Gc$ of Figure \hyperref[free_wave]{4}.

\medskip

\begin{figure}[ht]
\begin{center}
\resizebox{\textwidth}{!}{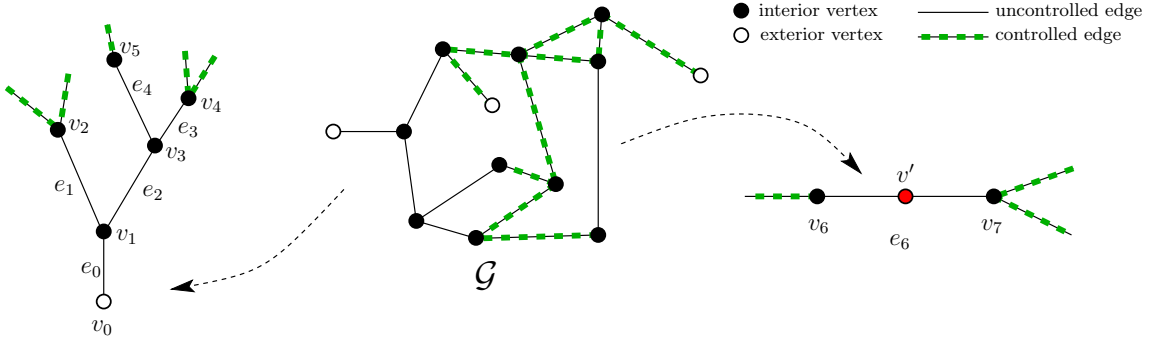}
\end{center}
\caption{\it A graph satisfying (GGCC). The uncontrolled edges form two trees. The left one has one exterior edge, which is considered as the root of the tree. The right one is a single edge, in the middle of which we add an artificial vertex $v'$. The control of the quantity $|\partial_t\varphi|^2+|\partial_x\varphi|^2$ propagates from the controlled edge inside the uncontrolled trees at speed almost $1$. In the left tree, once the control is done in all the edges above a vertex, it propagates into the edge below via the transmission conditions.}
\label{free_wave}
\end{figure}

First consider the left tree of Figure \hyperref[free_wave]{4}. It contains an exterior edge $v_0$ and we consider it as the root of the tree. Consider the edge $e_0$, assimilated to $(0,\ell_0)$, attached to this root, assimilated to the end $\ell_0$. Using the control \eqref{eq_multiplier_2}, we can control the quantity $|\partial_t\varphi|^2+|\partial_x\varphi|^2$ inside this edge if we control this quantity at $x=0$ in an interval slightly larger than $[\tau_1-\ell_0,\tau_2+\ell_0]$. 
Note here that the boundary condition imposed at $\ell_0$ has no importance. 
Also notice that the norm $\|\varphi\|_{L^2}^2$ also appears at the right side of \eqref{eq_lemma_obs_wave_1}, so this term can be used freely in our estimates. From \eqref{eq_multiplier_2}, to control $|\partial_t\varphi|^2+|\partial_x\varphi|^2$ at $x=0$ means to control it at the vertex $v_1$, for the edge $e_0$. Due to the transmission conditions at the vertex $v_1$, it is sufficient to control the value $\varphi$ and all the incoming fluxes in the other edges. Thus, it is sufficient to control $|\partial_t\varphi|^2+|\partial_x\varphi|^2$ at $v_1$ for the edges $e_1$ and $e_2$ during an interval of times slightly larger than $[\tau_1-\ell_0,\tau_2+\ell_0]$. Again, we can use the control \eqref{eq_multiplier_2} to control the quantity $|\partial_t\varphi|^2+|\partial_x\varphi|^2$ inside the edge $e_1$ of length $\ell_1$, if we control it at the vertex $v_2$ during an interval of times slightly larger than $[\tau_1-\ell_0-\ell_1,\tau_2+\ell_0+\ell_1]$; and inside the edge $e_2$ of length $\ell_2$, if we control it at the vertex $v_3$ during an interval of times slightly larger than $[\tau_1-\ell_0-\ell_2,\tau_2+\ell_0+\ell_2]$. We iterate this method until reaching vertices for which the above edges are all controlled edges in which we can control the $|\partial_t\varphi|^2+|\partial_x\varphi|^2$ up to a margin of time as small as needed. Notice that this process always ends at such vertices due to the structure of tree of the connected set of uncontrolled edges. Finally notice that the final control has to be done during an interval of times slightly larger than $[\tau_1-L',\tau_2+L']$ where $L'$ is the largest distance between the root $v_0$ and the top vertices (say $v_5$ here). This length $L'$ is exactly equal to $L/2$ where $L$ is the length of the longest path inside this uncontrolled part, here the one starting at $v_5$, going down to $v_0$ and back up to $v_5$.

\medskip

Now, we consider a tree consisting of uncontrolled edges and without any exterior edges, as the right one of Figure \hyperref[free_wave]{4}. We first add an artificial vertex $v'$ exactly at the middle of the tree: the longest path starting at $v'$ and going to the right has the same length as the one of the longest path starting at $v'$ and going to the left. Next, we argue exactly as above in each half tree, choosing $v'$ as the root. We obtain the control of the quantity $|\partial_t\varphi|^2+|\partial_x\varphi|^2$ inside the whole tree during a time $[\tau_1,\tau_2]$, up to control it in $\omega$ during an interval of times slightly larger than $[\tau_1-L',\tau_2+L']$ where $L'$ is the largest distance between the middle root $v'$ and the leaves. It remains to notice that $L'$ is exactly equal to $L/2$ where $L$ is the length of the longest path inside this uncontrolled part, one going from the furthest vertex at the left of $v'$, to the furthest vertex at its right.

\medskip

Performing this method in all the uncontrolled trees and adding all the estimates leads to the proof of Lemma \ref{lemma_obs_wave_1}, for any time $T$ strictly larger than the length of the longest uncontrolled path.  
\end{proof}

\medskip

{\noindent\bf Step 5: using the conservation of the energy.}
We recall the following well-known computation. On a edge $e$, we multiply the free wave equation \eqref{eq_free_wave_2} by $\partial_t \varphi$ and integrate:
\begin{align*}
 0&=\int_{t_1}^{t_2}\int_0^\ell (\partial_{t}^2 \varphi - \partial_{x}^2\varphi -q(x) \varphi)\partial_t \varphi \,\dd x\dd t\\ & = \int_{t_1}^{t_2}\int_0^\ell \left(\partial_{t}^2 \varphi \partial_t \varphi + \partial_{x}\varphi \partial_{x}\partial_{t} \varphi -q(x) \varphi\partial_t \varphi \right) \,\dd x\dd t  - \int_{t_1}^{t_2} \Big[\partial_x\varphi\partial_t\varphi\Big]_0^\ell \,\dd t\\
 &=- \frac 12 \Big[ \int_0^\ell \left(|\partial_t\varphi|^2+|\partial_x\varphi|^2 -q(x) |\varphi|^2\right) \,\dd x\Big]_{t_1}^{t_2}  - \int_{t_1}^{t_2} \Big[\partial_x\varphi\partial_t\varphi\Big]_0^\ell \,\dd t.
\end{align*}
Next, we add all these equalities on all the edges. Due to the conditions at the vertices, we obtain that 
$$E(t_1)=E(t_2)~~\text{ where }~~E(t):=\frac 12 \int_{\Gc} \left(|\partial_t\varphi|^2+|\partial_x\varphi|^2 -q(x) |\varphi|^2\right) \,\dd x.$$
\begin{lemma}[Weak observation bis]\label{lemma_obs_wave_2}
Let $T$ be as in Proposition \ref{prop_obs_wave}. Then, there exists a constant $C>0$ such that \begin{equation}\label{eq_lemma_obs_wave_2}
\|(\varphi_0,\varphi_1)\|_{H^1_\Delta(\Gc)\times L^2(\Gc)}^2 ~\leq~ C \left(\int_0^T \int_{\omega} |\partial_t \varphi(x,t)|^2 \dd x \dd t + \int_0^T \|\varphi(t)\|_{L^2(\Gc)}^2\dd t\right).
\end{equation}
\end{lemma}
\begin{proof}
Let $T$ be strictly larger than the length of the longest uncontrolled path. Choose $T'$ still satisfying this property but with $T'<T$. Set $\tau=T-T'$, $\tau_1=(T-\tau)/2$ and $\tau_2=(T+\tau)/2$. We apply Lemma \ref{lemma_obs_wave_1} to the interval $[\tau_1-T'/2,\tau_2+T'/2]$, which is exactly $[0,T]$. We obtain the estimate \eqref{eq_lemma_obs_wave_1}.

\medskip

Due to the conservation of energy, the left quantity of \eqref{eq_lemma_obs_wave_1} in Lemma \ref{lemma_obs_wave_1} can be replaced by $(\tau_2-\tau_1)\|(\varphi_0,\varphi_1)\|_{H^1_\Delta(\Gc)\times L^2}^2$ up to increasing the coefficient of $\|\varphi(t)\|_{L^2(\Gc)}$ in the upper bound. Also notice that the quantity $\nu\|\grad \varphi(t)\|_{L^2}^2$ in the right side of \eqref{eq_lemma_obs_wave_1} can be bounded by $\nu \|(\varphi_0,\varphi_1)\|_{H^1_\Delta(\Gc)\times L^2}$. Thus, taking $\nu$ sufficiently small, we deduce the estimate \eqref{eq_lemma_obs_wave_2}.
\end{proof}

\vspace{3mm}

{\noindent\bf Step 6: from weak to strong observation.}
To prove \eqref{eq_prop_obs_wave}, it remains to get rid of the $L^2$-norm in the right side of \eqref{eq_lemma_obs_wave_2}. The usual trick is a compactness argument using the fact that this norm is of lower order than the other terms of the inequality and that the set of ``invisible solutions" is reduced to $\{0\}$. These arguments can be found  for example in \cite{BLR_JEDP,BLR,BZ}.

\medskip

To finish the proof of Proposition \ref{prop_obs_wave}, we argue by contradiction. Assume that \eqref{eq_prop_obs_wave} does not hold: there exists a sequence $(\varphi_n)$ of solutions of the free wave equation with 
\begin{equation}\label{preuve_step6_1}
\|(\varphi_n,\partial_t\varphi_n)(t=0)\|_{H^1_\Delta(\Gc)\times L^2(\Gc)}^2=1~~\text{ and }~~\int_0^T\int_\omega |\partial_t\varphi_n|^2 \,\dd x\dd t\longrightarrow 0. 
\end{equation}
Due to the conservation of energy, the normalisation implies that the sequence $(\varphi_n)$ is bounded in $\Cc^1([0,T],L^2(\Gc))\cap \Cc^0([0,T],H^1_\Delta(\Gc))$. By Ascoli's theorem and using in particular the compact embedding $H^1_\Delta(\Gc) \hookrightarrow L^2(\Gc)$, we can assume that $(\varphi_n)$ converges in $\Cc^0([0,T],L^2(\Gc))$ to some function $\varphi_\infty$. In particular, $(\varphi_n)$ is a Cauchy sequence in $\Cc^0([0,T],L^2(\Gc))$. We apply Lemma \ref{lemma_obs_wave_2} to $\varphi_p-\varphi_q$: 
\begin{align*}
\|(\varphi_p,\partial_t\varphi_p)(0)&-(\varphi_q,\partial_t\varphi_q)(0)\|^2_{H^1_\Delta(\Gc)\times L^2(\Gc)}  \\
&\leq~ C \left(\int_0^T \int_{\omega} |\partial_t (\varphi_p-\varphi_q)(x,t)|^2 \dd x \dd t + \int_0^T \|(\varphi_p-\varphi_q)(t)\|_{L^2(\Gc)}^2\dd t\right).
\end{align*}
For large $p$ and $q$, the above upper bound is as small as needed due to \eqref{preuve_step6_1} and the fact that $(\varphi_n)$ is a Cauchy sequence in $\Cc^0([0,T],L^2(\Gc))$. This shows that $(\varphi_n,\partial_t\varphi_n)(t=0)$ is a Cauchy sequence in $H^1_\Delta(\Gc)\times L^2(\Gc)$ and, by continuity of the solution of the wave equation with respect to the initial data, $(\varphi_n)$ actually converges to $\varphi_\infty$ in  $\Cc^1([0,T],L^2(\Gc))\cap \Cc^0([0,T],H^1_\Delta(\Gc))$. From \eqref{preuve_step6_1}, we obtain that $\varphi_\infty$ is an ``invisible solution" in the sense that $\partial_t \varphi_\infty\equiv 0$ in $\omega$. It remains to show that the only invisible solution is $\varphi\equiv 0$ everywhere. This provides a contradiction with the convergence and the normalisation in \eqref{preuve_step6_1}, concluding the proof.

\begin{lemma}\label{lemma_invisible}
Let $T$ be an observation time as in Proposition \ref{prop_obs_wave}. If $\varphi$ is a solution of \eqref{eq_free_wave} with $\partial_t \varphi \equiv 0$ in $\omega\times [0,T]$, then $\varphi\equiv 0$ everywhere.
\end{lemma}
\begin{proof}
Consider the space $\Nc$ of all the invisible solutions, that is the space of solutions, with the regularity $\Cc^1([0,T],L^2(\Gc))\cap \Cc^0([0,T],H^1_\Delta(\Gc))$, satisfying  $\partial_t \varphi_\infty\equiv 0$ in $\omega\times [0,T]$. It is obviously a linear subspace. Moreover, applying Lemma \ref{lemma_obs_wave_2} to these functions, we have that 
$$\forall \varphi\in \Nc~,~~\|(\varphi,\partial_t\varphi)(t=0)\|_{H^1_\Delta(\Gc)\times L^2(\Gc)}^2 ~\leq~ C  \int_0^T \|\varphi(t)\|_{L^2(\Gc)}^2\dd t.$$
Using exactly the same trick as above, we can show that we can extract a convergent subsequence from any bounded sequence of $\Nc$, showing that $\Nc$ is finite-dimensional.

\medskip

Let $\varphi\in\Nc$ and $\tau>0$. We notice that $q$ is independent of the time, so, taking the time-derivative of the free wave equation, we obtain that $\partial_t\varphi$ also solves a free wave equation but a priori in a space with a lower regularity than $\varphi$. Let us show that, actually, $\partial_t\varphi$ has the suitable regularity to belong to $\Nc$.
We set $\psi_\tau=\frac 1\tau(\varphi(\tau+\cdot)-\varphi(\cdot))$, which is also solution of the free wave equation, but not necessarily in $\Nc$ due to the slight shift in time. 
We apply Lemma \ref{lemma_obs_wave_2} to $\psi_\tau$ to obtain 
$$\|(\psi_\tau,\partial_t \psi_\tau)(t=0)\|_{H^1_\Delta(\Gc)\times L^2}^2 ~\leq~ C \left(\frac 1\tau \int_T^{T+\tau} \int_{\omega} |\partial_t \varphi(x,t)|^2 \dd x \dd t + \int_0^T \|\psi_\tau(t)\|_{L^2(\Gc)}^2\dd t\right).
$$
Since $\varphi$ is differentiable with respect to time in $L^2(\Gc)$, when $\tau\rightarrow 0$, we have that $\psi_\tau$ converges to $\partial_t\varphi$ in $\Cc^0([0,T],L^2(\Gc))$. Thus, the last term of the above estimate is bounded. It is also clear that the first term of the upper bound is bounded when $\tau\rightarrow 0$. This shows that $(\psi_\tau,\partial_t \psi_\tau)(t=0)$ is bounded in ${H^1_\Delta(\Gc)\times L^2(\Gc)}$ when $\tau\rightarrow 0$. Up to extracting a subsequence, it must have a weak limit in this space. To identify this weak limit, we notice that, in any compact subinterval of any edge,  $(\psi_\tau,\partial_t \psi_\tau)(t=0)$ converges to $(\partial_t \varphi,\partial^2_{t}\varphi)(t=0)$ at least in the sense of distributions. Thus, the previous weak limit must be $(\partial_t \varphi,\partial^2_{t}\varphi)(t=0)$, which actually belongs to $H^1_\Delta(\Gc)\times L^2(\Gc)$. In other words, $\partial_t\varphi$ is a solution of the free wave equation with the regularity $\Cc^1([0,T],L^2(\Gc))\cap \Cc^0([0,T],H^1_\Delta(\Gc))$. Since the ``invisibility" is obviously a conserved through the time derivative, we have that $\partial_t \varphi$ also belongs to $\Nc$.

\medskip

We just prove that $\partial_t$ is a linear operator on $\Nc$. If $\Nc\neq \{0\}$, since $\Nc$ is a finite-dimensional space, this operator must have at least one eigenvalue $\lambda$ with a non-zero eigenfunction $\varphi$. It satisfies $\partial_t\varphi(x,t)=\lambda \varphi(x,t)$, implying $\varphi(x,t)=e^{\lambda t}\varphi(x)$ with $\varphi(x):=\varphi(x,0)$. Assume first that $\lambda\neq 0$. Coming back to the wave equation and the invisibility, it comes that 
\begin{equation}\label{eq_invisible}
(\Delta_\Gc+q(x)) \varphi(x)=\lambda^2\varphi(x)~~~\text{ and }~~~\varphi_{|\omega}\equiv 0.
\end{equation} 
We can argue as in Step 4: due to the structure of forest of $\Gc\setminus\omega$, the information that $\varphi\equiv 0$ propagates everywhere. Indeed, if it holds on all the edges attached to a vertex, except one, then it implies $\varphi(0)=\partial_x\varphi(0)=0$ at the end of this remaining edge. Then, the first equation of \eqref{eq_invisible} propagates the information on all the edge since it is a second order equation. Then, we can pass the information to the next vertex and so on. This would yield that $\varphi\equiv 0$, which has been excluded by construction.

\medskip

The only remaining case is $\lambda=0$ that is that $\varphi$ is actually a stationary solution of the free wave equation and the observation of $\partial_t\varphi$ is useless. It remains to remember that the suitable property has been assumed in the statement:  $(\Delta_\Gc+q)\varphi=0$ implies $\varphi\equiv 0$. In other words, there is no such non-zero stationary solution and we meet a final contradiction. This shows that $\Nc=\{0\}$.
\end{proof}

\subsection{Study of a graph including Neumann boundary conditions}\label{sec:neumann}
This whole section is devoted to the proof of Theorem \ref{th_neumann}. Our main ingredients are the geometry of the X graph, in particular its symmetries, and the stabilisation of the wave equation with pointwise damping studied in \cite{AHT}. From this last article, we can extract the following specific case of pointwise observation. 
\begin{prop}[{\cite[Proposition 4.1]{AHT}}]\label{prop_AHT}
Let $T>4$ be fixed. There exists a constant $K>0$ such that any solution $\psi$ 
of 
$$\left\{\begin{array}{ll} \partial_t^2 \psi(x,t)=\partial_x^2\psi(x,t)~~~&x\in (-1,1), t>0\\
\psi(-1,t)=\partial_x \psi (1,t)=0& t>0
\end{array}\right.$$
satisfies
$$\|\psi(t=0)\|_{H^1(-1,1)}^2 + \|\partial_t\psi(t=0)\|_{L^2(-1,1)}^2~\leq ~ K\int_0^T \left|\partial_t \psi(t,0)\right|^2 \dd t .$$ 
\end{prop}
We consider the X graph of Figure \hyperref[fig_ex]{2}, assuming that every edge has length $1$, that $e_1$ and $e_2$ have external vertices endowed with Neumann boundary condition and $e_3$ and $e_4$ have external vertices endowed with Dirichlet boundary condition. Considering the proof of the previous section, we see that it is sufficient to obtain the following weak observation. 
\begin{lemma}\label{lemma-X-wave}
Assume the above geometric setting and take $q=0$. For any $\tau_1<\tau_2$ and any $T>4$, there exists a constant $C>0$ such that, for all $\nu>0$, every solution $\varphi$ of the free wave equation \eqref{eq_free_wave_2} satisfies  
\begin{align}
\int_{\tau_1}^{\tau_2}\|(\varphi,\partial_t \varphi)(t)\|_{H^1_\Delta(\Gc)\times L^2 (\Gc)}^2 \,\dd t ~\leq~ & C \int_{\tau_1-T/2}^{\tau_2+T/2}   \left(\left(1+\frac 1\nu\right)\|\varphi(t)\|_{L^2(\Gc)}^2 + \nu\|\grad \varphi(t)\|_{L^2(\Gc)}^2\right.  \nonumber\\ 
&~~+\left. \left(1+\frac 1\nu\right) \int_{e_1\cup e_3} |\partial_t \varphi(x,t)|^2 \dd x \right) \dd t.\label{eq_lemma_obs_wave_neu}
\end{align}
\end{lemma}
\begin{proof}
We parameterize the four edges of the X graph as $e_1 =e_2= (-1,0)$ and $e_3=e_4 = (0,1)$ and we set $\varphi_j$ the restriction of the solution $\varphi$ to the edge $e_j$. Fix $\tau>0$ such that $T-2\tau>4$. Due to Step 3 of the proof of Proposition \ref{prop_obs_wave} in the previous section, we already know that \eqref{eq_multiplier_6} holds in the controlled edges, which can be translated here as the following control in $e_1\cup e_3$ and on the central vertex $x=0$:
\begin{align}
\int_{\tau_1-T/2+\tau}^{\tau_2+T/2-\tau}  &|\partial_t \varphi|^2 (0,t) \dd t + \int_{\tau_1-T/2+\tau}^{\tau_2+T/2-\tau} \int_{e_1\cup e_3} \big(| \partial_{t}\varphi|^2+| \partial_{x}\varphi|^2+|\varphi|^2\big)(x,t) \,\dd x\dd t \notag\\  
&\leq C \left(1+\frac 1\nu \right) \int_{\tau_1-T/2}^{\tau_2+T/2} \|\varphi(t)\|^2_{L^2(\Gc)} \dd t + \nu  \int_{\tau_1-T/2}^{\tau_2+T/2}\int_{e_1\cup e_3} |\partial_x \varphi(t)|^2 \dd x \dd t \notag\\
&  \qquad + C \left(1+\frac 1\nu \right) \int_{\tau_1-T/2}^{\tau_2+T/2}  \int_{e_1\cup e_3} |\partial_t \varphi (t)|^2 \dd x   \dd t. \label{eq_multiplier_neu}
\end{align}
To control the solution in $e_2\cup e_4$, we use the symmetry of the graph and we merge the left and right parts of the graph by setting  
$$\psi(x,t):=\left\{\begin{array}{ll}\varphi_1(x)+\varphi_2(x) & \text{ if }x\in (-1,0),\\ \varphi_3(x)+\varphi_4(x)& \text{ if }x\in (0,1).\end{array}\right.$$
We can check that $\psi$ belongs to $H^1(-1,1)$ (resp. $H^2(-1,1)$) if $\varphi$ belongs to $H^1(\Gc)$ (resp. $H^2(\Gc)$) due to the continuity and the Kirchhoff's transmission condition at the central vertex $x=0$. Moreover, $\psi$ satisfies Dirichlet boundary condition at $x=-1$, Neumann ones at $x=1$ and $\psi(0,t)=2\varphi(0,t)$. Finally, it is also immediate that $\psi$ is a solution of the free wave equation. Thus, we can apply Proposition  \ref{prop_AHT} and we obtain that, using the conservation of the energy of $\psi$,
\begin{equation}\label{eq_finale}
\int_{\tau_1}^{\tau_2} \|\psi(t)\|_{H^1(-1,1)}^2 + \|\partial_t\psi(t)\|_{L^2(-1,1)}^2 \dd t~\leq ~ K\int_{\tau_1-T/2+\tau}^{\tau_2+T/2-\tau} \left|\partial_t \psi(0,t)\right|^2 \dd t .
\end{equation}
We come back to $\varphi$ by writing
$$ \int_{e_2}|\varphi|^2 = \int_{-1}^0 |\varphi_2|^2 \leq \int_{-1}^0 \big( |\psi| + |\varphi_1|\big)^2\leq 2 \int_{-1}^0 \big(|\psi|^2 + |\varphi_1|^2\big) = 2 \int_{-1}^0 |\psi|^2 + 2 \int_{e_1} |\varphi|^2
$$
and the other estimates are similar. Then, using \eqref{eq_finale} and $\psi(0,t)=2\varphi(0,t)$, we obtain
\begin{align*}
\int_{\tau_1}^{\tau_2}\int_{e_2\cup e_4} \big(| \partial_{t}\varphi|^2+| \partial_{x}\varphi|^2+|\varphi|^2\big)(x,t) &\,\dd x\dd t  ~\leq ~ 8K\int_{\tau_1-T/2+\tau}^{\tau_2+T/2-\tau} \left|\partial_t \varphi(0,t)\right|^2 \dd t\\ & + 2 \int_{\tau_1}^{\tau_2}\int_{e_1\cup e_3} \big(| \partial_{t}\varphi|^2+| \partial_{x}\varphi|^2+|\varphi|^2\big)(x,t) \dd x   \dd t .
\end{align*}
Finally, we combine the above estimate with \eqref{eq_multiplier_neu} to obtain \eqref{eq_lemma_obs_wave_neu}.
\end{proof}
We can finally proceed to the proof of Theorem \ref{th_neumann}.
\begin{proof}[Proof of Theorem \ref{th_neumann}]To conclude the proof, it remains to use Lemma \ref{lemma-X-wave} and argue again as in Steps 5 and 6 of the proof of Proposition \ref{prop_obs_wave} in the previous section.
\end{proof}


\section{Control of the Schrödinger equation}
\label{sec_schr}

In this section, we discuss the exact controllability of the Schrödinger equation in a graph $\Gc$, that is the equation 
\begin{equation}\label{eq_schro_control}
\left\{\begin{array}{ll} i \partial_{t} u = - \Delta_\Gc u-q u + h 1_{\omega} ,& t>0,\\
u(t=0)=u_0\in L^2(\Gc), & 
\end{array}\right. 
\end{equation}
with $q\in L^\infty (\Gc)$ and $\Delta_{\Gc}$ the Laplacian defined on the graph $\Gc$, with either Dirichlet or Neumann conditions at the exterior vertices. In \eqref{eq_schro_control}, at time $t \geq 0$, $u(t,\cdot) : \Gc \to \mathbb{C}$ is the state and $h(t,\cdot) : \Gc\to \mathbb{C}$ is the control.

\subsection{The exact controllability under (GGCC)}

Note that Theorem \ref{prop:exactcontrolschro_intro} in the introduction states that (GGCC) is sufficient to ensure the exact controllability for the Schr\"odinger equation. This fact is a direct consequence of the general principle that the controllability of the wave equation implies the controllability of the Schr\"odinger equation.

\begin{proof}[Proof Theorem \ref{prop:exactcontrolschro_intro}]
Assume that $\Gc$ and $\omega$ satisfy (GGCC). We set $w = e^{-it(\|q\|_{\infty} + 1)} u$ where $u$ is solution of \eqref{eq_schro_control}. The function $w$ solves the new control problem
\begin{equation}\label{eq_free_schro_translat}
\left\{
\begin{array}{ll} i \partial_{t} w = - \Delta_\Gc w - (q(x) -\|q\|_{\infty} - 1) w +  e^{-it(\|q\|_{\infty} + 1)} h 1_{\omega} ,& \ \ \ \ \ t>0,\\
w(t=0)=u_0\in L^2(\Gc). & 
\end{array}\right. 
\end{equation}
Observe that $0$ does not belong to the spectrum of $\Delta_\Gc + q(x) -\|q\|_{\infty} - 1$. So one can apply Theorem \ref{th_exact_wave_intro_neu} to deduce that the associated controlled wave equation is exactly controllable for $T \geq L$, where $L$ is the length appearing in Definition \ref{defi_GGCC}. From \cite[Theorem 3.1]{Mil05}, we then deduce that \eqref{eq_free_schro_translat} is small-time exactly controllable in $L^2(\Gc)$. Finally, the equation \eqref{eq_schro_control} is small-time exactly controllable with $h\in L^2((0,T)\times \omega, \CC) $.
\end{proof}

\subsection{Study of a graph with badly approximable lengths} \label{subsect5.2}

The previous part establishes that (GGCC) is a sufficient condition for the exact controllability of \eqref{eq_schro_control}.  
To obtain examples showing that (GGCC) is not a necessary condition, we prove Theorem \ref{non_ggcc_sch_intro}, which studies the case of the X graph, illustrated in Figure \hyperref[fig_ex]{2}. The proof of this result relies on the symmetries of the X graph and on the following criterion from \cite{Mil05} (see also \cite{BZ}).

\begin{theorem}[{\cite[Theorem 5.1]{Mil05}}]
\label{lem:resolvent}
The equation \eqref{eq_schro_control} is exactly controllable in $L^2(\Gc)$ at some $T>0$ if and only if there exists $C>0$ such that
\begin{equation}
\label{eq:resolventschro}
    \|u\|_{L^2(\Gc)} \leq C (\|\Delta u + \lambda u\|_{L^2(\Gc)} + \|u\|_{L^2(\omega)}) \qquad \forall \lambda \in \mathbb{R},\ \forall u \in D(\Delta_{\Gc}).
\end{equation}
\end{theorem}

We are now ready to provide a sufficient and necessary condition for the controllability of the Schr\"odinger equation on the X graph.

\begin{proof}[Proof of Theorem \ref{non_ggcc_sch_intro}]
Without loss of generality, we can assume that $\ell_{\text{b}}=1$ and set $\ell:=\ell_{\text{t}}$. We parameterize the four edges of the X graph as $e_1 = (-1,0)$, $e_2=(-1,0)$, $e_3=(0,\ell)$ and $e_4 = (0,\ell)$. Due to Theorem \ref{lem:resolvent}, we only have to show that \eqref{eq:resolventschro} holds if and only if $\ell$ is a badly approximable number. First notice that \eqref{eq:resolventschro} trivially holds for $\lambda \leq 0$ since the operator $\Delta_{\Gc}$ is non-positive.

\medskip 

{\noindent \bf The “if” implication.} We argue by contradiction.  Assume that $\ell$ is a badly approximable number and that there exist $(u_n) \in D(\Delta_{\Gc})^{\mathbb{N}}$ and $(\lambda_n) \in [0,+\infty)^{\mathbb{N}}$ such that
\begin{equation}\label{abs}
    \|u_n\|_{L^2(\Gc)} = 1~,~~ \|(\Delta_{\Gc} + \lambda_n) u_n\|_{L^2(\Gc)} \to 0~\text{ and }~ \|u_n\|_{L^2(\omega)} \to 0.
\end{equation}
We first observe that $\lambda_n\rightarrow 0$ is not possible due to the strict negativity of $\Delta_{\Gc}$ and $\|(\Delta_{\Gc} + \lambda_n) u_n\|_{L^2(\Gc)} \to 0$.  

\medskip

{\noindent \bf Step 1: Symmetry decomposition.} For any $u\in D(\Delta_{\Gc})$, we set
$(u^1,u^2,u^3,u^4)$ its restrictions to each edge $e_1$, $e_2$, $e_3$ and $e_4$ and we use the symmetry of the X graph to set
\begin{equation}
\label{eq:deffgschro}
  f(x)  = \frac{u^1(x) - u^2(x)}{2},\ x \in (-1,0), \qquad g(x) = \frac{u^3(x) - u^4(x)}{2},\ x \in (0,\ell),
  \end{equation}
  \begin{equation}
  \label{eq:defhschro}
h(x) = \begin{dcases*} \frac{u^1(x) + u^2(x)}{2},\ x \in (-1,0),\\ \frac{u^3(x) + u^4(x)}{2},\ x \in (0,\ell).\end{dcases*}
  \end{equation}
Notice that the functions $f$, $g$ and $h$ are functions defined on classical intervals, the graph structure being hidden in this setting. Actually, one can check that
\begin{equation*}
    u \in D(\Delta_{\Gc}) \Leftrightarrow \left\{\begin{array}{l} f \in H^2(-1,0) \cap H_0^1(-1,0),\\ g\in H^2(0,\ell) \cap H_0^1(0,\ell),\\ h \in H^2(-1,\ell) \cap H_0^1(-1,\ell),\end{array}\right.
\end{equation*}
 and a direct computation gives
\begin{align*}
    \|u\|_{L^2}^2 &= \sum_{j=1}^{4} \|u^i\|_{L^2(e_j)}^2 = 2\left(\|f\|_{L^2(-1,0)}^2 + \|g\|_{L^2(0,\ell)}^2 + \|h\|_{L^2(-1,\ell)}^2\right),\\
    \|(\Delta_{\Gc} + \lambda )u\|_{L^2}^2 &= \sum_{j=1}^{4} \|(\Delta + \lambda )u^i\|_{L^2(e_j)}^2  \\
    & = 2\left(\|(\Delta + \lambda )f\|_{L^2(-1,0)}^2+ \|(\Delta + \lambda )g\|_{L^2(0,\ell)}^2 + \|(\Delta + \lambda )h\|_{L^2(-1,\ell)}^2 \right).
\end{align*}
We define now $(f_n)$, $(g_n)$ and $(h_n)$ from the sequence $(u_n)$ as above. We get
\begin{equation}\label{eq_norm}  \|f_n\|_{L^2(-1,0)}^2 +  \|g_n\|_{L^2(0,\ell)}^2 +  \|h_n\|_{L^2(-1,\ell)}^2 = 1/2,\end{equation}
\begin{equation}\label{eq_conv}   \Big\|\big(\Delta + \lambda_n \big)f_n\Big\|_{L^2(-1,0)}+\Big\|\big(\Delta + \lambda_n \big)g_n\Big\|_{L^2(0,\ell)}+\Big\|\big(\Delta + \lambda_n \big)h_n\Big\|_{L^2(-1,\ell)} \to 0\ \text{as}\ n \to +\infty,\end{equation}
\begin{equation}\label{eq_lim}    \int_{\omega \cap (-1,0)} |f_n + h_n|^2  + \int_{\omega \cap (0,\ell)} |g_n + h_n|^2 \to 0\ \text{as}\ n \to +\infty.\end{equation}
We use \eqref{eq_conv} to provide estimations for $f_n$ and $g_n$ 
at the central node. For instance, the first convergence infers
$$f_n''=-\lambda_n f_n+ r_n,\ \ \ \ \text{in}\ (-1,0), \ \ \text{ where}\ \ \ \ \|r_n\|_{L^2(-1,0)}\rightarrow 0\ \text{as}\ n \to +\infty.$$
We impose the boundary condition $f_n(0)=0$ and, solving the ODE by the variation of the constant formula, we get that
$$f_n(x)= {\alpha_n}  \sin (\sqrt{\lambda_n} x) + \frac{1}{\sqrt{\lambda_n}}\int_0^x\sin(\sqrt{\lambda_n}(x-y))r_n(y)\dd y, $$
where the integral term goes to zero when $n\rightarrow +\infty$. We repeat this explicit resolution to the other identities of \eqref{eq_conv} and we deduce that, when $n\rightarrow +\infty$,
\begin{equation}\begin{split}\label{eq_exp}
f_n(x) &= \alpha_n \sin(\sqrt{\lambda_n} x) + o \Big(\frac{1}{\sqrt{\lambda_n}}\Big),\qquad x\in[-1,0],\\
g_n(x) &= \beta_n \sin(\sqrt{\lambda_n} x) +  o\Big(\frac{1}{\sqrt{\lambda_n}}\Big),\qquad x\in[0,\ell].
\end{split}\end{equation}

\medskip

{\noindent \bf Step 2: Reduction to the resolvent estimate on an interval.}
Now, we claim that there exists $c >0$ such that 
\begin{equation}\label{low_bound}
|\alpha_n | \geq c~~\text{ and }~~  |\beta_n | \geq c.
\end{equation}
Indeed, if \eqref{low_bound} is not true, then either $(\alpha_n)$ or $(\beta_n)$  
goes to $0$, up to the extraction of a subsequence. Consider, for instance, $\alpha_n \to 0$ (the other case is similar). Then $\|f_n\|_{L^2(-1,0)} \to 0$ and, from the first limit of \eqref{eq_lim} and from  \eqref{eq_conv}, we would have
$$ \Big\|\big(\Delta + \lambda_n \big)h_n\Big\|_{L^2(-1,\ell)}~~\text{ and }~~\ \|h_n\|_{L^2(\omega \cap (-1,\ell))} \to 0.$$
Recall that Schrödinger equation is controllable in any interval with a control localised in any open subset, see for instance \cite[Theorem 4.2]{Laurent}. Due to Theorem \ref{lem:resolvent}, this provides a resolvent estimate for the Schrödinger equation in the interval $(-1,\ell)$ with the observation $\omega \cap (-1,0)$ and we necessarily have $\|h_n\|_{L^2( -1,\ell)} \to 0$. 
By using the second limit of \eqref{eq_lim}, we then have $\|g_n\|_{L^2(\omega \cap (0,\ell))} \to 0$. Then by conjugating with the second limit of \eqref{eq_conv} we have
$$ \Big\|\big(\Delta + \lambda_n \big)g_n\Big\|_{L^2(0,\ell)}~~\text{ and }~~\ \|g_n\|_{L^2(\omega \cap (0,\ell))} \to 0.$$
So from the resolvent estimate of the Schrödinger equation in the interval $(0,\ell)$ with the observation $\omega \cap (0,\ell)$, we necessarily have $\|g_n\|_{L^2(0,\ell)} \to 0$.  The combination of $\|f_n\|_{L^2(-1,0)} \to 0$, $\|h_n\|_{L^2( -1,\ell)} \to 0$ and $\|g_n\|_{L^2( 0,\ell)} \to 0$ is then a contradiction with respect to \eqref{eq_norm}. This shows \eqref{low_bound}.

\medskip

{\noindent \bf Step 3: Comparison of the two quasimodes.} We apply the estimate \eqref{eq_exp} to the Dirichlet boundary conditions. Since $f_n(-1) = 0$, we must have $\sin(\sqrt{\lambda_n}) = o\big(\frac{1}{\sqrt{\lambda_n}}\big)$, which implies the existence of integers $(p_n)$ such that 
\begin{equation}\label{eq_approx_lambda_1}
\sqrt{\lambda_n} = \pi p_n +  o\big(\frac{1}{\sqrt{\lambda_n}}\big). 
\end{equation}
Similarly, from $g_n(\ell)=0$, we obtain the existence of integers $(q_n)$ such that 
\begin{equation}\label{eq_approx_lambda_2}
\sqrt{\lambda_n} = \frac{\pi q_n}{\ell} +  o\big(\frac{1}{\sqrt{\lambda_n}}\big).
\end{equation}
Gathering \eqref{eq_approx_lambda_1} and \eqref{eq_approx_lambda_2}, we get 
\begin{equation}\label{eq_approx_lambda_3}
\ell p_n - q_n =  o\Big(\frac{1}{\sqrt{\lambda_n}}\Big)~\xrightarrow[~n\longrightarrow +\infty~]{}~0.
\end{equation}
Since we assumed that $\ell$ is badly approximable, it is not possible to extract converging sequences from $(p_n)$ or $(q_n)$ because, at the limit, \eqref{eq_approx_lambda_3} would imply that $\ell$ is rational. Thus, 
we should have $|p_n| \to +\infty$ and $|q_n| \to +\infty$. Using \eqref{eq_approx_lambda_1} or \eqref{eq_approx_lambda_2}, we deduce that $\lambda_n\to +\infty$ and that both $|p_n|$ and $|q_n|$ are of order $\mathcal{O}(\sqrt{\lambda_n})$. Finally, 
\begin{equation*}
    \Big|\ell - \frac{ q_n}{p_n} \Big| = o\Big(\frac{1}{p_n\,\sqrt{\lambda_n}}\Big)= o\Big(\frac{1}{p_n^2}\Big),
\end{equation*} 
that provides the contradiction with the fact that $\ell$ is a badly approximable number.

\medskip

\noindent
{\bf The “only if” part.} We just proved that, when $\ell$ is a badly approximable number,  \eqref{abs} holds and thus the controllability is ensured. To prove the other implication, we first observe that, when $\ell\in\QQ$, there exists eigenmodes localized in $e_3$ and $e_4$. Such modes obviously contradict \eqref{eq:resolventschro}. 
More generally, assume that $\ell$ is not a badly approximable number:
there exist two sequences $(p_n)$ and $(q_n)$ of integers such that  
\begin{equation}\label{eq_badly}
q_n |q_n\ell-p_n|\longrightarrow 0.
\end{equation}
Notice that, in particular,  $p_n\sim \ell q_n$.
Obviously, if $\ell$ is irrational, then \eqref{eq_badly} implies that $q_n\rightarrow +\infty$. If $\ell=p/q$ is rational, we can choose $p_n=np$ and $q_n=nq$ to satisfy both \eqref{eq_badly} and $q_n\rightarrow +\infty$. We define $u_n\in L^2(\Gc)$ by: 
\begin{itemize}
\item $u_n\equiv 0$ on $e_1$ and $e_3$, the observed edges,
\item $u_n(x)=\sin(q_n \pi x)$ for $x\in (-1,0)$ in the edge $e_2$,
\item $u_n(x)=\frac {\ell q_n}{p_n}\sin (p_n \pi x/\ell)$ for $x\in (0,\ell)$ in the edge $e_4$.
\end{itemize}
As in Section \ref{sec_quasimode}, we can easily check that $u_n\in D(\Delta_\Gc)$ and that
\begin{equation}\label{eq_badly2}
\|u_n\|_{L^2}^2=\frac 1{2}+\frac {\ell^3 q_n^2}{2 p_n^2}\geq \frac 1{2}.
\end{equation}
We also have that $(\Delta + q_n^2\pi^2)u_n$ vanishes everywhere except in $e_4$ and thus
\begin{align*}
\|(\Delta + q_n^2\pi^2)u_n\|_{L^2(\Gc)}^2&=\left\|\pi^2  \left(-\frac{p_n^2}{\ell^2} + q_n^2 \right) \frac {\ell q_n}{p_n}\sin(p_n\pi\cdot/\ell)\right\|^2_{L^2(0,\ell)}\\
&= \frac{\ell^3\pi^4 q_n^2}{2 p_n^2}\left|-\frac{p_n^2}{\ell^2} + q_n^2\right|^2.
\end{align*}
Due to \eqref{eq_badly}, we have $p_n = \ell q_n + o\big(\frac 1 {q_n}\big)$ and thus $p_n^2=\ell^2 q_n^2+o(1)$ and 
$$ \|(\Delta + q_n^2\pi^2)u_n\|_{L^2(\Gc)}\longrightarrow 0.$$
Together with \eqref{eq_badly2}, this estimate precludes the existence of a constant $C>0$ such that 
$$\|u_n\|_{L^2} \leq C \|(\Delta + q_n^2\pi^2)u_n\|_{L^2(\Gc)} + C\| u_n\|_{L^2(\omega)},$$
and then the exact controllability is not guaranteed in the whole space $L^2(\Gc)$.
\end{proof}


\section{Control of the heat equation}\label{sec_parabolic}

In this section, we study the null-controllability of the heat equation in a graph $\Gc$. 
Let us consider the controlled heat equation in $L^2(\Gc)$
\begin{equation}\label{eq_free_heat}
\left\{
\begin{array}{ll} \partial_{t} u = \Delta_\Gc u +q(x) u +  h 1_{\omega} ,& \ \ \ \ \ t>0,\\
u(t=0)=u_0\in L^2(\Gc), & 
\end{array}\right. 
\end{equation}
with $q\in L^\infty (\Gc)$ and $\Delta_{\Gc}$ the Laplacian defined on the graph $\Gc$, with either Dirichlet or Neumann conditions at the exterior vertices. In \eqref{eq_free_heat}, at time $t \geq 0$, $u(t,\cdot) : \Gc \to \mathbb{R}$ is the state and $h(t,\cdot) : \omega \to \mathbb{R}$ is the control.

\subsection{The null-controllability under (GGCC)}

We recall that Theorem \ref{prop:nullcontrolheat_intro} presented in the introduction is a direct consequence of \cite[Theorem 1.4]{AB} because the Apraiz, Barcena-Petisco's condition is a less geometric way to write (GGCC), see Theorem \ref{prop_GGCC} and Section \ref{sec_GGCC}. However, we underline that Theorem \ref{prop:nullcontrolheat_intro} is also a direct consequence of our study of the wave equation. Indeed, we have the following alternative proof.

\begin{proof}[Proof of Theorem \ref{prop:nullcontrolheat_intro}]
Assume that $\Gc$ and $\omega$ satisfy (GGCC). 
First, let us set $w = e^{-t(\|q\|_{\infty} + 1)} u$ that solves from \eqref{eq_free_heat} the new controlled problem
\begin{equation}\label{eq_free_heat_translat}
\left\{
\begin{array}{ll} \partial_{t} w = \Delta_\Gc w + (q(x) -\|q\|_{\infty} - 1) w +  e^{-t(\|q\|_{\infty} + 1)} h 1_{\omega} ,& \ \ \ \ \ t>0,\\
w(t=0)=u_0\in L^2(\Gc). & 
\end{array}\right. 
\end{equation}
Observe that $0$ does not belong to the spectrum of $\Delta_\Gc + q(x) -\|q\|_{\infty} - 1$. So one can apply Theorem \ref{th_exact_wave_intro_neu} to deduce that the associated controlled wave equation is exactly controllable for $T \geq L$, where $L$ is the length appearing in Definition \ref{defi_GGCC}. From \cite[Theorem 3.4]{Mil06}, we then deduce that \eqref{eq_free_heat_translat} is small-time null-controllable. By coming back to the variable $u = e^{t(\|q\|_{\infty} + 1)} w$, we finally obtain that \eqref{eq_free_heat} is small-time null-controllable.
\end{proof}

As a remark, let us notice that both proofs of Theorem \ref{prop:nullcontrolheat_intro}, the one from \cite[Theorem 1.4]{AB} and the second one written just above, lead to an estimate of the control cost, that is there exists $C=C(\Gc , \omega)>0$ such that for every $T>0$, for every $u_0 \in L^2(\Gc)$, there exists $h \in L^2((0,T)\times \omega)$ satisfying 
\begin{equation}
    \label{eq:controlcostheat}
    \|h\|_{L^2((0,T)\times\omega)} \leq C e^{C/T} \|u_0\|_{L^2(\Gc )},
\end{equation}
such that the solution $u \in \Cc([0,T],L^2(\Gc))$ of \eqref{eq_free_heat} satisfies $u(t=T) = 0$.
Also notice that \cite[Theorem 1.4]{AB} consider lower-order time and spatial dependent perturbations of the Laplace operator as 
\begin{equation}\label{eq_free_heat_perturbed}
\left\{\begin{array}{ll} \partial_{t} u = \Delta_\Gc u+ \partial_x(b(t,x) u) + a(t,x) u + h 1_{\omega}, & t>0\\
u(t=0)=u_0\in L^2(\Gc), & 
\end{array}\right. 
\end{equation}
where $b=b(t,x), a=a(t,x) \in L^{\infty}((0,T)\times \Gc)$. 

\subsection{Study of a graph with at most polynomially approximable lengths}
We just show that (GGCC) is a sufficient condition for the null-controllability of \eqref{eq_free_heat}. As already mentioned in \cite[Remarks 1.1 and 1.2]{AB}, it turns out that (GGCC) could also be a necessary condition for the null-controllability of \eqref{eq_free_heat}, for instance when the ratios of the lengths of the uncontrolled edges of graph are rationnally dependent. Indeed, in such a configuration, one can construct explicit eigenfunctions supported outside the control zone, disproving the associated observability inequality to heat equation that is equivalent to the null-controllability of \eqref{eq_free_heat}, see for instance \cite[Theorem 2.44]{Cor07}. However, in the remaining part of this section, we show that, in general, (GGCC) is not a necessary condition for the null-controllability of \eqref{eq_free_heat}. This fact will obviously follow once we prove Theorem \ref{prop:ggccnotnecessaryheat_intro}. 

\begin{theorem}[{\cite[Corollary 8.6]{DZ06}}]
\label{lem:boundarycontrolheat}
Let $N \geq 3$. Let $\Gc $ be a star-graph with $N$ edges $(e_i)_{i=1,\dots, N}$, $N$ exterior vertices $(v_i)_{i=1,\dots, N}$, with respective lengths $(\ell_i)_{i=1,\dots, N}$. Assume that $\ell_j/\ell_k$ with $j\neq k \in \{2, \dots, n\}$ are at most polynomially approximable numbers. Then, the boundary control heat equation
\begin{equation}\label{eq_free_heat_boundary}
\left\{\begin{array}{ll} \partial_{t} u = \Delta_\Gc u ,& t>0,\\
u(t,v_1) = h(t), & t >0,\\
u(t,v) = 0,& t >0,\ v \in \mathcal{V}_{\mathrm{ext}} \setminus \{v_1\},\\
u(t=0)=u_0\in L^2(\Gc), & 
\end{array}\right. 
\end{equation}
is small-time null-controllable, i.e. for every $T>0$, for every $u_0 \in L^2(\Gc)$, there exists $h \in L^2(0,T)$ such that the solution $u \in \Cc([0,T],H^{-1}(\Gc)) \cap L^2(0,T;L^2(\Gc))$ of \eqref{eq_free_heat_boundary} satisfies $u(t=T) = 0$.
\end{theorem}
Theorem \ref{prop:ggccnotnecessaryheat_intro} follows from the above result by using a general strategy passing from the boundary control result to an interior control result. Such a procedure has been widely used, see for instance \cite[Theorem 2.2]{AK11}.

\begin{proof}[Proof of Theorem \ref{prop:ggccnotnecessaryheat_intro}]
Let us take $\hat{\omega} \subset \subset \omega$, and we consider $\hat{\Gc} = \Gc \setminus \overline{\hat{\omega}}$. Clearly, $\hat{\Gc}$ is composed of two connected graphs, $\hat{\Gc}_1$ is a graph with the same structure as $\Gc$ and $\hat{\Gc}_2$ is a graph composed with only one edge. We also take $\theta \in \Cc^{\infty}(\Rr)$ and $\eta \in \Cc^{\infty}([0,T])$ such that
\begin{equation}
    \theta \equiv 1\ \text{in}\ \Gc \setminus \omega,\ \theta = 0 \ \text{in}\ \overline{\hat{\omega}},\ \eta=1\ \text{in}\ [0,T/4]\ \text{and}\ \eta \equiv 0\ \text{in}\ [3T/4,T].
\end{equation}
Let $U$ be the solution to \eqref{eq_free_heat} associated to $u_0 \in L^2(\Gc )$ and $h=0$. We solve now the two boundary control problems
\begin{equation}\label{eq_free_heat_boundary1}
\left\{\begin{array}{ll} \partial_{t} \hat{u}_1 = \Delta_{\hat{\Gc}_1} \hat{u}_1 ,& t>0,\\
\hat{u}_1(t,\hat{v_1}) = h_1(t), & t >0,\\
\hat{u}_1(t,v) = 0,& t >0,\ v \in \mathcal{V}_{\mathrm{ext}} \setminus \{\hat{v_1}\},\\
\hat{u}_1(t=0)=u_0 1_{\hat{\Gc}_1},\ \hat{u}_1(t=T)=0. & 
\end{array}\right. 
\end{equation}
and 
\begin{equation}\label{eq_free_heat_boundary2}
\left\{\begin{array}{ll} \partial_{t} \hat{u}_2 = \Delta_{\hat{\Gc}_2} \hat{u}_2 ,& t>0,\\
\hat{u}_2(t,\hat{v_2}) = h_2(t), & t >0,\\
\hat{u}_2(t,v) = 0,& t >0,\ v \in \mathcal{V}_{\mathrm{ext}} \setminus \{\hat{v_2}\},\\
\hat{u}_2(t=0)=u_0 1_{\hat{\Gc}_2},\ \hat{u}_2(t=T)=0. & 
\end{array}\right. 
\end{equation}
To solve \eqref{eq_free_heat_boundary1}, we use Theorem \ref{lem:boundarycontrolheat}, while to solve \eqref{eq_free_heat_boundary2}, we use for instance \cite[Theorem 2.4]{AK11}, i.e. the null-controllability of the heat equation with one boundary control, a result that goes back to \cite{FR71}. Let us set
\begin{equation}
   \hat{u} = \hat{u}_1 + \hat{u}_2,\  u = \theta(x) \hat{u} + (1- \theta(x)) \eta(t) U.
\end{equation}
Then $u$ is a solution to \eqref{eq_free_heat} with $h$ given by
\begin{equation}
\label{eq:defhcontrolheat}
    h = (1- \theta) \eta'(t) U + 2 \theta'(x) (\hat{u}_x - \eta(t) U_x) + \theta''(x) (\hat{u} - \eta(t) U),
\end{equation}
satisfying furthermore
$$u(t=T)=0.$$
By using the properties of $\eta$ and $\theta$, we have that the support of $h$ is contained in $[0,T]\times\omega$. Moreover, the following regularity property holds.

\medskip

Let us check now that $h \in L^{2}((0,T)\times \omega)$. 
The first term in \eqref{eq:defhcontrolheat} belongs 
to $L^{2}((0,T)\times \omega)$ because the Laplacian $\Delta_{\Gc}$ on the graph $\Gc$ equipped with Dirichlet boundary conditions generates a strongly continuous semigroup in $L^2(\Gc)$ and $u_0 \in L^2(\Gc)$. The last two terms of \eqref{eq:defhcontrolheat} have also the desired regularity property due to the well-known local regularizing effect of the heat equation set in an interval, stated here as a lemma.
\begin{lemma}
\label{lem:regheat}
Let $y \in \Cc([0,T],H^{-1}(0,\ell)) \cap L^2(0,T;L^2(0,\ell))$ and $f \in L^2((0,T)\times (0,\ell))$ be such that
\begin{align}\label{parab}\partial_t y - \partial_{x}^2 y = f\ \text{in}\ \mathcal{D}'((0,T)\times (0,\ell)),\qquad y(0,\cdot) = 0\qquad \text{in}\ (0,\ell).\end{align}
Then for any nonempty $\mathcal O \subset \subset (0,\ell)$, we have $y \in L^2(0,T;H^1(\mathcal O))$.
\end{lemma}
\begin{proof}
The proof is rather standard and consists in considering $\widetilde y=y \chi$ defined by $\chi\in \Cc^\infty([0,\ell],[0,1])$ such that $\chi(x)=1$ for $x\in \mathcal O$ and $\mathrm{supp}(\chi)\subset \subset (0,\ell)$. The function $\widetilde y$ solves a parabolic equation as in \eqref{parab} when Dirichlet boundary conditions are satisfied and with a source term that is not $f$ but rather $f-(\partial_{x}^2 \chi)y - 2 \partial_{x} \chi \partial_x y\in L^2(0,T ;H^{-1}(0,\ell))$. The result then follows from \cite[Theorem 10.42]{luc} which infers $$\widetilde y \in \Cc\big([0,T],L^2(0,\ell)\big)\cap L^2\big(0,T;H^{1}_0(0,\ell)\big)\cap H^1\big(0,T;H^{-1}(0,\ell)\big).$$
This ends the proof.
\end{proof}

\medskip

We come back to the proof of Theorem \ref{prop:ggccnotnecessaryheat_intro} and we apply Lemma \ref{lem:regheat} to $y=z_1 = \hat{u}_1 - \eta(t)U$ and to $y=z_2 = \hat{u}_2 - \eta(t) U$ satisfying respectively
$$\partial_t z_1 - \partial_{x}^2 z_1 =  -\eta'(t) U \ \text{in}\ \mathcal{D}'((0,T)\times (0,\ell_1')),\qquad z_1(0,\cdot) = 0\qquad \text{in}\ (0,\ell_1'),$$
and 
$$\partial_t z_2 - \partial_{x}^2 z_2 =  -\eta'(t) U \ \text{in}\ \mathcal{D}'((0,T)\times (0,\ell_2')),\qquad z_2(0,\cdot) = 0\qquad \text{in}\ (0,\ell_2'),$$
where $(0,\ell_1')$ denotes the controlled edge of $\hat{\Gc}_1$ and $(0,\ell_2')$ denotes the (unique) edge of $\hat{\Gc}_2$. By using the property of the support of $\theta$, we deduce that $\theta'(x)(\hat{u} - \eta(t) U)_{x}$ and $\theta''(x) (\hat{u} - \eta(t) U)$ belong to $L^2((0,T)\times \Gc)$. This concludes the proof.
\end{proof}


\section{Discussions}

In this final section, we provide a detailed discussion on several key topics related to the control and observation of partial differential equations (PDEs) on graphs. Specifically, we explore the optimality of the observation and control time for the wave equation, examining the conditions under which minimal time control can be achieved. Additionally, we address the question of the (potential) minimal control time for the Schrödinger equation and investigate the necessity of the Graph Geometric Control Condition (GGCC) for ensuring controllability on a generic graph. Furthermore, we analyze the polynomial stability of the wave equation, providing insights into its long-term behavior under various conditions. Finally, we compare and contrast boundary control and internal control of PDEs on graphs, highlighting their respective advantages, limitations, and implications for different types of networked structures.

\subsection{Optimality of the observation/control time for the wave equation}\label{optimal_time}

In this paper, we obtain the controllability of the wave equation as soon as (GGCC) holds and with any control time (strictly) larger than the length $L$ appearing in Definition \ref{defi_GGCC}. It is therefore natural to wonder if $L$ is actually the {\it optimal} control time, that is the infimum of all the times $T$ for which the control holds. We can answer shortly “no”. Indeed, let us consider a three edges star-graph equipped with Dirichlet boundary conditions. Assume we observe every edge only on the exterior side so that the unobserved part of the graph consists again in a three edges star-graph of lengths $\ell_1>\ell_2>\ell_3$ meeting at one vertex. This $\bot$-shaped graph is represented in Figure \hyperref[fig_3star]{5}. 
For this graph, the optimal time for the observation of the free wave equation is 
\begin{equation}\label{def_Tstar}
T^*=\max(2\ell_2,\ell_1+\ell_3).
\end{equation}

\begin{figure}[ht]
\begin{center}
\resizebox{0.35\textwidth}{!}{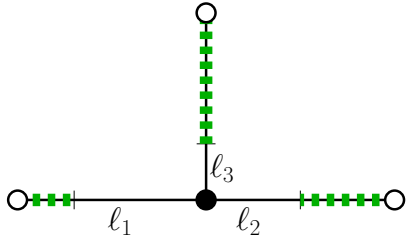}
\end{center}
\caption{\it The unobserved part of the studied star-graph consists in three edges of lengths $\ell_1>\ell_2>\ell_3$ meeting at $x=y=0$.}\label{fig_3star}
\end{figure}

\noindent To enlight the discussion, the proof of this fact is postpone in Appendix \ref{sec_optimal_time}.
Here, we simply notice that the optimal time for observation is not always the length $\ell_1+\ell_2$ of the longest unobserved path, the sufficient time of Proposition \ref{prop_obs_wave}, neither the length $\ell_2+\ell_3$ of the shortest unobserved path, or the length $\ell_1+\ell_3$, the longest time needed for a wave to travel from one end of the unobserved part to another end. 

\medskip

Actually, this time $T^*$ of \eqref{def_Tstar} is the exactly the one introduced by Avdonin and Zhao in \cite{AZ21}: cover the uncontrolled part $\Gc\setminus\omega$ of the graph by a watershed and minimize the length of the longest river (see the geometric description of Definition \ref{defi_watershed}), the optimal time is twice this length. The optimal watershed for the graph of Figure \hyperref[fig_3star]{5} is reached by a river covering the unobserved part of $e_2$ and two rivers covering the unobserved part of $e_1\cup e_3$ starting at each end this segment and meeting at the middle. This watershed provides a time $T^*$ which is exactly the optimal time \eqref{def_Tstar}. In this sense, we can say that the method of Avdonin and Zhao in \cite{AZ21} is better than ours since it provides a better time. However, we notice that their method is purely based on d'Alembert's formula and is less robust with respect to changes in the PDE. It is also noteworthy that Avdonin and Zhao themselves argue in \cite{AZ21} that their method is not optimal either. First notice that they state that there exist graphs for which their time is optimal, but they only provide examples, not a general fact, see Part 3 of Theorem 1 of \cite{AZ21} and its proof. They are also able to give an example where their method does not provide the optimal time for the “shape control”, that is the problem of controlling the shape function $u$ only, no target being specified for $\partial_t u$, see Remark 1 of \cite{AZ21}. 

\medskip

To conclude this discussion, as far as we know, computing the optimal time for the observation/control of the wave equation on a general graph is still an open question. We know that our time coming from (GGCC) is not optimal in general and that the time provided by the watershed description of Avdonin and Zhao is always shorter. We also know that their description do not provide the optimal time for “shape control”. However, our experience on this topics leads us to the following conjecture.
\begin{claim}
The watershed description of Avdonin and Zhao in \cite{AZ21} provides the optimal time for the observation and control of the free wave equation on a graph.  
\end{claim}

\subsection{On the (possible) minimal time for the Schrödinger equation}
\label{sec:discussionschromin}

In Theorem \ref{non_ggcc_sch_intro}, we exhibit a specific configuration where the Schrödinger equation turns out to be exactly controllable for some $T>0$ despite (GGCC) is not satisfied. A natural question would consist in investigating if there is indeed a minimal time of control or not. Actually, based on our experience on the subject of the internal control of linear Schrödinger equation where as far as we know, minimal time never appears, we conjecture the following.

\begin{claim}\label{claim_schro}
Let us consider the X graph with four edges, two (say $e_1$ and $e_2$) of length $\ell_{\text{b}}>0$ and two ($e_3$ and $e_4$), of length $\ell_{\text{t}}>0$ and $\omega = e_1 \cup e_3$. Assume that the ratio $\ell_{\text{t}}/\ell_{\text{b}}$ is a badly approximable number then the Schr\"odinger equation 
\begin{equation}
\label{eq:schrodiscussminimal}
\left\{
\begin{array}{ll} i \partial_{t} u = - \Delta_\Gc u + h 1_{\omega} ,& \ \ \ \ t>0,\\
u(t=0)=u_0\in L^2(\Gc), & 
\end{array}\right. 
\end{equation}
is small-time exactly controllable in $L^2(\Gc)$.
\end{claim}

One could hope to prove the above statement by exploiting the following general result, see \cite[Remark 5.1]{Laurent}.
\begin{theorem}
\label{lem:resolventImproved}
Let $\varepsilon >0$. If there exists $C_1, C_2 >0$ such that
\begin{equation}
\label{eq:resolventschroImproved}
    \|u\|_{L^2(\Gc)} \leq C_1 (|\lambda|^{-\varepsilon} \|\Delta_{\Gc} u + \lambda u\|_{L^2(\Gc)} + \|u\|_{L^2(\omega)}) \qquad \forall \lambda \geq  C_2,\ \forall u \in D(\Delta_{\Gc}).
\end{equation}
Then, \eqref{eq:schrodiscussminimal} is exactly controllable in $L^2(\Gc)$ for any time $T>0$.
\end{theorem}
However, the arguments in the proof of Theorem \ref{non_ggcc_sch_intro} shows that \eqref{eq:resolventschroImproved} should never hold, even if $\ell_{\text{t}}/\ell_{\text{b}}$ is a badly approximable number. Indeed, we can see there that \eqref{eq:resolventschroImproved} with $\varepsilon=0$ is the best that we can show. Thus, to prove Conjecture \ref{claim_schro}, one needs to carefully investigate the observability of the time evolution Schrödinger equation with the observation set $\omega$ instead of the stationary resolvent estimate \eqref{eq:resolventschroImproved}.

\subsection{On the necessity of (GGCC) for a generic graph}
\label{sec:discussionGGCCSchroAlmost}
In this article, we have seen that (GGCC) is not always necessary to the control of PDEs, except for the wave equation with Dirichlet boundary conditions. Our examples are very specific and we may wonder whether they are representative of the general situation. In the following discussion, we consider that picking a “generic” graph consists of having a fixed graph structure and randomly choosing the lengths of the edges with respect to the Lebesgue measure, which is certainly equivalent to choosing both the lengths and the graph structure by any reasonable random process.

\medskip

First, we consider the case of the wave equation with mixed boundary conditions. We know that (GGCC) is not necessary in general due to the specific example of Theorem \ref{th_neumann}. This example is based on the pointwise observation result stated in \cite{AHT}. Here, we can point a general intuition that controlling a graph without (GGCC) is very related to controlling at a single point the same PDE on a segment: the pointwise control of the wave equation in $(0,1)$ fails for Dirichlet boundary condition (\cite{JTZ}) but may be possible for mixed boundary conditions (\cite{AHT}). This is due to the fact that the modes $\cos\big(\frac{2n+1}2\pi\cdot\big)$ of the later case always have value $1$ at $x=p/q$ with odd $p$. In \cite{AHT}, it is shown that, if the pointwise control is acting at any other place, then the exact control fails. So, we expect that the situation of Theorem \ref{th_neumann} can only occur for specific length ratios. 
\begin{claim}\label{conj_neu}
Let $\Gc$ be a graph with $N$ edges $(e_i)_{i=1,\dots, N}$ and $\omega$ be an open subset of $\Gc$. Then, generically with respect to the lengths $(\ell_i)_{i=1,\dots, N}$, (GGCC) is a necessary and sufficient condition for the internal control of the wave equation, whatever the boundary conditions are.
\end{claim}
As with all the conjectures in this section, we leave the study to future works.

\medskip

Concerning the Schrödinger equation on graphs, Theorem \ref{non_ggcc_sch_intro} shows that (GGCC) is not necessary for the internal control, the X graph with (very) specific lengths providing indeed a counterexample. But this situation also seems to be very particular due to the following result.
\begin{theorem}
\label{th:ggccnecessarysuffschro}
Let $\Gc$ be a star-graph with $N$ edges $(e_i)_{i=1,\dots, N}$, with respective lengths $(\ell_i)_{i=1,\dots, N}$. Suppose that the lengths $(\ell_i)_{i=1,\dots, N}$ are linearly independent over $\mathbb Q$. Assume that $\Gc$ and $\omega$ do not satisfy (GGCC). Then, for every $T>0$, the Schrödinger equation with Dirichlet boundary conditions
\begin{equation*}
\left\{
\begin{array}{ll} i \partial_{t} u = - \Delta_\Gc u + h 1_{\omega} ,& \ \ \ \ t>0,\\
u(t=0)=u_0\in L^2(\Gc), & 
\end{array}\right. 
\end{equation*}
is not exactly controllable in $L^2(\Gc)$ at time $T$.
\end{theorem}
From Theorem \ref{th:ggccnecessarysuffschro} and Theorem \ref{prop:exactcontrolschro_intro}, we have that given a generic star-graph, (GGCC) is a necessary and sufficient condition for the internal control of the Schrödinger equation. Hence, the counterexample exhibited in Theorem \ref{non_ggcc_sch_intro} seems to be very specific. The proof of Theorem \ref{th:ggccnecessarysuffschro} is based on the following result telling that, for the above configuration, there are eigenfunctions that localise on two edges. We denote here by $(-\lambda_n,\phi_n)$ the sequence of eigenvalues and $L^2-$normalized eigenfunctions of the Laplacian operator $\Delta_\Gc$.
\begin{theorem}[{\cite[Theorem 1.4]{BK04}}]
Let $\Gc$ be a star-graph with $N$ edges $(e_i)_{i=1,\dots, N}$, with respective lengths $(\ell_i)_{i=1,\dots, N}$. Suppose that the lengths $(\ell_i)_{i=1,\dots, N}$ are linearly independent over $\mathbb Q$. Then given two edges $e_i$ and $e_j$, there exists a subsequence $(\phi_{n_k})$ such that for every $f \in L^2(\Gc)$, the following limit holds
\begin{equation}
\label{eq:limiteigen}
    \lim_{k \to +\infty} \int_{\Gc} |\phi_{n_k}|^2(x) f(x) \dd x= \frac{1}{\ell_i + \ell_j} \left(\int_0^{\ell_i} f(x) \dd x + \int_0^{\ell_j} f(x) \dd x \right).
\end{equation}
\end{theorem}
Then, the proof of Theorem \ref{th:ggccnecessarysuffschro} relies on the equivalence between the exact controllability of the Schrödinger equation and the resolvent estimate, as stated in Theorem \ref{lem:resolvent}. Indeed, assume that $\Gc$ and $\omega$ do not satisfy (GGCC) then there exist two uncontrolled edges $e_i$ and $e_j$. We apply \eqref{eq:limiteigen} with $f = 1_{\omega}$ and we deduce
$$ \lim_{k \to +\infty} \|\phi_{n_k}\|_{L^2(\omega)} = 0.$$
This limit disproves any resolvent estimate of the type \eqref{eq:resolventschro} if we set $u=\phi_{n_k}$ and $\lambda=\lambda_{n_k}$, where $k$ is chosen large enough.

Actually, we conjecture that (GGCC) is necessary and sufficient for the control of the Schrödinger equation for almost all graphs.
\begin{claim}\label{conj_sch}
Let $\Gc$ be a graph with $N$ edges $(e_i)_{i=1,\dots, N}$ and $\omega$ be an open subset of $\Gc$. Then generically with respect to the lengths $(\ell_i)_{i=1,\dots, N}$, (GGCC) is a necessary and sufficient condition for the internal control of the Schrödinger equation.
\end{claim}
A natural way to tackle this conjecture would consist in exploiting \cite[Theorem 4.1]{Col} telling us that the minimal supports of the semi-classical measures for a generic graph are the cycles and the path connecting two exterior vertices. Then, roughly speaking, $\omega$ has to intersect these supports, i.e. every cycle and every path connecting two exterior vertices. Hence, $\Gc$ and $\omega$ have to satisfy (GGCC).\\

On the contrary, by Theorem \ref{prop:ggccnotnecessaryheat_intro} and Appendix \ref{diophantine}, the internal control of the heat equation turns out to be completely different in the following sense.
\begin{theorem}
\label{th:ggccnecessarysuffheat}
Let $\Gc$ be a star-graph with $N$ edges $(e_i)_{i=1,\dots, N}$ and $\omega$ be an open subset of $\Gc$. Then generically with respect to the lengths $(\ell_i)_{i=1,\dots, N}$, for every $T>0$, the heat equation with Dirichlet boundary conditions
\begin{equation*}
\left\{
\begin{array}{ll}  \partial_{t} u = - \Delta_\Gc u + h 1_{\omega} ,& \ \ \ \ t>0,\\
u(t=0)=u_0\in L^2(\Gc), & 
\end{array}\right. 
\end{equation*}
is small-time null-controllable in $L^2(\Gc)$.
\end{theorem}
We naturally expect this last result to extend to more general types of graph.
\begin{claim}\label{conj_heat}
Let $\Gc$ be a graph with $N$ edges $(e_i)_{i=1,\dots, N}$ and $\omega$ be an open subset of $\Gc$. Then generically with respect to the lengths $(\ell_i)_{i=1,\dots, N}$, the heat equation is small-time null-controllable.
\end{claim}

\subsection{Polynomial stability of the wave equation}

Assume that the exact controllability of the Schr\"odinger equation \eqref{eq_schro_control_intro} holds (as in Definition \ref{defcontsch}), while the geometric control condition (GGCC) is not satisfied (as in Subsection \ref{subsect5.2}). Under these conditions, and according to the general implication \cite[Theorem 2.3]{AL14}, we have a polynomial stability result for the following damped wave equation:
\begin{equation}\label{intro_wave_damp_poly} 
\left\{ 
\begin{array}{ll} 
\partial_{t}^2 u - \Delta_\Gc u + a(x) \partial_t u = 0, \qquad t>0, \\
(u,\partial_t u)|_{t=0}=(u_0,u_1) \in D(\Delta_{\mathcal{G}}) \times H^1_\Delta(\Gc), 
\end{array}
\right. 
\end{equation}
where the damping coefficient $a \geq 0$ belongs to $L^\infty (\Gc)$ and $\omega \subset \mathrm{supp}(a)$. More precisely, we obtain the existence of $C>0$ such that, for all $(u_0,u_1) \in D(\Delta_{\mathcal{G}}) \times H^1_\Delta(\Gc)$,
\begin{equation}\label{dec-poly}
\forall t\geq 0,\
\left\|(u(t),\partial_t u(t))\right\|_{H^1_\Delta(\Gc)\times L^2(\Gc)} \leq \frac{C}{\sqrt{t}}
\left\|(u_0,u_1)\right\|_{D(\Delta_{\mathcal{G}}) \times H^1_\Delta(\Gc)}.
\end{equation}
 In the general case where (GGCC) fails, the damped wave equation with Dirichlet boundary condition is not uniformly stable but we can still expect to obtain a weak polynomial decay of the type \eqref{dec-poly} as soon as the lengths of the undamped edges have irrational ratio. As a guiding example, we recall the work \cite{JTZ} where the wave equation on $(0,1)$ is damped at a single place with a Dirac damping $a(x)=\delta_{x=x_0}$. It is shown that the rate of decay depends on the Diophantine properties of $x_0$. Following the above discussion, we claim that the following results hold.
\begin{prop}
Consider the damped wave equation on the X graph of Figure \hyperref[fig_ex]{2} with Dirichlet boundary conditions. If $\ell_{\text{t}}/\ell_{\text{b}}$ is rational, then some solutions of \eqref{intro_wave_damp_poly} do not converge to zero. If $\ell_{\text{t}}/\ell_{\text{b}}$ is a badly approximable number, then every solutions goes to zero and we have the estimate \eqref{dec-poly}.
\end{prop}
\begin{claim}
Consider the damped wave equation on the X graph of Figure \hyperref[fig_ex]{2}. If $\ell_{\text{t}}/\ell_{\text{b}}$ is a at most $\sigma-$approximable number, then every solutions of \eqref{intro_wave_damp_poly} goes to zero and we have the estimate
$$\forall t\geq 0,\
\left\|(u(t),\partial_t u(t))\right\|_{H^1_\Delta(\Gc)\times L^2(\Gc)} \leq \frac{C}{t^{\frac 12(\sigma-1)}}
\left\|(u_0,u_1)\right\|_{D(\Delta_{\mathcal{G}}) \times H^1_\Delta(\Gc)}.$$
\end{claim}

\subsection{Comparison between boundary/internal control of PDEs on graphs}
\label{sec:boundaryinternalpdes}

The goal of this part consists in comparing internal control results to boundary control results for the wave equation, the Schrödinger equation and the heat equation. In the context of boundary control, let us recall that (GGCC) means that $\Gc$ is a tree, i.e. $\Gc$ does not contain cycles, and that the control is efficient at all the exterior vertices except maybe one. In what follows, we use the notation $\mathcal{C}$ for the controlled exterior vertices.

We begin by the wave equation. Let us recall the following result, due to Schmidt \cite{Schmidt}, see also for instance \cite[Theorem 2.7]{DZ06}. 
\begin{theorem}[Schmidt, 1992]
\label{th:schmidt}
If $\Gc$ is a tree and the set $\mathcal{C}$ contains all the exterior vertices, except at most one, then the wave equation is exactly controllable in the state space $L^2(\Gc) \times H^{-1}(\Gc)$ in any time $T \geq T^*$ with boundary controls in $L^2(0,T)$, where $T^*$ is twice the length of the largest simplest past connecting the uncontrolled vertices with the controlled ones.
\end{theorem}

The previous result, i.e. Theorem \ref{th:schmidt} is optimal in the following sense, see \cite[Section 6.3]{DZ06}.
\begin{theorem}
If $\Gc$ is a tree and there exist at least two uncontrolled vertices, then the wave equation is not exactly controllable in the state space $L^2(\Gc) \times H^{-1}(\Gc)$ whatever the time $T>0$ is, with boundary controls in $L^2(0,T)$.
\end{theorem}
As a consequence of the two above results, we observe that as in the case of internal controls, (GGCC) is a necessary and sufficient condition for the exact controllability of the wave equation with boundary controls. \\

Let us now consider the Schrödinger equation. By using standard results that allow to pass from an exact controllability result for the wave equation to a small-time exact controllability result for the Schrödinger equation (see \cite{Mil05}), we have the following theorem.
\begin{theorem}
\label{th:boundarycontrolSchropos}
If $\Gc$ is a tree and the set $\mathcal{C}$ contains all the exterior vertices, except at most one, then for every $T>0$ the Schrödinger equation is exactly controllable at time $T$ in the state space $H^{-1}(\Gc)$ with boundary controls in $L^2(0,T)$.
\end{theorem}
However, as in the context of internal controls, (GGCC) is not a necessary condition for the exact controllability of the Schrödinger equation with boundary controls. Indeed, we consider again the X graph: a star-graph with four edges, two (say $e_1=(0,v_1)$ and $e_2=(0,v_2)$) of length $\ell_{\text{b}}>0$ and two ($e_3=(0,v_3)$ and $e_4=(0,v_4)$), of length $\ell_{\text{t}}>0$. The boundary controls act on the exterior vertices $v_1$ and $v_3$. This geometry does not satisfy (GGCC) because the control is efficient at $2$ exterior vertices while there are $4$ exterior vertices.
\begin{theorem}
\label{th:boundarycontrolSchro}
Consider the above geometric setting of the X graph $\Gc$. The Schr\"odinger equation 
\begin{equation*}
\left\{
\begin{array}{ll} i \partial_{t} u = - \Delta_\Gc u ,& \ \ \ \ t>0,\\
u(t,v_1)=h_1(t),\ u(t,v_3) = h_3(t),\ u(t,v_2)=u(t,v_4) = 0,\\
u(t=0)=u_0\in H^{-1}(\Gc), & 
\end{array}\right. 
\end{equation*}
is exactly controllable in $H^{-1}(\Gc)$ for some time $T>0$ if and only if the ratio $\ell_{\text{t}}/\ell_{\text{b}}$ is a badly approximable number.
\end{theorem}
The proof of Theorem \ref{th:boundarycontrolSchro} is an adaptation of the one of Theorem \ref{non_ggcc_sch_intro}, but for the sake of completeness we provide it in Appendix \ref{sec:boundaryschroproofNoGGCC}.\\

We finally go back to the heat equation where the following result holds, as a consequence of Theorem \ref{th:schmidt} and \cite{Mil06}, that allows to pass from an exact controllability result for the wave equation to a small-time null-controllability result for the heat equation.
\begin{theorem}
\label{th:boundarycontrolheatpos}
If $\Gc$ is a tree and the set $\mathcal{C}$ contains all the exterior vertices, except at most one, then for every $T>0$ the heat equation is null-controllable at time $T$ in the state space $H^{-1}(\Gc)$ with boundary controls in $L^2(0,T)$.
\end{theorem}
Note also that one can deduce Theorem \ref{th:boundarycontrolheatpos} by using Theorem \ref{prop:nullcontrolheat_intro} and the general strategy passing from the internal control result to a boundary control result, see for instance \cite[Theorem 2.2]{AK11}. By recalling Theorem \ref{lem:boundarycontrolheat}, i.e. \cite[Corollary 8.6]{DZ06} that establishes the possibility of controlling the heat equation in arbitrary time in some specific star-shaped graph with only one boundary control, we have that (GGCC) is a sufficient but not necessary condition for the null-controllability of the heat equation with boundary controls. Therefore, the situation is analogue to the case of internal controls, see Theorem \ref{prop:nullcontrolheat_intro} and Theorem \ref{prop:ggccnotnecessaryheat_intro}.

\medskip


\appendixtitleon
\appendixtitletocon

\begin{appendices}

\section{Diophantine approximation}\label{diophantine}

As we can see in this article, the speed with which a real number can be approximated by rational numbers can have a great influence in mathematical studies, particularly in control theory. Studying the speed of convergence of the rational approximation is the purpose of the theory of Diophantine approximations.
Based on the pigeonhole principle, Dirichlet proves in 1842 the following result.
\begin{theorem}[Dirichlet's approximation theorem]\label{th_Dirichlet_hist}
For any real number $\alpha$ and any $N\in\NN^*$, there exist integers $p$ and $q$ with $1\leq q\leq N$, such that 
$$ \left|\alpha - \frac {p}q\right| < \frac 1{qN}.$$
\end{theorem}
In particular, there exist sequences $(p_n)$ and $(q_n)$ of integers such that
\begin{equation}\label{eq_approx1}
q_n\xrightarrow[~n\longrightarrow+\infty~]{}+\infty~~~\text{ and }~~~\left|\alpha - \frac {p_n}{q_n}\right| < \frac 1{q_n^2}.
\end{equation}
Indeed, if $\alpha=p/q$, we can choose $p_n=np$ and $q_n=nq$ and if $\alpha\not\in\QQ$, we apply Theorem \ref{th_Dirichlet_hist} with increasing $N$ and notice that $1/N\leq 1/q$.
We recall that, if $\alpha$ is irrational, the “best approximations" $p_n/q_n$ are provided by the expansion in {\it continuous fraction}
\begin{equation}\label{eq_continuousfraction}
\alpha=a_0+\frac 1{a_1+\frac 1{a_2+\frac 1{a_3+\frac 1{a_4+\dots}}}}=:[a_0;a_1,a_2,a_3,a_4,\dots],
\end{equation}
see for example \cite{HW}. The coefficients $(a_n)$ are called the {\it partial quotients}.

\medskip

In the present article, we need the “simultaneous version” of Theorem \ref{th_Dirichlet_hist}. This very classical result can be found in many textbooks, see for example \cite{Bugeaud} or \cite[Theorem VII]{Cassels}.
\begin{theorem}\label{th_Dirichlet}
For any real numbers $\alpha_1$,\ldots, $\alpha_d$ and any $N\in\NN^*$, there exist integers $p_j$ and $q$ with $1\leq q\leq N^d$, such that 
\begin{equation}\label{eq_th_Dirichlet}
\forall j=1,\ldots, d~,~~\left|\alpha_j - \frac {p_j}q\right| \leq \frac 1{qN}.
\end{equation}
\end{theorem}
Notice that \eqref{eq_th_Dirichlet} provides a simultaneous approximation of $d$ numbers with an error of order $1/q_n^{1+1/d}$, generalizing the bound $1/q_n^2$ of \eqref{eq_approx1}. We use in this article the slightly different version stated in Proposition \ref{coroprop_Dirichlet}. This last result is an immediate corollary of the classical theorem.
\begin{proof}[Proof of Proposition \ref{coroprop_Dirichlet}]
We apply Theorem \ref{th_Dirichlet} with the numbers $(\alpha_i)$. If the resulting $q$ is larger than $\lfloor \sqrt{N}\rfloor$, Proposition \ref{coroprop_Dirichlet} is proved. If not, we set $k=\lfloor \sqrt{N}\rfloor$, $q'=kq$ and $p'_j=kp_j$. Obviously, $\frac{p'_j}{q'}=\frac{p_j}q$, $q'\geq \lfloor \sqrt{N}\rfloor$ and $q'= \lfloor \sqrt{N}\rfloor q \leq \lfloor \sqrt{N}\rfloor^2\leq N^d$. It remains to notice that the bound in \eqref{eq_th_Dirichlet} becomes $\frac 1{qN}=\frac {\lfloor \sqrt{N}\rfloor}{q'N}\leq \frac 1{q'\sqrt{N}}$.
\end{proof}

\medskip

In the view of the bound \eqref{eq_approx1} provided by Dirichlet's approximation theorem, it is natural to study the speed of convergence of the rational approximations. Actually, in the present paper, we are rather interested in numbers with a bad approximation rate, which leads to the following definition.
\begin{defi}\label{defi_sigma_approx}
Let $\sigma>0$. We say that a number $\alpha\in\RR$ is {\bf at most $\sigma-$approximable} if there exists $C=C(\sigma)>0$ such that, for all $(p,q)\in\ZZ\times\NN^*$, 
$$\left|\alpha-\frac pq\right|\geq \frac C{q^{\sigma}}.$$
If $\alpha$ is at most $\sigma-$approximable for some $\sigma > 0$, we say that it is {\bf at most polynomially approximable}.
\end{defi}
The above definitions are not commonly accepted ones but they are convenient in our context. However, the special case of numbers that are “as far from rationals as possible" has a well-accepted denomination.
\begin{defi}\label{defi_badly_approx}
A real number $\alpha\in\RR$ is said to be {\bf badly approximable} if there exists $C>0$ such that, for all $(p,q)\in\ZZ\times\NN^*$, 
$$\left|\alpha-\frac pq\right|\geq \frac C{q^2}.$$
\end{defi}
In other words, badly approximable numbers are numbers that are at most $2$-appro\-xi\-mable and due to the bound \eqref{eq_approx1}, this is the worst possible rate. In 1844, Liouville proves one of his most famous theorems in \cite{Liouville}. 
\begin{theorem}[Liouville's theorem on diophantine approximation]\label{th_Liouville}
If $\alpha$ is an algebraic number of degree $d\geq 2$, that is that $\alpha$ is an irrationnal number which is a root of a polynomial of $\ZZ[X]$ of degree $d$, then there exists $C(\alpha)>0$ such that 
\begin{equation}\label{eq_Liouville}
\forall p\in\ZZ~,~~\forall q\in\NN^*~,~~\left|\alpha-\frac pq \right|\geq \frac {C(\alpha)}{q^d}.
\end{equation}
\end{theorem}
In particular, for $d=2$, this result implies that any algebraic number of degree two, as $\sqrt{2}$, is a badly approximable number. It also implies that any irrational algebraic number is at most polynomially approximable. Liouville uses this fact to construct the first example of transcendental number, the famous Liouville's constant $\sum_{n\geq 1} {10^{-n!}}$. Later, the works of Thue, Siegel, Dyson and Roth show that the lower bound $C(\alpha)/q^d$ of \eqref{eq_Liouville} can be replaced by $C(\alpha,\varepsilon)/q^{2+\varepsilon}$ for any $\varepsilon>0$, see \cite{Roth}. However, it is still  unknown in general if algebraic numbers of degrees other than $d=2$ may be badly approximable numbers or not, even if numerical experiments suggest that this is not the case.
At least, it is known that the class of badly approximable numbers is larger than the set of algebraic numbers of degree $2$: it is actually the set of irrational numbers $\alpha$ whose sequence of partial quotients $(a_n)$, coming from the expansion in continued fraction \eqref{eq_continuousfraction}, is bounded, see \cite[Theorem 6, p.24]{Lang} or \cite{HW} for example.

\medskip

To conclude this brief and incomplete review, let us summarise some interesting facts to bear in mind when considering our results.
\begin{itemize}
\item {\it The algebraic numbers of degree $2$ are examples of badly approximable numbers but there exist many other badly approximable numbers}. For example, it is sufficient to consider a sequence $(a_n)$ which is quasi-periodic but not periodic. Indeed, Lagrange has shown that algebraic numbers of degree $2$ are exactly the irrational numbers such that the sequence $(a_n)$ is ultimately periodic, see for example \cite[Chapter 10]{HW}. Surprisingly, there even exist numbers that are both transcendental and badly approximable, as $\sum_{n\in\NN} 3^{-2^n}$, see \cite{Shallit}.

\item {\it There exist examples of non-badly approximable numbers that are less artificial than Liouville's constant}. Indeed, as shown by Euler himself in 1737, the expansion of “his" number in continued fractions is 
$$e=[2;1,2,1,1,4,1,1,6,1,1,8,1,1,10,1,1,12,\dots].$$
Actually, the best rate of approximation of $e$ is of order $\frac{\ln(\ln q)}{q^2\ln q}$ as shown in \cite{Davis}. In particular, notice that $e$ is at most $\sigma-$approximable for any $\sigma>2$.
The case of $\pi$, on the other hand, seems still open. 
\item {\it Almost all numbers are not badly approximable but almost all are at most polynomially approximable}. Indeed, Borel and Bernstein have shown that the set of badly approximable numbers is a set of Lebesgue measure zero, see \cite[Chapter 11]{HW}. On the contrary, Khinchin proves in 1926 that, as soon as $\sigma>2$, the set of numbers that are at most $\sigma-$approximable is of full Lebesgue measure, see \cite[p. 120]{Cassels}.
\end{itemize}

\section{Reflexion and transmissions of waves at an interior vertex}\label{section_waves_io}

The one-dimensional wave equation has well-known solutions, which can be expressed using d’Alembert’s formula. In particular, traveling waves of the form $\varphi(x\pm t)$ represent valid solutions. In this appendix, we examine the behavior of such traveling wave solutions as they approach a vertex of the graph. Our goal is to analyze how the wave interacts with the vertex and to understand how the wave’s energy is distributed or split among the edges connected to the vertex. This investigation provides deeper insight into the dynamics at graph vertices, which is crucial for studying wave propagation and controllability on graphs. 

\medskip

We consider the following toy model: let $I_i=(0,+\infty)$, for $i=1$, $2$ and $3$, be three infinite edges, connected at their endpoint $x=0$. Let $\varphi\in\Cc^2(\RR,\RR)$. We consider $u=(u_1,u_2,u_3)$ defined on the three edges infinite star-graph by 
$$\forall x\in I_i~,~~\forall t\in\RR~,~~ u_i(x,t)=\alpha_i \varphi(x+t) + \beta_i \varphi(-x+t),$$
where $\alpha_i\in\RR$ is the amplitude of the wave arriving to the vertex $x=0$ and $\beta_i\in\RR$ is the amplitude of the outgoing wave. To ensure that $u\in\Cc^2(\Gc,\RR)$ is a solution of 
$$\left\{\begin{array}{l}
\partial_{t}^2 u_i(x,t)=\partial_{x}^2 u_i(x,t) ~~~ ~~~~~ ~~\forall x\in I_i,\ i\in\{1,2,3\},\\
u_1(0,t)=u_2(0,t)=u_3(0,t),\\
\sum_{i=1}^3 \partial_{x} u_i(0,t)=0,\end{array}\right.
$$
we simply have to require
$$\alpha_1+\beta_1=\alpha_2+\beta_2=\alpha_3+\beta_3~~~\text{ and }~~~\sum_{i=1}^3 \alpha_i = \sum_{i=1}^3 \beta_i.$$
Obviously, for any choice of the ingoing coefficients $\alpha_i$, the outgoing coefficients $\beta_i$ are determined uniquely. Several choices of the coefficients provide interesting particular solutions that may be relevant to understand the wave equation on graphs, see Figure \hyperref[fig_waves_io]{6}.

\begin{figure}[htp]
\begin{center}
\resizebox{\textwidth}{!}{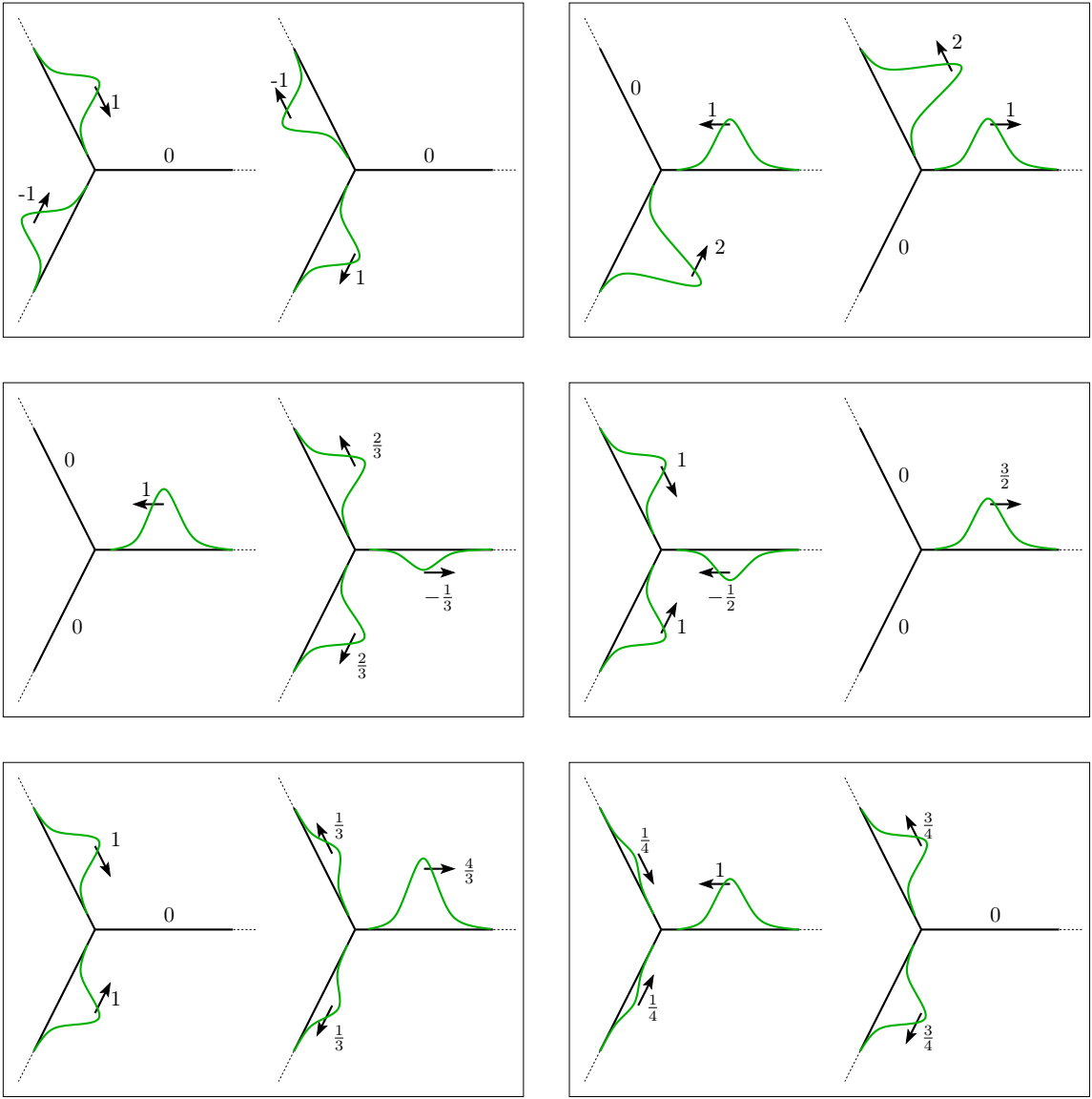}
\end{center}
\caption{\it Waves are arriving at a vertex and generate outgoing waves. In each box, the numbers indicate the amplitudes of ingoing waves (left) and the amplitudes of the outgoing waves (right). Second and third lines: the right cases are the same as the left ones, using the reversibility of the time in the wave equation. All the situations can be generated by combining these cases, using rotations or other symmetries.}\label{fig_waves_io}
\end{figure}

\medskip 

Choosing $\varphi$ as a very localised bump, we can construct exact solutions that stay for some times in a unobserved area, outside the control set $\omega$. As an elementary example, consider an unobserved part of the graph being a star-shaped graph with three edges $e_1$, $e_2$ and $e_3$ of lengths $\ell_1>\ell_2>\ell_3$. Figure \hyperref[fig_waves_io_2]{7} represents unobserved solutions of the wave equation in this part of the graph, using the above simple construction. We can see that the time needed to observe the solutions of the wave equation should be greater than $2\ell_2$ and than $(\ell_1+\ell_3)$. Notice that none of these times correspond to the length of the shortest path, neither to the length of the longest one.

\begin{figure}[htp]
\begin{center}
\resizebox{\textwidth}{!}{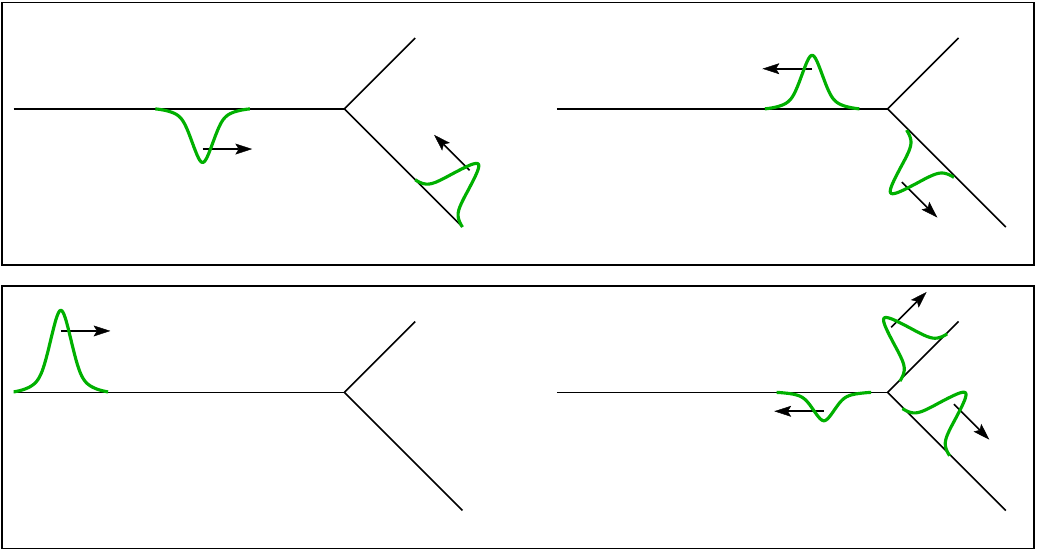}
\end{center}
\caption{\it Particular solutions of the wave equation in a part of a graph where three edges meet. The first solution use the odd symmetry with respect to the vertex. By considering bumps as localized as needed, the solution remains in this part of the graph during a time as close to $2\ell_2$ as wanted. The second solution is created by sending a wave from the longest edge  and remains in this part of the graph during a time as close to $(\ell_1+\ell_3)$ as wanted.}\label{fig_waves_io_2}
\end{figure}


\section{Example of optimal time for the observation of the free waves}\label{sec_optimal_time}

We consider here the particular case of Figure \hyperref[fig_3star]{5}. The start-graph $\Gc$ has a $\bot$-shape, with an horizontal basis being the segment $(-1,1)$ and a vertical edge of length $1$ being attached at $x=0$. A solution of the wave equation on $\Gc$ is decomposed in a function $x\mapsto u(x)$ defined on the horizontal basis $(-1,1)$ and a function $y\mapsto v(y)$ in the vertical edge. The free wave equation is then given by
\begin{equation}\label{eq_Tshape}
\left\{\begin{array}{ll} \partial_{t}^2 u(x,t)=\partial_{x}^2 u(x,t)~~~&x\in (-1,0)\cup(0,1),\, t>0,\\
\partial_{t}^2 v(y,t)=\partial_{x}^2 v(y,t)~~~&y\in (0,1),\, t>,0\\
u(-1,t)=u(1,t)=v(1,t)=0 & t>0,\\
u(0,t)=v(0,t) & t>0,\\
\partial_x u(0^+,t) - \partial_x u(0^-,t) + \partial_y v(0^+,t)=0 & t>0.
\end{array}\right. 
\end{equation}
The observation of the solution is made in the set 
$$\omega:= \{x\in (-1,-\ell_1)\cup (\ell_2,1)\}\,\cup\,\{y\in (\ell_3,1)\}$$
so that the unobserved part consists in three edges of lengths $\ell_1$, $\ell_2$ and $\ell_3$ meeting at one vertex. To fix the notation, we also assume that $\ell_1>\ell_2>\ell_3$. 
We underline that the specific shape of $\Gc$, three edges of the same length, will not be important in the analysis. The important shape is the unobserved part of the star-graph, so our discussion can be adapted to any graph having for unobserved part a forest of trees with three branches.

\medskip

In this setting, we have a first result showing that the time 
$$T^*:=\max(2\ell_2,\ell_1+\ell_3)$$
is necessary for the observation.
\begin{prop}
Let $T$ be a time strictly less than $T^*$. Then, there exists a non-zero solution $\varphi:=(u,v)\in\Cc^1([0,T],L^2(\Gc))\cap \Cc^0([0,T],H^1_0(\Gc))$ of \eqref{eq_Tshape} with $$\int_0^T \int_{\omega} |\partial_t \varphi(x,t)|^2 \dd x \dd t=0.$$
\end{prop}
\begin{proof}
It is sufficient to use the laws of transmission of Appendix \ref{section_waves_io} and
the solutions of Figure \hyperref[fig_waves_io_2]{7}. 

\medskip

If $T$ is strictly less than $2\ell_2$, we can consider a function $v\equiv 0$ and $u(x,t)=h(x+t)-h(-x+t)$ with $h$ a small bump localised in $(\ell_2-\varepsilon,\ell_2)$ with small $\varepsilon>0$. This solution $\varphi$ presents two symmetric bumps that travel horizontally. They will not enter the part $(-1,-\ell_1)$ nor $(\ell_2,1)$ before the time $2\ell_2-\varepsilon$.

\medskip

Now, assume that $T$ is strictly less than $\ell_1+\ell_3$. For small times $t$, we can consider a function $v\equiv 0$ and $u(x,t)=h(x-t)$ with $h$ a small bump localised in $(-\ell_1,-\ell_1+\varepsilon)$ with small $\varepsilon>0$. This solution $\varphi$ presents a bump travelling to the right until it reaches the vertex, at a time as close to $\ell_1$ as wanted. Then, it splits according to the laws of Appendix \ref{section_waves_io} and, due to the finite speed of propagation, it remains unobserved until the time $\ell_1+\ell_3-\varepsilon$.
\end{proof}

The second result shows the optimality of $T^*$.
\begin{prop}\label{prop_optimal}
Let $T$ be a time (strictly) larger than $T^*$. There exists a constant $C>0$ such that, for all solution $\varphi=:(u,v)\in\Cc^1([0,T],L^2(\Gc))\cap \Cc^0([0,T],H^1_0(\Gc))$ of \eqref{eq_Tshape}, we have
\begin{equation}\label{eq_prop_3star}
\|(\varphi_0,\varphi_1)\|_{H^1_0(\Gc)\times L^2(\Gc)}^2 \leq C \int_0^T \int_{\omega} |\partial_t \varphi(x,t)|^2 \dd x \dd t.
\end{equation}
\end{prop}
\begin{proof}
The main trick is to split $u$ into the sum $u=u_{\text{e}}+u_{\text{o}}$ of its even and odd part by setting
$$\forall x\in [-1,1]~,~~u_{\text{e}}(x):=\frac{u(x)+u(-x)}2 ~~~\text{ and }~~~   u_{\text{o}}(x):=\frac{u(x)-u(-x)}2.$$
We notice that $u_{\text{o}}$ does not interact with $v$ and that both $u_{\text{e}}$ and $u_{\text{o}}$ are determined by their values on $[-1,0]$ or $[0,1]$ only. We can split the problem in two equations $$\left\{\begin{array}{ll}  \partial_{t}^2  u_{\text{o}} = \partial_{x}^2 u_{\text{o}}, &\text{ in }(0,1),\\
u_{\text{o}}(0,t)=u_{\text{o}}(1,t)=0,\end{array}\right.$$
and 
$$\left\{\begin{array}{ll}   
\partial_{t}^2  w = \partial_{x}^2 w, ~~ & \text{ in }(-1,0)\cup (0,1),\\
w(-1,t)=w(1,t)=0, & t>0,\\
w(0^-,t)=w(0^+,t), & t>0,\\
2 \partial_x w(0^-,t)=\partial_x w(0^+,t), & t>0,
\end{array}\right.$$
with
$$w(x):=\left\{\begin{array}{ll} u_{\text{e}}(x),~~&\text{ if }x\in (-1,0),\\ v(x), & \text{ if }x\in (0,1).\end{array}\right.$$
Next, we reproduce the strategy of the proof of Proposition \ref{prop_obs_wave} on each equation. 
The only point where we have to be careful is the coefficient $2$ in the transmission condition at $0$ for $w$. We will explain why it is harmless. The other details of the proof will be skipped since they are the same as the ones of the proof of Proposition \ref{prop_obs_wave}.

\medskip

First notice that, obviously, the energies 
$\|(\varphi,\partial_t \varphi)(t)\|_{H^1_0\times L^2}^2$ and 
$\|(u_{\text{o}},\partial_t u_{\text{o}})(t)\|_{H^1_0\times L^2}^2$ are independent of the time because we are considering solutions of classical free waves equation. So this is also the case of the energy
\begin{equation}\label{eq_3star1}
2 \int_{-1}^0 ( |\grad w|^2 +|\partial_t w|^2) \dd x + \int_0^1 ( |\grad w|^2 +|\partial_t w|^2) \dd x.
\end{equation}
This can be checked directly or by using the relations between $\varphi=(u,v)$, $u_{\text{o}}$ and $w$.

\medskip

Arguing as in the proof of Proposition \ref{prop_obs_wave}, we obtain an estimate 
\begin{equation}\label{eq_3star2}
\|(w,\partial_t w)(t=0)\|_{H^1_0\times L^2}^2 ~\leq~  C \int_{0}^{T} \int_{(-1,-\ell_1)\cup (\ell_3,1)} |\partial_t w(x,t)|^2 \dd x \dd t
\end{equation}
for any $T>\ell_1+\ell_3$. Indeed, $w$ solves a wave equation and the observation in $(-1,-\ell_1)$ and in $(\ell_3,1)$ spreads from both side at speed one inside the segment, as shown in Step 2 of the proof of Proposition \ref{prop_obs_wave}. This step also shows that the value of $w$ at $0^+$ is also controlled. Arguing as in Step 4, we can pass through the vertex, since the factor $2$ in the transmission condition at $x=0$ is harmless in the argument. It remains to follow mutatis mutandis the last steps of the proof of Proposition \ref{prop_obs_wave}, where the conservation of energy has to be understood as noticed above, which is again harmless during the estimations, simply changing some constants $C$, which is not our point here.

\medskip

We argue in the same way to obtain an observation of $u_{\text{o}}$ in $(\ell_2,1)$. We obtain the estimate
$$\|(u_{\text{o}},\partial_t u_{\text{o}})(t=0)\|_{H^1_0\times L^2}^2  ~\leq~  C \int_{0}^{T} \int_{(\ell_2,1)} |\partial_t u_{\text{o}}(x,t)|^2 \dd x \dd t$$
for any $T>2\ell_2$. 

\medskip

Now, we go back to the original solution $\varphi=(u,v)$ by using the fact $v(y)=w(y)$ for $y\in (0,1)$, $u(x)=w(-x)+u_{\text{o}}(x)$ for all $x\in (0,1)$ and $u(x)=w(x)-u_{\text{o}}(-x)$ for all $x\in (-1,0)$. Thus, we get that, for any $T>T^*$, there exists $C>0$ such that  
\begin{align*}
\|(\varphi,\partial_t \varphi)(t=0)\|_{H^1_0\times L^2}^2  ~\leq~&  C \int_0^T \int_{(\ell_2,1)} |\partial_t u_{\text{o}}(x,t)|^2 \dd x \dd t\\ &+C  \int_0^T \int_{(-1,-\ell_1)\cup (\ell_3,1)} |\partial_t w(x,t)|^2 \dd x \dd t.
\end{align*} 
For $x\in (\ell_2,1)$, we have $u_{\text{o}}(x)=u(x)-u_{\text{e}}(x)=u(x)-w(-x)$. Thus, we have
\begin{align*}
\|(\varphi,\partial_t \varphi)(t=0)\|_{H^1_0\times L^2}^2  ~\leq~&  C \int_0^T \int_{(\ell_2,1)} |\partial_t u(x,t)|^2 \dd x \dd t + C \int_0^T \int_{(-1,-\ell_2)} |\partial_t w(x,t)|^2 \dd x \dd t\\ &+C  \int_0^T \int_{(-1,-\ell_1)\cup (\ell_3,1)} |\partial_t w(x,t)|^2 \dd x \dd t.
\end{align*} 
The observation in $(-1,-\ell_2)$ seems wrong because the interval is not contained in $\omega$. But we simply bound it by the conserved energy \eqref{eq_3star1}, which is itself bounded by the observation \eqref{eq_3star2}. Thus, we obtain the bound
\begin{align*}
\|(\varphi,\partial_t \varphi)(t=0)\|_{H^1_0\times L^2}^2  ~\leq~&  C \int_0^T \int_{(\ell_2,1)} |\partial_t u(x,t)|^2 \dd x \dd t\\ &+C  \int_0^T \int_{(-1,-\ell_1)\cup (\ell_3,1)} |\partial_t w(x,t)|^2 \dd x \dd t.
\end{align*} 
In $(\ell_3,1)$, $\partial_t w=\partial_t v$. In $(-1,-\ell_1)$, $\partial_t w(x)=\frac 12 (\partial_t u(x)+\partial_t u(-x))$ and the term $\partial_t u(-x)$ can be observed because $-x$ belongs to $(\ell_1,1)$, which is included in $(\ell_2,1)$. Thus, we obtain the estimate
\begin{align*}
\|(\varphi,\partial_t \varphi)(t=0)\|_{H^1_0\times L^2}^2  ~\leq~&  C \int_0^T \int_{(-1,-\ell_1)\cup (\ell_2,1)} |\partial_t u(x,t)|^2 \dd x \dd t\\ &+C  \int_0^T \int_{(\ell_3,1)} |\partial_t v(x,t)|^2 \dd x \dd t,
\end{align*}
which is exactly \eqref{eq_prop_3star}.
\end{proof}

\section{Boundary control of the Schrödinger equation on the X graph}
\label{sec:boundaryschroproofNoGGCC}

In this part, we consider the boundary control system
\begin{equation}
\label{eq_schro_control_bound}
\left\{
\begin{array}{ll} i \partial_{t} u = - \Delta_\Gc u ,& \ \ \ \ t>0,\\
u(t,v_1)=h_1(t),\ u(t,v_3) = h_3(t),\ u(t,v_2)=u(t,v_4) = 0,& \ \ \ \ t>0,\\
u(t=0)=u_0\in H^{-1}(\Gc), & 
\end{array}\right. 
\end{equation}
and we prove Theorem \ref{th:boundarycontrolSchro}. First, let us recall the well-known equivalence between the exact controllability of \eqref{eq_schro_control_bound} in the state space $H^{-1}(\Gc)$ and the observability of the adjoint system in $H_0^1(\Gc)$, see for instance \cite[Chapters 10 and 11]{TW09}, in particular \cite[Theorem 11.2.1]{TW09}.
\begin{theorem}
\label{th:boundcontrolschroduality}
Let $T>0$. The controlled Schrödinger equation \eqref{eq_schro_control_bound} is exactly controllable in $H^{-1}(\Gc)$ with boundary controls in $L^2(0,T)$ if and only if there exists $C>0$ such that for every $\varphi_0 \in H_0^1(\Gc)$, the solution $\varphi \in \Cc([0,T];H_0^1(\Gc))$ of 
\begin{equation}
\label{eq_schro_control_bound_adjoint}
\left\{
\begin{array}{ll} i \partial_{t} \varphi = - \Delta_\Gc \varphi ,& \ \ \ \ t>0,\\
\varphi(t,v_1)= \varphi(t,v_2)= \varphi(t,v_3)=\varphi(t,v_4) = 0,& \ \ \ \ t>0,\\
\varphi(t=0)=\varphi_0\in H_0^{1}(\Gc), & 
\end{array}\right. 
\end{equation}
satisfies
\begin{equation}
    \|\varphi_0\|_{H_0^{1}(\Gc)}^2 \leq C \int_0^T \left(|\partial_x u(t,v_1)|^2 +  |\partial_x u(t,v_3)|^2\right)\dd t.
\end{equation}
\end{theorem}
We then deduce the following result from \cite{Mil05}.
\begin{theorem}[{\cite[Theorem 5.1]{Mil05} or \cite[Proposition 6.6.1]{TW09}}]
\label{lem:resolventboundary}
The equation \eqref{eq_schro_control_bound} is exactly controllable in $H^{-1}(\Gc)$ with boundary controls in $L^2(0,T)$ at some $T>0$ if and only if there exists $C>0$ such that
\begin{equation}
\label{eq:resolventschroboundary}
    \|u\|_{H_0^1(\Gc)}^2 \leq C (\|\Delta u + \lambda u\|_{H_0^1(\Gc)}^2 + |\partial_x u(v_1)|^2 +  |\partial_x u(v_3)|^2) \quad \forall \lambda \in \mathbb{R},\ \forall u \in D(|\Delta_{\Gc}|^{3/2}).
\end{equation}
\end{theorem}
This resolvent estimate of Theorem \ref{lem:resolventboundary} follows from Theorem \ref{th:boundcontrolschroduality} and \cite[Theorem 5.1]{Mil05}. Indeed, using the notations of \cite{Mil05}, we take the operator $\mathcal{A}=\Delta_{\Gc}$ acting on $X=H_0^1(\Gc)$, whose domain is $D(\mathcal{A})=D(|\Delta_{\Gc}|^{3/2})$. We have the following description of this space
\begin{align*}
 D(\mathcal{A})=D(|\Delta_{\Gc}|^{3/2}) = &\Big\{u=(u_j)_{j\in J} \in \prod_{j\in J} H^3(e_j)~:\\
 &\ u_j(v_k)=u_{j'}(v_k),\ \partial^2_{x}u_j(v_k)=\partial^2_{x}u_{j'}(v_k),\  k\in K_{\rm int},\ j,j'\in \Ec_k;\\
&\ \sum_{j\in\Ec_k} \frac{\dd u_j}{\dd n_j}(v_k)=0,\ k\in K_{\rm int};\\
&\ u_j(v_k)=\partial^2_{x}u_j(v_k)=0,\ k \in K_{\rm ext}\Big\}.
\end{align*}
The admissible observation operator $\mathcal{C} \in \mathcal{L}(D(\mathcal{A});\mathbb C^2)$ is defined by $\mathcal{C}u = (\partial_{x}u (v_1), \partial_{x} u(v_3))$.
We are now ready to prove Theorem \ref{th:boundarycontrolSchro} by adapting the proof of Theorem \ref{non_ggcc_sch_intro}. We only prove the  “if part” and we keep the same notations of the proof of Theorem \ref{non_ggcc_sch_intro}.

\begin{proof}[Proof of Theorem \ref{th:boundarycontrolSchro}]
Due to Theorem \ref{lem:resolventboundary}, we only have to show that \eqref{eq:resolventschroboundary} holds true if $\ell$ is a badly approximable number. First notice that \eqref{eq:resolventschroboundary} trivially holds for $\lambda \leq 0$ since the operator $\Delta_{\Gc}$ is non-negative. Moreover, we have the following unique continuation property for the eigenfunctions of $\mathcal A$ by Cauchy-Lipschitz theorem. For $\lambda \in \mathbb R,\ u \in D(\mathcal{A})$,
$$ (\mathcal A u + \lambda u = 0\ \text{and}\ \mathcal C u = 0) \Rightarrow u \equiv 0.$$
So by using \cite[Proposition 6.6.4]{TW09}, we can only restrict to prove \eqref{eq:resolventschroboundary} for $\lambda \geq M$, where $M>0$ is a constant that would be specified later.

\medskip

We argue by contradiction.  Assume that $\ell$ is a badly approximable number and that there exist $(u_n)_{n \in \mathbb{N}} \in D(|\Delta_{\Gc}|^{3/2})^{\mathbb{N}}$ and $(\lambda_n)_{n \in  \mathbb{N}} \in ([M, +\infty))^{\mathbb{N}}$ such that
\begin{equation}\label{absboundary}
    \|u_n\|_{H_0^1(\Gc)} = 1,\ \|(\Delta_{\Gc} + \lambda_n) u_n\|_{H_0^1(\Gc)} \to 0~\text{ and }~ |\partial_x u_n(v_1)| +|\partial_x u_n(v_3)| \to 0,
\end{equation}
as $n \to +\infty$.

\medskip

{\noindent \bf Step 1: Symmetry decomposition.} For any $u\in D(|\Delta_{\Gc}|^{3/2})$, we set $f$, $g$ and $h$ as in \eqref{eq:deffgschro} and \eqref{eq:defhschro}. One can check that
\begin{equation*}
    u \in D(|\Delta_{\Gc}|^{3/2}) \Leftrightarrow \left\{\begin{array}{l} f \in \{\psi \in H^3(-1,0)\ ;\ \psi(-1)=\psi(0)=\partial_{xx}\psi(-1)=\partial_{xx}\psi(0)=0\},\\ g\in \{\psi \in H^3(0,\ell)\ ;\ \psi(0)=\psi(\ell)=\partial_{xx}\psi(0)=\partial_{xx}\psi(\ell)=0\},\\ h \in \{\psi \in H^3(-1,\ell)\ ;\ \psi(-1)=\psi(\ell)=\partial_{xx}\psi(-1)=\partial_{xx}\psi(\ell)=0\},\end{array}\right.
\end{equation*}
 and a direct computation gives
\begin{align*}
    \|\partial_{x} u&\|_{L^2}^2  = \sum_{j=1}^{4} \|\partial_{x} u^i\|_{L^2(e_j)}^2  = 2\big(\|\partial_{x} f\|_{L^2(-1,0)}^2 + \|\partial_{x} g\|_{L^2(0,\ell)}^2 + \|\partial_{x} h\|_{L^2(-1,\ell)}^2\big),\\
    \|(\Delta_{\Gc} + \lambda )\partial_{x}u&\|_{L^2}^2  = \sum_{j=1}^{4} \|(\Delta + \lambda ) \partial_{x}u^i\|_{L^2(e_j)}^2  \\
    & ~ = 2\left(\|(\Delta + \lambda )\partial_{x}f\|_{L^2(-1,0)}^2+ \|(\Delta + \lambda )\partial_{x}g\|_{L^2(0,\ell)}^2 + \|(\Delta + \lambda )\partial_{x}h\|_{L^2(-1,\ell)}^2\right).
\end{align*}
We define now $(f_n)_{n \in \mathbb{N}}$, $(g_n)_{n \in \mathbb{N}}$ and $(h_n)_{n \in \mathbb{N}}$ from the sequence $(u_n)_{n \in \mathbb{N}}$ as above. We get
\begin{equation}\label{eq_normB}  \|\partial_{x}f_n\|_{L^2(-1,0)}^2 +  \|\partial_{x}g_n\|_{L^2(0,\ell)}^2 +  \|\partial_{x}h_n\|_{L^2(-1,\ell)}^2 = 1/2,\end{equation}
\begin{equation}\label{eq_convB}   \Big\|\big(\Delta + \lambda_n \big)\partial_{x}f_n\Big\|_{L^2(-1,0)}+\Big\|\big(\Delta + \lambda_n \big)\partial_{x}g_n\Big\|_{L^2(0,\ell)}+\Big\|\big(\Delta + \lambda_n \big)\partial_{x}h_n\Big\|_{L^2(-1,\ell)} \to 0,\end{equation}
\begin{equation}\label{eq_limB}    |\partial_{x}(f_n + h_n)(-1)|^2  +  |\partial_{x}(g_n + h_n)(\ell)|^2 \to 0,\end{equation}
$\text{as}\ n \to +\infty$. Here we have used that the $H_0^1$-norms of $\psi$ are equivalent to the $L^2$-norm of $\partial_{x} \psi$. We then use \eqref{eq_convB} to provide estimations for $f_n$ and $g_n$ 
at the central node. For instance, the first convergence infers
$$\partial^2_{x}(\partial_x f_n)=-\lambda_n \partial_{x}f_n+ r_n,\ \ \ \ \text{in}\ (-1,0), \ \ \text{ where}\ \ \ \ \|r_n\|_{L^2(-1,0)}\rightarrow 0\ \text{as}\ n \to +\infty.$$
We impose the boundary condition $\partial^2_{x} f_n(0)=0$ and, solving the ODE by the variation of the constant formula, we get that
\begin{equation}
    \label{eq:explicitf_ndx}
    \partial_{x}f_n(x)= {\alpha_n}  \cos \left(\sqrt{\lambda_n} x \right) + \frac{1}{\sqrt{\lambda_n}}\int_0^x\sin \left(\sqrt{\lambda_n}(x-y) \right)r_n(y)\dd y. 
\end{equation}
The same type of identity holds for $\partial_x g_n$, i.e. 
\begin{equation}
    \label{eq:explicitg_ndx}
    \partial_{x}g_n(x)= {\beta_n}  \cos \left(\sqrt{\lambda_n} x \right) + \frac{1}{\sqrt{\lambda_n}}\int_0^x\sin \left(\sqrt{\lambda_n}(x-y) \right)s_n(y)\dd y,
\end{equation}
where $\|s_n\|_{L^2(-1,0)}\rightarrow 0\ \text{as}\ n \to +\infty$. We repeat this explicit resolution to the other identity of \eqref{eq_convB} and we deduce that, when $n\rightarrow +\infty$,
\begin{equation}\begin{split}\label{eq_expB}
\partial_{x}f_n(x) &= \alpha_n \cos \left(\sqrt{\lambda_n} x \right) + o_{\Cc([-1,0])} \Big(\frac{1}{\sqrt{\lambda_n}}\Big),\qquad x\in[-1,0],\\
\partial_{x} g_n(x) &= \beta_n \cos \left(\sqrt{\lambda_n} x \right) +  o_{\Cc([0,\ell])}\Big(\frac{1}{\sqrt{\lambda_n}}\Big),\qquad x\in[0,\ell],\\
\partial_{x} h_n(x) &= \gamma_n \cos \left(\sqrt{\lambda_n} x \right) +  o_{\Cc([-1,\ell])}\Big(\frac{1}{\sqrt{\lambda_n}}\Big),\qquad x\in[-1,\ell].
\end{split}\end{equation}

\medskip

{\noindent \bf Step 2: Reduction to the resolvent estimate on an interval.}
Now, we claim that there exists $c >0$ such that 
\begin{equation}\label{low_boundB}
|\alpha_n | \geq c~~\text{ and }~~  |\beta_n | \geq c.
\end{equation}
Indeed, if \eqref{low_boundB} is not true, then either $(\alpha_n)$ or $(\beta_n)$  
goes to $0$, up to the extraction of a subsequence. Consider, for instance, $\alpha_n \to 0$ (the other case is similar). Then $\|\partial_{x} f_n\|_{\Cc([-1,0])} \to 0$ and, from the first limit of \eqref{eq_limB} and from  \eqref{eq_convB}, we would have
$$ \Big\|\big(\Delta + \lambda_n \big)\partial_{x} h_n\Big\|_{L^2(-1,\ell)}~~\text{ and }~~\ |\partial_{x} h_n(-1)| \to 0.$$
Recall that Schrödinger equation is controllable in $H^{-1}(\Gc)$ in the interval $(-1,\ell)$ with one Dirichlet boundary control at the point $-1$ in $L^2(0,T)$, see for instance \cite{Mac94}. This provides a resolvent estimate for the Schrödinger equation in the interval $(-1,\ell)$ with the observation in $-1$ and we necessarily have $\|\partial_x h_n\|_{L^2( -1,\ell)} \to 0$. This $L^2$-convergence can be transformed into a $\Cc^0$-convergence. Indeed, from the third equation of \eqref{eq_expB}, we have that $\|\gamma_n \cos (\sqrt{\lambda_n} x)\|_{L^2(-1,\ell)} \to 0$ so 
$$ \gamma_n^2 \left(\frac{\ell+1}{2} + \frac{\sin(\sqrt{\lambda_n} \ell) - \sin(-\sqrt{\lambda_n})}{2 \sqrt{\lambda_n}}\right) \to 0.$$
Then recalling $\lambda_n \geq M$ and choosing $M \geq \frac{16}{(\ell+1)^2}$, we deduce that $\gamma_n \to 0$ hence $$\|\partial_{x} h_n\|_{\Cc([-1,\ell])} \to 0.$$
By using the second limit of \eqref{eq_limB}, we then have $|\partial_x g_n(\ell)| \to 0$. Then by conjugating with the second limit of \eqref{eq_convB} we have
$$ \Big\|\big(\Delta + \lambda_n \big)\partial_{x} g_n\Big\|_{L^2(0,\ell)}~~\text{ and }~~\ |\partial_{x} g_n(\ell)| \to 0.$$
So from the resolvent estimate of the Schrödinger equation in the interval $(0,\ell)$ with the observation in $\ell$ and we necessarily have $\|\partial_{x} g_n\|_{L^2(0,\ell)} \to 0$.
The combination of $\|\partial_x f_n\|_{L^2(-1,0)} \to 0$, $\|\partial_x h_n\|_{L^2( -1,\ell)} \to 0$ and $\|\partial_x g_n\|_{L^2(0,\ell)} \to 0$ is then a contradiction with respect to \eqref{eq_normB}. This shows \eqref{low_boundB}.

\medskip

{\noindent \bf Step 3: Comparison of the two quasimodes.} We now conclude as in the Step 3 of the proof of Theorem \ref{non_ggcc_sch_intro} by using \eqref{eq:explicitf_ndx} and \eqref{eq:explicitg_ndx} together with Dirichlet boundary conditions on the second order derivatives. For instance, we have
$$\partial^2_{x}f_n(x)= - {\alpha_n} \sqrt{\lambda_n} \sin \left(\sqrt{\lambda_n} x \right) + \int_0^x\cos \left(\sqrt{\lambda_n}(x-y) \right)r_n(y)\dd y,$$
So by using $\partial_{xx} f_n(-1)=0$, we get
$\sin \left(\sqrt{\lambda_n} \right) = o\left(\frac{1}{\sqrt{\lambda_n}}\right)$. We finally  obtain a contradiction with the fact that $\ell$ is a badly approximable number.
\end{proof}

\end{appendices}



\begin{thebibliography}{99}

\bibitem{Ale83}
S. Alexander.
\newblock Superconductivity of networks. {A} percolation approach to the effects of disorder.
\newblock {\it Phys. Rev. B}, 27(3):1541--1557, 1983.

\bibitem{AK11}
F. Ammar-Khodja, A. Benabdallah , M. Gonz\'alez-Burgos  and L. de Teresa.
\newblock Recent results on the controllability of linear coupled parabolic problems: A survey.
\newblock {\it Math. Control Relat. Fields.}, 1(3):267--306, 2011.

\bibitem{AHT}
K. Ammari, A. Henrot and M. Tucsnak.
\newblock Asymptotic behaviour of the solutions and optimal location of the actuator for the pointwise stabilization of a string. 
\newblock {\it Asymptotic Anal.}, 28(3-4):215--240, 2001. 

\bibitem{AS} K. Ammari and F. Shel. 
\newblock {\it Stability of elastic multi-link structures}.
\newblock SpringerBriefs Math. Springer, Cham, 2022.


\bibitem{AT} K. Ammari and M. Tucsnak.
\newblock Stabilization of second order evolution equations by a class of unbounded feedbacks.
\newblock {\it ESAIM Control Optim. Calc. Var.} 6:361--386, 2001.

\bibitem{AL14} N. Anantharaman and M. L\'eautaud, 
\newblock Sharp polynomial decay rates for the damped wave equation on the torus. With an appendix by St\'ephane Nonnenmacher. 
\newblock {\it Anal. PDE.,} 7(1):159--214, 2014.

\bibitem{ALM16}
N. Anantharaman, M. L\'eautaud and F. Maci\`a.
\newblock Wigner measures and observability for the Schr\"odinger equation on the disk.
\newblock {\it Invent. Math.,} 206(2):485--599, 2016.

\bibitem{AM14}
N. Anantharaman and F. Maci\`a.
\newblock Semiclassical measures for the Schr\"odinger equation on the torus.
\newblock {\it J. Eur. Math. Soc. (JEMS),} 16(6):1253--1288, 2014.


\bibitem{AB}
J. Apraiz and  J.~A. B\'arcena-Petisco.
\newblock Observability and control of parabolic equations on networks with loops.
\newblock {\it J. Evol. Equ.,} 23(2):33, 2023.


\bibitem{AEL23}
S.~A. Avdonin, J. Edward and G.~R. Leugering.
\newblock Controllability for the wave equation on graph with cycle and delta-prime vertex conditions.
\newblock {\it Evol. Equ. Control Theory,}  12(6):1542--1558, 2023. 

\bibitem{AEZ23}
S. Avdonin, J. Edward and Y. Zhao.
\newblock Shape, velocity, and exact controllability for the wave equation on a graph with cycle.
\newblock {\it Algebra i Analiz}, 35(1):3--32, 2023.

\bibitem{AI95}
S.~A. Avdonin and S.~A. Ivanov. 
\newblock {\it Families of exponentials}.
\newblock Cambridge Univ. Press, Cambridge, 1995.

\bibitem{AZ21}
S.~A. Avdonin and Y. Zhao.
\newblock Exact controllability of the 1-D wave equation on finite metric tree graphs.
\newblock {\it Appl. Math. Optim.,} 83(3):2303--2326, 2021.

\bibitem{AZ22}
S.~A. Avdonin and Y. Zhao.
\newblock Exact controllability of the wave equation on graphs.
\newblock {\it Appl. Math. Optim.,} 85(2):44, 2022.

\bibitem{BS} 
J.~M. Ball and M. Slemro.
\newblock Nonharmonic Fourier series and the stabilization of distributed semilinear control systems.
\newblock {\it Comm. Pure Appl. Math.,} 32:555--587, 1979.

\bibitem {BLR_JEDP}
C. Bardos, G. Lebeau and J. Rauch.
\newblock Contr\^ole et stabilisation pour l'\'equation des ondes.
\newblock {\it Journ. Équ. Dériv. Partielles, St.-Jean-De-Monts 1987}, 13:1--15. 

\bibitem {BLR}
C. Bardos, G. Lebeau and J. Rauch.
\newblock Sharp sufficient conditions for the observation, control and stabilization of waves from the boundary.
\newblock {\it SIAM Journal on Control and Optimization,} 30(5):1024--1065, 1992.

\bibitem{BK04}
G. Berkolaiko, J.~P. Keating and B. Winn.
\newblock No quantum ergodicity for star graphs.
\newblock {\it Comm. Math. Phys.,} 250(2): 259--285, 2004.

\bibitem{BK13}
G. Berkolaiko and P.~A. Kuchment. {\it Introduction to quantum graphs}. Mathematical Surveys and Monographs, 186, Amer. Math. Soc., Providence, RI, 2013.

\bibitem{Bugeaud}
Y. Bugeaud. 
\newblock {\it Approximation by algebraic numbers}. 
\newblock Cambridge Tracts in Mathematics \no 160. Cambridge University Press, 2004.

\bibitem{BZ12}
N. Burq and M. Zworski.
\newblock Control for Schr\"odinger operators on tori.
\newblock {\it Math. Res. Lett.}, 19(2), 309--324, 2012.
\bibitem{BZ}
N. Burq and M. Zworski. 
\newblock Geometric control in the presence of a black box.
\newblock {\it J. Am. Math. Soc.,} 17(2):443--471, 2004. 


\bibitem{Cassels} J.~W.~S Cassels. {\it An introduction to diophantine approximation}. University Press, 1965.

\bibitem{Col}
Y. Colin~de~Verdi\`ere. 
\newblock Semi-classical measures on quantum graphs and the Gau\ss\ map of the determinant manifold.
\newblock {\it Ann. Henri Poincar\'e.,} 16(2) 347--364, 2015.
  
\bibitem{Cor07}
J.~M. Coron. 
\newblock {\it Control and nonlinearity}.
\newblock Mathematical Surveys and Monographs. 136, Amer. Math. Soc., Providence, RI, 2007.

\bibitem{DZ06}
R. D\'ager and E. Zuazua. 
\newblock {\it Wave propagation, observation and control in $1\text{-}d$ flexible multi-structures}.
\newblock Math\'ematiques \& Applications (Berlin), 50, Springer, Berlin, 2006.

\bibitem{Davis}
C.~S. Davis. 
\newblock Rational approximations to $e$.
\newblock {\it J. Aust. Math. Soc., Ser. A}, 25:497--502, 1978. 


\bibitem{FR71}
H.~O. Fattorini and D.~L. Russell.
\newblock Exact controllability theorems for linear parabolic equations in one space dimension.
\newblock {\it Arch. Rational Mech. Anal.,} 43:272--292, 1971.

\bibitem{FJK87}
C. Flesia, R. Johnston and H. Kunz.
\newblock Strong localization of classical waves: A numerical study.
\newblock {\it Europhysics Letters ({EPL})}, 3(4):497--502, 1987.

\bibitem{Haraux} A. Haraux.
\newblock Une remarque sur la stabilisation de certains syst\`emes du deuxi\`me ordre en temps.
\newblock {\it Portugal Math.,} 46(3):245--258, 1989.

\bibitem{HW}
G.~H. Hardy and E.~M. Wright.
\newblock {\it An introduction to the theory of numbers. 4th ed.}
\newblock Oxford: At the Clarendon Press. xvi, 1960.

\bibitem{Jaffard}
S. Jaffard.
\newblock Contr\^ole interne exact des vibrations d'une plaque rectangulaire.
\newblock {\it Portugal. Math.}, 47(4):423--429, 1990.

\bibitem{JTZ}
S. Jaffard, M. Tucsnak and E. Zuazua.
\newblock Singular internal stabilization of the wave equation. 
\newblock {J. Differ. Equations} 145(1):184--215, 1998. 


\bibitem{LLS94}
J.~E. Lagnese, G.~R. Leugering and E.~J.~P.~G. Schmidt. 
\newblock {\it Modeling, analysis and control of dynamic elastic multi-link structures}. 
\newblock Systems \& Control: Foundations \& Applications, Birkh\"auser Boston, Boston, MA, 1994.

\bibitem{Lang}
S. Lang. 
\newblock {\it Introduction to Diophantine Approximations.} 
\newblock New Expanded Edition. Springer-Verlag New York, 1995.


\bibitem{Laurent}
C. Laurent.
\newblock Internal control of the Schr\"odinger equation.
\newblock {\it Math. Control Relat. Fields.,} 4(2):161--186, 2014.

\bibitem{Lebeau} G. Lebeau. 
\newblock {\it \'Equation des ondes amorties.}  
\newblock Algebraic and geometric methods in mathematical physics, Mathematical Physics Studies, 19, Kluwer Academic Publishers Group, Dordrecht, 1996


\bibitem{luc} J. Le Rousseau, G. Lebeau, L. Robbiano. 
\newblock {\it Elliptic Carleman estimates and applications to stabilization and controllability. Volume I. Dirichlet boundary conditions on Euclidean space}.  
\newblock Progress in Nonlinear Differential Equations and Their Applications 97. Subseries in Control. Cham: Birkhäuser, 2022




\bibitem{Lions}
J.~L. Lions. 
\newblock {\it Contr\^olabilit\'e exacte, perturbations et stabilisation de syst\`emes distribu\'es. Tome 1: Cont\^olabilit\'e exacte.} 
\newblock Recherches en Math\'ematiques Appliqu\'ees, 8. Paris etc.: Masson. 1988. 

\bibitem{Liouville}
J. Liouville.
\newblock Communication \`a propos des classes tr\`es \'etendues de quantit\'es dont la valeur n'est ni rationnelle ni m\^eme r\'eductible \`a des irrationnels alg\'ebriques. 
\newblock {\it Comptes-rendus de l'Académie des sciences},  18:883--885, 910--911, 1844.

\bibitem{Mac94}
E. Machtyngier.
\newblock  Exact controllability for the Schr\"odinger equation.
\newblock {\it SIAM J. Control Optim.,} no.~1, 24--34, 1904


\bibitem{Mil05}
L. Miller.
\newblock Controllability cost of conservative systems: resolvent condition and transmutation.
\newblock {\it J. Funct. Anal.,} 218(2):425--444, 2005.

\bibitem{Mil06}
L. Miller.
\newblock The control transmutation method and the cost of fast controls.
\newblock {\it SIAM J. Control Optim.,} 45(2):762--772, 2006.

\bibitem{ML71}
R. Mittra and S.~W. Lee.
\newblock {\it Analytical Techniques in the Theory of Guided Waves}.
\newblock New York: Macmillan, 1971.

\bibitem{Morawetz}
S.~C. Morawetz.
\newblock The decay of solutions of the exterior initial-boundary value problem for the wave equation.
\newblock {\it Commun. Pure Appl. Math.,} 14:561--568, 1961. 


\bibitem{Pau36}
L. Pauling.
\newblock The diamagnetic anisotropy of aromatic molecules.
\newblock {\it The Journal of Chemical Physics}, 4(10):673--677, 1936.


\bibitem{Nicaise}
S. Nicaise. 
\newblock Control and stabilization of $2\times 2$ hyperbolic systems on graphs.
\newblock {\it Math. Control Relat. Fields.,} 7(1):53--72, 2017. 

\bibitem{pazy} A. Pazy.
\newblock {\it Semigroups of linear operators and applications to partial differential equations}.
\newblock Appl. Math. Sci., 44, Springer-Verlag, New York, 1983.

\bibitem{RTTM05}
K. Ramdani, T. Takahashi, G. Tenenbaum, M. Tucsnak.
\newblock A spectral approach for the exact observability of infinite-dimensional systems with skew-adjoint generator.
\newblock {\it J. Funct. Anal.}, no.~1, 193--229, 2005.

\bibitem{RT}
J. Rauch and M. Taylor.
\newblock Exponential decay of solutions to hyperbolic equations in bounded domains.
\newblock {\it Indiana University Mathematical Journal,} 24:79--86, 1974. 

\bibitem{Roth}
F.~K. Roth.
\newblock Rational approximations to algebraic numbers.
\newblock {\it Mathematika}, 2:1--20, 1955. 

\bibitem{Schmidt}
E.~J.~P.~G. Schmidt.
\newblock On the modelling and exact controllability of networks of vibrating strings.
\newblock {\it SIAM J. Control Optimization.,} 30(1):229--245, 1992. 

\bibitem{Shallit}
J. Shallit.
\newblock Simple continued fractions for some irrational numbers.
\newblock {\it J. Number Theory}, 11:209--217, 1979. 

\bibitem{TW09}
M. Tucsnak and G. Weiss, 
\newblock{\it Observation and control for operator semigroups.}
\newblock Birkh\"auser Advanced Texts: Basler Lehrb\"ucher, Birkh\"auser Verlag, Basel, 2009.
 

\end{thebibliography}
\end{document}